\documentclass[12pt]{extarticle}

\usepackage{amsmath, amsthm, amssymb, color}
\usepackage[colorlinks=true,linkcolor=blue,urlcolor=blue]{hyperref}
\usepackage{graphicx}
\usepackage{caption}
\usepackage{mathtools}
\usepackage{hyperref}
\usepackage{enumerate}
 \usepackage{comment}
\usepackage{wrapfig}

\usepackage{cleveref}
\usepackage{verbatim}

\usepackage{tikz,tikz-cd,tikz-3dplot}
\usetikzlibrary{decorations.markings}
\usetikzlibrary{shapes.geometric}
\usetikzlibrary{matrix}
\usetikzlibrary{arrows}

\usepackage{amssymb}
\usetikzlibrary{matrix}
\usetikzlibrary{arrows}
\usepackage{algorithm}
\usepackage[noend]{algpseudocode}
\usepackage{caption}
\usepackage[normalem]{ulem}
\usepackage{subcaption}
\usepackage{multicol}
\oddsidemargin \evensidemargin
\usepackage{makecell}
\usepackage{array}
\usepackage{enumitem}

 \newtheorem{theorem}{Theorem}[section]

\newtheorem{proposition}[theorem]{Proposition}
\newtheorem{lemma}[theorem]{Lemma}
\newtheorem{corollary}[theorem]{Corollary}

\newtheorem{conjecture}[theorem]{Conjecture}

\theoremstyle{definition}

\newtheorem{remark}[theorem]{Remark}

\newtheorem{example}[theorem]{Example}

\newcommand{\PP}{\mathbb{P}}
\newcommand{\RR}{\mathbb{R}}
\newcommand{\QQ}{\mathbb{Q}}

\newcommand{\ZZ}{\mathbb{Z}}
\newcommand{\NN}{\mathbb{N}}
\newcommand{\Gr}{{\rm Gr}}
\newcommand{\Mat}{{\rm Mat}}

\title{\bf Landau Analysis in the Grassmannian}

\author{Benjamin Hollering, Elia Mazzucchelli, \\
Matteo Parisi, and Bernd Sturmfels}

\date{}
\begin{document}

\maketitle

\begin{abstract} \noindent
Momentum twistors for scattering amplitudes in particle physics are lines in three-space.
We develop Landau analysis for Feynman integrals in this setting. The resulting discriminants and resultants are identified with Hurwitz and Chow forms of incidence varieties in products of Grassmannians. We study their degrees and factorizations, and the kinematic regimes in which the fibers of the Landau map are rational or real. Identifying this map with the amplituhedron map on positroid varieties, and the associated recursions with promotion maps, yields a geometric mechanism for the emergence of positivity and cluster structures in planar $\mathcal{N}=4$ super Yang–Mills theory.

\end{abstract}

\makebox[0pt][l]{\hspace*{14cm}\raisebox{11.5cm}[0pt][0pt]{\small MPP-2026-49}}

\section{Introduction}
\label{sec:intro}
In perturbative quantum field theory, the computation of scattering amplitudes requires the evaluation of Feynman integrals. After integration, the resulting amplitude is a meromorphic, multivalued function of the external kinematic data. In general, the class of functions that can arise is not known \emph{a priori}: amplitudes may involve multiple polylogarithms, elliptic integrals, 
Calabi-Yau periods, etc.
Even in the absence of a complete classification of the relevant function space, one may still ask a fundamental structural question: where are the singularities of the amplitude, such as poles and branch cuts, located in kinematic space? 

At the level of the integrand, the situation is considerably different. Before integration, one has a rational function of the external and loop kinematics. Its singularities are poles. \emph{Landau analysis} provides a systematic framework, rooted in algebraic geometry, for studying these poles and how they give rise to singularities of the integrated amplitude, thereby creating a bridge between integrands and integrals~\cite{Landau1960}. The outcome is a collection of discriminants
in kinematic space containing the singularities of the corresponding integral.

The explicit form of Feynman integrals depends on the choice of kinematic variables~\cite{Weinzierl}. Recent work of Fevola, Mizera and Telen \cite{FMT1, MT} develops
algebraic geometry for Landau analysis in the framework of Gel'fand, Kapranov and Zelevinsky \cite{GKZ}.
They use the Lee-Pomeransky representation of Feynman integrals \cite{Weinzierl}.
There, the integrand is governed by the graph polynomial $\mathcal{U} + \mathcal{F}$. 
Its coefficients are linear functions in the kinematic parameters.
One seeks special parameter values for which the
hypersurface $\{\mathcal{U} + \mathcal{F} = 0\}$ is~singular.
The principal Landau determinant in \cite{FMT1} is a variant of
the principal $A$-determinant in \cite{GKZ}. It factors into
irreducible discriminants which identify the leading
singularities of the integral.

Our approach rests on the momentum representation of Feynman integrals, formulated with momentum twistors~\cite{hodges2013eliminating} and the Grassmannian ${\rm Gr}(2,4)$. This formalism provides a natural compactification of spacetime, in which \emph{conformal symmetries} of complexified Minkowski spacetime are realized by the linear action of ${\rm SL}(4,\mathbb{C})$~\cite{Drummond:2010qh, Mason:2009qx}.
Writing
 $N$ and $D$ as polynomials in the Pl\"ucker coordinates of
        $d+\ell$ lines in $\PP^3$, our integrals take the form
     \begin{equation}
     \label{eq:LMintegral}
        \mathcal{I}(\mathbf{M})\,\,=\int_\Gamma \frac{N(\mathbf{L};\mathbf{M})}{D(\mathbf{L};\mathbf{M})} {\rm d}\mu(\mathbf{L}) \, .
    \end{equation}
   
        The Feynman integral~\eqref{eq:LMintegral} is a function of the
            $d$ external lines    $\mathbf{M}=(M_1,\ldots,M_d) \in {\rm Gr}(2,4)^d$.
            One integrates over
            $\ell$ internal lines
             $\mathbf{L}=(L_1,\ldots,L_\ell) \in {\rm Gr}(2,4)^\ell$,
             for some cycle
    $\Gamma$ and measure ${\rm d}\mu$ on ${\rm Gr}(2,4)^\ell$. The denominator $D$ is a product over pairwise line incidences, corresponding to propagator conditions of Feynman integrals,  written as in \cite[Eq. (2)]{HMPS1}:
    \begin{equation} 
    \label{eq:Denom}
        D(\mathbf{L};\mathbf{M})\,\,:=\,\,\prod_{a<b} \langle L_a L_b \rangle \prod_{a,i}  \langle L_a M_i \rangle.
    \end{equation}

A proper treatment of Landau analysis for the integral~\eqref{eq:LMintegral} requires the stratification of the \textit{on-shell space}, which is the vanishing locus of~\eqref{eq:Denom} in the space of internal and external kinematics~\cite{Hannesdottir:2022iterated,helmer2024landau,Landau1960}. The \textit{Landau variety} is the branch locus of the \textit{Landau map}, the projection from the on-shell space to the space of external data $\mathbf{M}$~\cite{Landau1960}; it contains the locus where~\eqref{eq:LMintegral} develops singularities, namely poles or branch cuts. In this work we focus on \textit{leading} (and \textit{super-leading}) Landau singularities, arising from strata for which the generic fiber of the Landau map is zero-dimensional. The finitely many points in these fibers are the \textit{leading singularities}. They also correspond to maximal residues of the integrand in~\eqref{eq:LMintegral}, maximal iterated discontinuities and algebraic prefactors in expressions of the integral~\cite{ArkaniHamed:2009dn, Cachazo:2008vp}.

In this setting, we study the following properties of the fibers of the Landau map:
\begin{itemize}
    \item the \emph{generic size} of the fiber corresponds to the \emph{LS degree};
    \item a \emph{change} in the size of the fiber is detected by \emph{LS discriminants} and \emph{SLS resultants};
    \item \emph{reality}:
    all points in the fiber are real when
    the graph is outerplanar and ${\bf M}$
    lies in the positive region $\mathcal{M}^{>0}$. This implies positivity of
    LS discriminants and supports the expectation that amplitudes in planar $\mathcal{N}=4$ SYM are regular on $\mathcal{M}^{>0}$;
    \item \emph{rationality}: special fibers whose points are rational functions of the external kinematics ${\bf M}$.
    The specialized LS discriminants and SLS resultants factor into cluster variables for
    a higher Grassmannian. This perspective sheds light on the emergence of cluster structures in scattering amplitudes in planar $\mathcal{N}=4$ SYM.
\end{itemize}

\begin{center}
\begin{tabular}{c c}
\hline
\emph{Property of the fiber} & \emph{Geometric/physical interpretation} \\
\hline
Generic size & LS degree \\
Change of size & LS discriminant and SLS resultant \\
Reality of points & Positivity, regularity on $\mathcal{M}^{>0}$ \\
Rationality of points & Cluster structure \\
\hline
\end{tabular}
\end{center}

\smallskip

Momentum twistor variables are well-suited for planar dual conformal invariant integrals~\cite{Drummond:2010qh}. Our results therefore have direct applications to scattering amplitudes in planar $\mathcal{N}=4$ super Yang--Mills (SYM) theory, where additional remarkable structures arise, notably \emph{positive geometries} and \emph{cluster algebras}~\cite{amplituhdr,GGSVV}. At fixed multiplicity and loop order, the full color-ordered amplitude can be written as a single integral of the form~\eqref{eq:LMintegral}, whose integrand is the \textit{canonical form} of the loop amplituhedron~\cite{amplituhdr}. The on-shell varieties considered in this work arise naturally in the boundary stratification of loop amplituhedra. This perspective also suggests refinements of Landau analysis involving the amplituhedron boundary and its adjoint hypersurface, which encodes the numerator $N$ in~\eqref{eq:LMintegral}~\cite{chicherin2026geometric,dennen2017landau, Ranestad:2024adjoints}. Related geometric approaches in momentum-twistors include the work of Song He \emph{et al.}~\cite{HJLY}.

Another remarkable structure of planar $\mathcal{N}=4$ SYM is the appearance of cluster algebras in the singularities of amplitudes, first observed by Golden \emph{et al.}~\cite{GGSVV}. This was followed by the discovery of \emph{cluster adjacencies}~\cite{DFG} and related structures~\cite{GP23}, culminating in the cluster bootstrap program~\cite{CHDDDFGvHMP}. 
The recent work~\cite{EZLPTSW} explains the emergence of
cluster structures at tree level in terms of the positive geometry of the amplituhedron. The key idea  is to formulate BCFW recursion in terms of maps between cluster algebras, the \emph{BCFW promotion maps}. 

These cluster structures are closely tied to striking positivity phenomena in planar $\mathcal{N}=4$ SYM. A central expectation is that singularities of amplitudes do not intersect  \textit{positive kinematic space} $\mathcal{M}^{>0}$, namely the quotient of a positive Grassmannian by the torus action.
Amplitudes are regular on $\mathcal{M}^{>0}$. From this perspective, Landau analysis provides a natural arena in which to study both positivity and the cluster structure of amplitude singularities.

What has remained missing in the literature is a first-principles explanation for the emergence of positivity and cluster structures in planar $\mathcal{N}=4$ SYM, together with a systematic algebraic-geometric framework for Landau analysis in momentum-twistor space. We provide such a framework by treating the fibers of the Landau map as primary objects and studying their size, degeneration, reality, and rationality. The main question we address is how these fibers can be studied intrinsically in terms of line-incidence geometry on ${\rm Gr}(2,4)$, and how their geometry controls (leading, super-leading, and next-to-leading) Landau singularities. 

Our answer combines algebraic geometry, combinatorics, and computation: we interpret LS discriminants and SLS resultants as multigraded  Hurwitz forms and  Chow forms, develop recursive factorization formulas for them, and relate their positive and rational regimes to positroid geometry, promotion maps, and amplituhedra. This yields a mechanism for the emergence of positivity and cluster algebraic structures in scattering amplitudes of planar $\mathcal{N}=4$ SYM, while providing effective tools for computing Landau singularities.

\smallskip

We now discuss the structure and main contributions of this paper.
In Section \ref{sec:feynman} we review the representation in (\ref{eq:LMintegral}) from Feynman integrals in momentum space. 
This translation arises because (complexified) Minkowski space
is naturally compactified by the Grassmannian ${\rm Gr}(2,4)$.
This explains why lines in $\PP^3$ and their incidences make up 
   the denominator (\ref{eq:Denom}) of the integrand,
 and serves as the foundation for our
  approach to   Landau analysis.

Section \ref{sec:line-incidence} starts with a review,
 based on \cite{HMPS1}, of the incidence variety
 $V_G$ of a Feynman graph $G$.
 This leads us to  the Landau map $\psi_u$ in (\ref{eq:LandauMap}).
 The fibers of $\psi_u$ are the main objects of study in this article.
 We focus on planar Landau diagrams arising from outerplanar graphs $G$. In our language $G$ is the subgraph of the dual to a Landau diagram, whose vertices correspond to loop variables.
Theorem \ref{thm:outerplanar} characterizes the
irreducible decomposition of their incidence varieties.
In Section \ref{sec:discriminants} we introduce
discriminants for leading singularities (LS)
and resultants for super-leading singularities (SLS).
We define these as multigraded Chow forms \cite{OT} and Hurwitz forms
\cite{PSS} of line incidence varieties. This ensures that they are irreducible.
We derive explicit formulas for small graphs 
in Proposition \ref{prop:eins} and Theorem~\ref{thm:16168}.

In Section \ref{sec:hurwitz}
we determine the degree of the SLS resultant  (Theorem \ref{thm:osserman}) and
the expected degree of the LS discriminant (Proposition \ref{prop:bbexpected}).
Example \ref{ex:5cycle} shows how these degrees are found with
 the software from \cite[Section 6]{PSS}.
Section \ref{sec:nls} is concerned with
next-to-leading singularities (NLS). Here  the general
fiber of the Landau map $\psi_u$ is a curve.
The double box diagram yields an elliptic curve
(see Theorem \ref{thm:NLSdoublebox}).
The fiber becomes a singular curve when the NLS discriminant vanishes.
Higher genus cases are featured in Examples
\ref{ex:multigenera} and \ref{ex:NLSthreelines}.

The classical Poisson formula \cite[\S 13.1.A]{GKZ}
computes the resultant of several polynomials by
evaluating one polynomial at the common zeros of the others.
This leads to  recursive formulas for resultants.
In Section \ref{sec:recursive} we develop an analogous theory
for LS discriminants and SLS resultants. The main results in that section
are the factorization formulas in Theorems
\ref{thm:discr_fact} and~\ref{thm:res_fact}. In~\cite{Caron-Huot:2025recursive} the authors explored recursive structures in Landau analysis relying on the optical theorem for scattering amplitudes.
By contrast, our approach is  rooted in algebraic geometry,
and it arises naturally in the language of discriminants and resultants~\cite{GKZ}.

In Section \ref{sec:rationality}
we assume that the boundary lines in ${\bf M}$ satisfy
certain prescribed incidences. Physically, when two external lines intersect, the corresponding external momentum at that vertex of the dual graph is on shell; otherwise it is off shell, while the propagators remain massless.
This degenerations allow for geometric constructions 
for all points in the fiber of the Landau map.
These constructions translate into rational formulas,
like those in Example \ref{ex:rat_trian}.
Theorems \ref{thm:trivalenttree} and \ref{thm:rat_tr_of_l_gon}
describe this scenario for  trees and triangulated $\ell$-gons.
Rational fibers yield factorizations of the LS discriminant,
as shown in Theorem~\ref{thm:tr_rat}.

Section \ref{sec:positivity}
centers around the Reality Conjecture  \ref{conj:reality_conjecture},
which states that the fiber of the Landau map
is fully real whenever the boundary data ${\bf M}$ 
come from the positive Grassmannian.
 Affirmative results for special cases 
appear in Proposition \ref{prop:copositivity_resultant}
and Theorem \ref{thm:copositive}.
We also present computational validations of 
 our conjecture
using tools from numerical algebraic geometry \cite{BT}.
These imply that the LS discriminant is Grassmann copositive (Theorem 
\ref{thm:copositive}). 

Positroids and amplituhedra
have become important for particle physics 
in recent years. We connect these 
with line incidence varieties in Section \ref{sec:positroid}.
Our Landau map $\psi_u$ equals the
amplituhedron map on a positroid variety.
This rests on
plabic graphs, Grassmann graphs and vector-relation
configurations~\cite{m4tiling, VRC}.
Theorems  \ref{thm:positroid} and \ref{thm:positroid2} identify 
LS discriminants and SLS resultants with 
Chow-Lam forms and Hurwitz-Lam forms~\cite{PS}.

In Section \ref{sec:promotion_maps} we show that the promotion maps introduced in~\cite{VRC} match our recursive line constructions, in which solutions to Schubert problems are substituted into discriminants for smaller graphs. This provides the geometric mechanism underlying the recursive factorization of Landau singularities and connects rational Landau diagrams with rational positroids. Section \ref{sec:cluster} is devoted to the Cluster Factorization Conjecture \ref{conj:CFC}, which states that the irreducible factors of rational LS discriminants are cluster variables. This explains the emergence of cluster structures in Landau singularities from first principles. There we prove this conjecture for trees in Theorem \ref{th:cluster_fact_trees} and present further examples beyond the tree case. In Section \ref{sec:future} we conclude with a discussion of directions for future research.


\section{From Feynman Graphs to Lines in 3-Space}
\label{sec:feynman}
Lines and the coordinates that represent them correspond to {\em momentum twistors} in physics.
The denominator $D({\bf L};{\bf M})$ seen in (\ref{eq:Denom}) is a function of
$\ell+d$ momentum twistors $L_i$ and $M_j$. This function is a product of {\em propagators}.
The resulting polynomial is a pairwise product of momentum twistors.
A factor vanishes if and only if the two lines of a pair intersect in $\PP^3$.

In what follows we derive this representation~(\ref{eq:LMintegral}) for a 
Feynman integral,
and we explain where the lines in $\PP^3$ come from.
This section is expository, and it is intended primarily for
 mathematicians to whom the many different lives
of a Feynman integral are unfamiliar.
Our point of departure is the momentum representation \cite[Section 2.5.1]{Weinzierl}.
In what follows, we assume that the given  Feynman graph is planar
and that spacetime is four-dimensional.

We start with $d$ external momenta $p_1,\ldots,p_d$ and
$\ell$ internal momenta $k_1,\ldots,k_\ell$.
Each of these is a vector in Minkowski space $\RR^{1,3}$.
In coordinates, we write
$p_i = (p_{i0},p_{i1},p_{i2},p_{i3})$ for $i \in [d] $ 
and $k_j = (k_{j0},k_{j1},k_{j2},k_{j3})$ and $j \in [\ell]$.
Here we abbreviate $[\ell] = \{1,\ldots,\ell\}$. Momentum vectors
are multiplied using the Lorentzian inner product on $\RR^{1,3}$. For example,
 \begin{equation}\label{eq:Lorentzian_product}
     p_i \cdot k_j \,=\, p_{i0} k_{j0} - p_{i1} k_{j1} - p_{i2} k_{j2} - p_{i3} k_{j3}  \quad
{\rm and} \quad
k_j^2 \,= \, k_j \cdot k_j \,= \,k_{j0}^2 - k_{j1}^2 - k_{j2}^2 - k_{j3}^2 \, .
 \end{equation} 
We draw  the given Feynman graph as a planar diagram.
The diagram has $\ell$ loops and $d$ legs (pendant edges).
Let $e$ be the number of internal edges.
There are $e+1-\ell$ internal vertices.
We direct the $d+e$ edges arbitrarily. Each directed edge is labeled
by a momentum vector as follows. The  $d$ legs are labeled by
$p_1,\ldots,p_d$. For each loop we select one of its edges,
and we label these edges by $k_1,\ldots,k_\ell$. The other
$e-\ell$ internal edges are labeled so that momentum conversation
holds at every internal vertex.
This means that the sum of the labels of the incoming edges
equals the sum of the labels of the outgoing edges.
With this rule, every edge has a unique label,
which is an integer linear combination of the $p_i$ and $k_j$.

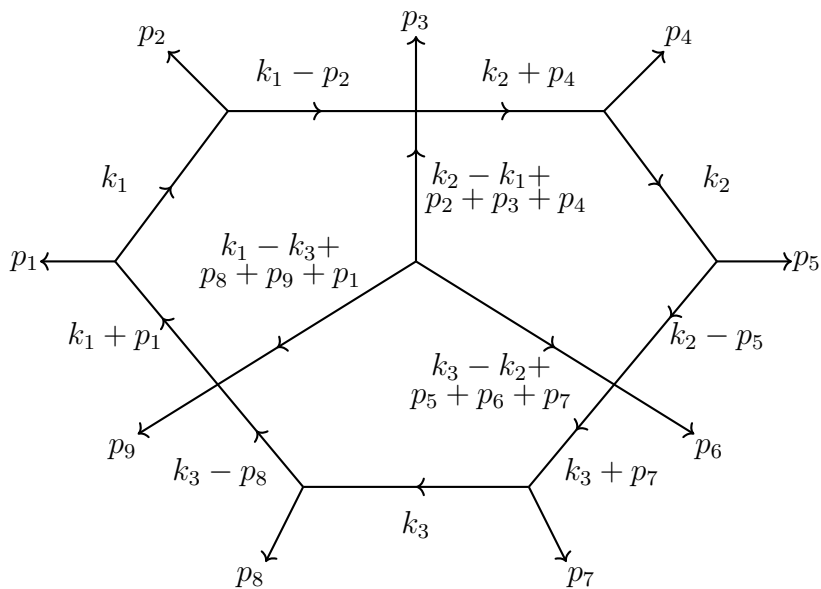
\begin{figure}[h]
\centering
\begin{tikzpicture}

    \begin{scope}[very thick,decoration={markings,mark=at position 0.25 with {\arrow{>}}, mark = at position .75 with {\arrow{>}}}] 
        \draw[postaction={decorate}, black, thick] (-2.5, 2)--(2.5, 2);
        \draw[postaction={decorate}, black, thick] (4, 0)--(1.5, -3);
        \draw[postaction={decorate}, black, thick] (-1.5, -3)--(-4, 0);
    \end{scope}

    \begin{scope}[very thick,decoration={markings,mark=at position 0.5 with {\arrow{>}}}] 
    
        \draw[postaction={decorate}, ->, black, thick] (0,0)--(0,3);
        \draw[postaction={decorate}, ->, black, thick] (0,0)--(3.7,-2.3);
        \draw[postaction={decorate}, ->, black, thick] (0,0)--(-3.7,-2.3);
        \draw[postaction={decorate}, black, thick] (2.5, 2)--(4, 0);
        \draw[postaction={decorate}, black, thick] (1.5, -3)--(-1.5,-3);
        \draw[postaction={decorate}, black, thick] (-4, 0)--(-2.5, 2);
    \end{scope}

    \draw[->, black, thick] (-2.5, 2)--(-3.3, 2.8);
    \draw[->, black, thick] (2.5, 2)--(3.3, 2.8);
    \draw[->, black, thick] (-4, 0)--(-5,0);
    \draw[->, black, thick] (4, 0)--(5,0);
    \draw[->, black, thick] (-1.5, -3)--(-2, -4);
    \draw[->, black, thick] (1.5, -3)--(2, -4);

    \node at (-1.8, .2) {$k_1-k_3+$};
    \node at (-1.8, -.2) {$p_8 + p_9 + p_1$};
    \node at (1, 1.15) {$k_2-k_1+$};
    \node at (1.2, .8) {$p_2+p_3+p_4$};
    \node at (1, -1.4) {$k_3 - k_2 + $};
    \node at (1, -1.8) {$p_5+p_6+p_7$};
    \node at (-4, -1) {$k_1+p_1$};
    \node at (4, -1) {$k_2-p_5$};
    \node at (-4, 1.1) {$k_1$};
    \node at (4, 1.1) {$k_2$};
    \node at (-1.5, 2.5) {$k_1-p_2$};
    \node at (1.5, 2.5) {$k_2 + p_4$};
    \node at (2.6, -2.8) {$k_3 + p_7$};
    \node at (-2.6, -2.8) {$k_3-p_8$};
    \node at (0, -3.5) {$k_3$};

    \node at (-3.5, 3) {$p_2$};
    \node at (3.5, 3) {$p_4$};
    \node at (-5.2, 0) {$p_1$};
    \node at (-2.2, -4.2) {$p_8$};
    \node at (2.2, -4.2) {$p_7$};
    \node at (0, 3.2) {$p_3$};
    \node at (3.9, -2.5) {$p_6$};
    \node at (-3.9, -2.5) {$p_9$};
    \node at (5.2, 0) {$p_5$};
\end{tikzpicture}
\caption{The triple pentagon is a Feynman graph at $\ell\!=\!3$ loops for $d\!=\!9$ interacting~particles.}
\label{fig:TP1}
\end{figure}

We write $\mathbf{P}=(p_1,\ldots,p_d) \in (\mathbb{R}^{1,3})^d$ 
 for the external momenta and 
  $\mathbf{K}=(k_1,\ldots,k_\ell) \in (\mathbb{R}^{1,3})^\ell$ for the loop momenta.
  With this abbreviation, our Feynman integral takes the form
  \begin{equation} \label{eq:feynmanKP}
    I(\mathbf{P})\,\,\,=\,\,\int_\Gamma \frac{N(\mathbf{K};\mathbf{P})}{D(\mathbf{K};\mathbf{P})} {\rm d}\mu(\mathbf{K}) \, ,
\end{equation}
    for some cycle $\Gamma$ and some integration measure ${\rm d}\mu$ on $(\mathbb{C}^{1,3})^\ell$. 
    Here $N,D$ are polynomials in the coordinates of the momenta.
    The denominator $ D(\bf{K};{\bf P}) $ is the product of
    the squares of all $d+\ell$ edge labels.
    For the legs we multiply the squares $p_i^2    = p_{i0}^2 - p_{i1}^2 - p_{i2}^2 - p_{i3}^2$.
The internal edges contribute the squares of certain linear
combinations of the $p_i$ and $k_j$. 
The internal momenta are integrated out in (\ref{eq:feynmanKP})
and the resulting integral $I(\mathbf{P})$ is a
function of the external momenta.
This function is transcendental,
and one is interested in its analytic properties.
Of special interest are the singularities, namely
 poles and branch cuts. Landau analysis 
 largely disregards the numerator, and, for our purposes, we may set $N({\bf K};{\bf P}) = 1$.

We illustrate the labeling and the construction of the integral for
a Feynman graph called the {\em triple pentagon}.
Here $\ell = 3, d = 9$ and $e = 12$, so the graph has ten internal vertices.
One central trivalent vertex is surrounded by a ring of nine vertices, each
of which has a leg. The central vertex lies on three pentagons.
In the $i$th pentagonal loop, we use the label $k_i$
for the edge opposite to the central vertex. The labeled triple pentagon is shown in
Figure~\ref{fig:TP1}.

In Figure \ref{fig:TP1} we see that, at each of the ten vertices,
the sum of the incoming labels equals the sum of the outgoing labels.
At the central vertex this is due to momentum conversation:
$$ \qquad p_1 + p_2 + p_3 + p_4 + p_5 + p_6 + p_7 + p_8 + p_9 \,\, = \,\, 0 \qquad
{\rm in} \,\, \, \RR^{1,3} \, . $$

We now perform a linear change of coordinates. Namely we replace
each of the momenta by dual variables. This means that we introduce
momentum vectors for each of the regions in the Feynman graph. The nine external regions
are labeled by $x_1,x_2,\ldots,x_9$. The three pentagons are
labeled by $y_1,y_2,y_3$, as shown in Figure \ref{fig:TP2}.
The rule is that the variable labeling a directed edge is the difference of the variables of the
regions it separates. In our example,
$$ k_1 = x_2-y_1,
 k_2 = x_5-y_2,
 k_3 = x_8-y_3,\,
 p_1 = x_1-x_2,
 p_2 = x_2-x_3,
 p_3 = x_3-x_4, \ldots,
  p_9 = x_9-x_1. $$
  The coordinates $({\bf X},{\bf Y})$ are called {\em dual variables}, because
  they reflect the duality of planar graphs. Note that in general the number of particles of a Feynman graph need not be the same as the number of external regions $d$.
  In dual variables, $D({\bf K};{\bf P})$ in (\ref{eq:feynmanKP}) becomes
  $$ D({\bf X}; {\bf Y}) \,\,\,\, = \,
  \begin{matrix} &
  ( x_1 - y_1)^2  (x_2-y_1)^2 (x_3-y_1)^2 \,
  (x_4-y_2)^2 (x_5-y_2)^2 (x_6-y_2)^2 \\ & \,
 (x_7-y_3)^2 (x_8-y_3)^2 (x_9-y_3)^2 \,\,
(y_1-y_2)^2 (y_2-y_3)^2 (y_3-y_1)^2 \, .
\end{matrix}
  $$
As before, each  of the $12$ factors in this product is evaluated using the 
Lorentzian inner product~\eqref{eq:Lorentzian_product}. 
For instance,
  the third factor is written in scalar quantities as follows:
  \begin{equation}
  \label{eq:neuneins} (x_3 - y_1)^2 \,\,=\,\,
  (x_{30}-y_{10})^2 -
  (x_{31}-y_{11})^2 -
(x_{32}-y_{12})^2 -
(x_{33}-y_{13})^2 .
\end{equation}

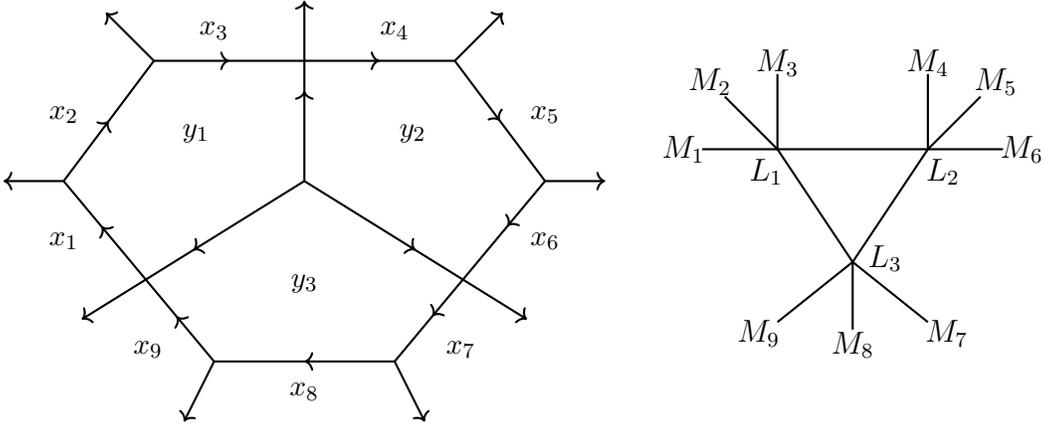
\begin{figure}[h]
    \centering
\begin{tikzpicture}[scale = .8, every node/.style={font=\small}]

    \begin{scope}[very thick,decoration={markings,mark=at position 0.25 with {\arrow{>}}, mark = at position .75 with {\arrow{>}}}] 
        \draw[postaction={decorate}, black, thick] (-2.5, 2)--(2.5, 2);
        \draw[postaction={decorate}, black, thick] (4, 0)--(1.5, -3);
        \draw[postaction={decorate}, black, thick] (-1.5, -3)--(-4, 0);
    \end{scope}

    \begin{scope}[very thick,decoration={markings,mark=at position 0.5 with {\arrow{>}}}] 
    
        \draw[postaction={decorate}, ->, black, thick] (0,0)--(0,3);
        \draw[postaction={decorate}, ->, black, thick] (0,0)--(3.7,-2.3);
        \draw[postaction={decorate}, ->, black, thick] (0,0)--(-3.7,-2.3);
        \draw[postaction={decorate}, black, thick] (2.5, 2)--(4, 0);
        \draw[postaction={decorate}, black, thick] (1.5, -3)--(-1.5,-3);
        \draw[postaction={decorate}, black, thick] (-4, 0)--(-2.5, 2);
    \end{scope}

    \draw[->, black, thick] (-2.5, 2)--(-3.3, 2.8);
    \draw[->, black, thick] (2.5, 2)--(3.3, 2.8);
    \draw[->, black, thick] (-4, 0)--(-5,0);
    \draw[->, black, thick] (4, 0)--(5,0);
    \draw[->, black, thick] (-1.5, -3)--(-2, -4);
    \draw[->, black, thick] (1.5, -3)--(2, -4);

    \node at (-1.8, .8) {$y_1$};
    \node at (1.8, .8) {$y_2$};
    \node at (0, -1.7) {$y_3$};
    \node at (-4, -1) {$x_1$};
    \node at (4, -1) {$x_6$};
    \node at (-4, 1.1) {$x_2$};
    \node at (4, 1.1) {$x_5$};
    \node at (-1.5, 2.5) {$x_3$};
    \node at (1.5, 2.5) {$x_4$};
    \node at (2.6, -2.8) {$x_7$};
    \node at (-2.6, -2.8) {$x_9$};
    \node at (0, -3.5) {$x_8$};
\end{tikzpicture}
\begin{tikzpicture}
\node at (0,0) {};
\begin{scope}[shift = {(0, .5)}]
    

    \draw[black, thick] (2, 3)--(4, 3);
    \draw[black, thick] (2, 3)--(3, 1.5);
    \draw[black, thick] (4, 3)--(3, 1.5);

    \draw[black, thick] (2, 3)--(1, 3);
    \draw[black, thick] (2, 3)--(1.3, 3.7);
    \draw[black, thick] (2, 3)--(2, 4);

    \draw[black, thick] (4, 3)--(4, 4);
    \draw[black, thick] (4, 3)--(5, 3);
    \draw[black, thick] (4, 3)--(4.7, 3.7);

    \draw[black, thick] (3, 1.5)--(4, .7);
    \draw[black, thick] (3, 1.5)--(2, .7);
    \draw[black, thick] (3, 1.5)--(3, .6);
    
    \node at (1.85, 2.70) {$L_1$};
    \node at (4.20, 2.70) {$L_2$};
    \node at (3.43, 1.55) {$L_3$};

    \node at (2, 4.15) {$M_3$};
    \node at (.75, 3) {$M_1$};
    \node at (1.10, 3.9) {$M_2$};

    \node at (4, 4.15) {$M_4$};
    \node at (5.25, 3) {$M_6$};
    \node at (4.9, 3.9) {$M_5$};

    \node at (4.25, .55) {$M_7$};
    \node at (3, .4) {$M_8$};
    \node at (1.75, .55) {$M_9$};

    \end{scope}

\end{tikzpicture}
\caption{Regions of the triple pentagon labeled by dual variables (left).
The dual graph labeled by momentum twistors (right).
These are lines in $3$-space with prescribed incidences.}
\label{fig:TP2}
\end{figure}

We finally perform a second linear change of coordinates, namely 
from  the dual variables $({\bf X},{\bf Y})$ to 
momentum twistors $({\bf L},{\bf M})= 
(L_1,\ldots,L_\ell,M_1,\ldots,M_d)$.
Each $L_j$ and each $M_k$ is a line in $\PP^3$.
The $\ell+d$ lines can be represented as the row spans of $2 \times 4$ matrices as follows:
\begin{equation}
\label{eq:getLM}
L_j \,\, = \,\, \begin{pmatrix}
1 & 0 & y_{j0}+y_{j3} & y_{1j} - i y_{2j} \\
0 & 1 & y_{j1}+ i y_{j2} & y_{j0} - y_{j3} \end{pmatrix}
\quad {\rm and} \quad
 M_k \,\, = \,\, \begin{pmatrix}
1 & 0 & x_{k0}+x_{k3} & x_{k1} - i x_{k2} \\
0 & 1 & x_{k1}+i x_{k2} & x_{k0} - x_{k3} \end{pmatrix}.
\end{equation}
Here, $i = \sqrt{-1}$. The formula (\ref{eq:getLM}) is derived in
\cite[Section 8]{HMPS1}. Its key property is as follows:
the determinant of any $4 \times 4$ matrix obtained by stacking
two of the $2 \times 4$ matrices vanishes if and only if the
two lines intersect in $\PP^3$. Remarkably, the
factors in the  denominator of our integrand
are precisely  the algebraic translations of this geometric condition.
This is the content of \cite[Equation (29)]{HMPS1}. For instance, 
the expression in (\ref{eq:neuneins}) is the $4 \times 4$ determinant
$$ \langle L_1 M_3 \rangle \,\,:=\,\,{\rm det} \begin{pmatrix}
  L_1 \\ M_3
\end{pmatrix}\, \,\, = \,\,\, (x_{30}-y_{10})^2 -
  (x_{31}-y_{11})^2 -
(x_{32}-y_{12})^2 -
(x_{33}-y_{13})^2 \, .
$$
We use the notation $\langle L_j L_k \rangle$ and $\langle L_j M_k \rangle$ for such $4 \times 4$ determinants. Each of these
determinants
is the inner product of two Pl\"ucker coordinate vectors,
as in \cite[Equation (2)]{HMPS1}.

In summary,  for the triple pentagon, the denominator  $ D({\bf L}; {\bf M}) $ in   (\ref{eq:Denom})
is the product
\begin{equation}\label{eq:prop_tr} \!\!\!\!\!
    \langle L_1 M_1\rangle \langle L_1 M_2\rangle \langle L_1 M_3\rangle \langle L_2 M_4\rangle\langle L_2 M_5\rangle \langle L_2 M_6\rangle \langle L_3 M_7\rangle \langle L_3 M_8\rangle\langle L_3 M_9\rangle
\langle L_1 L_2\rangle \langle L_1 L_3\rangle\langle L_2 L_3\rangle .
\end{equation} 
The $12$ factors correspond to the $12$ edges in the dual graph on the
right in Figure \ref{fig:TP2}.
The denominator is $D({\bf X};{\bf Y})$ in dual variables,
and it equals $D({\bf K};{\bf P})$ in  (\ref{eq:feynmanKP})
when written in the
original momentum coordinates. The latter is the
product of the $12$ edge labels in Figure \ref{fig:TP1}.

The same two coordinate transformations can be carried out for any
planar Feynman graph. One ends up with a graph with
$\ell $ vertices $L_i$ and $d$ vertices $M_k$, like that
shown on the right of Figure \ref{fig:TP2}. The latter graph
specifies incidence relations among $\ell+d$ lines in~$\PP^3$. 

To study singularities of the integral~\eqref{eq:LMintegral}, one introduces \textit{Landau diagrams}. A Landau diagram looks like a scalar Feynman graph, and it represents a stratum in the vanishing locus of~\eqref{eq:Denom}. 
The left picture in Figure~\ref{fig:TP2} is a Landau diagram for the underlying Feynman integral, where all propagators in~\eqref{eq:prop_tr} are set to zero. In the following, we rather work with the \textit{dual} graph to a Landau diagram, as for example the graph on the right in Figure~\ref{fig:TP2}.
For simplicity, we also refer to the dual graph as the Landau diagram.
  That diagram precisely encodes the incidences of lines in $3$-space cutting out the stratum under consideration. 

\section{Incidence Varieties and the Landau Map}
\label{sec:line-incidence}

We now turn to incidence varieties of lines in $\PP^3$,
following the set-up in \cite[Section 3]{HMPS1}. Our ambient space
${\rm Gr}(2,4)^\ell$ is the product of $\ell$ Grassmannians
which sits inside  $(\PP^5)^\ell$ via the Pl\"ucker embedding.
For a graph $G \subset \binom{[\ell]}{2}$,  the
incidence variety $V_G$ is the set of tuples $(L_1,\ldots,L_\ell)$
such that $\langle L_i L_j \rangle = 0$ for all $ij \in G$.
The dimension of $V_G$ and a characterization for which graphs $V_G$ is irreducible or a complete intersection is given in \cite[Section 5]{HMPS1}. 

The {\em multidegree} $[V_G]$ is the class of the incidence variety $V_G$ in the Chow ring
$$ A^* \bigl( (\PP^5)^\ell , \mathbb{Z} \bigr) \,\,= \,\,
\ZZ[t_1,t_2,\ldots,t_\ell] /\langle t_1^6, t_2^6, \ldots,t_\ell^6\rangle . $$
We write this as a homogeneous polynomial whose degree is the codimension of $V_G$ in $(\PP^{5})^\ell$:
\begin{equation}
\label{eq:multidegree}
 [V_G] \,\, = \,\,  \sum_{u \in \NN^\ell} \gamma_u \cdot t_1^{5-u_1} t_2^{5-u_2} \cdots \,t_\ell^{5-u_\ell}. 
 \end{equation}
Here $u$ runs over nonnegative integer vectors $u$ with
$|u| = \sum_{i=1}^\ell u_i  =  {\rm dim}(V_G) =: d$.
The coefficient $\gamma_u$ is the {\em LS degree} of the auxiliary graph $G_u$,
which is obtained by attaching $u_i$ pendant edges at the vertex $i$ of $G$.
Hence $G_u$ has $\ell+d$ vertices and $|G|+d$ edges. In relation to Section~\ref{sec:feynman},
the auxiliary graph $G_u$ should be understood as the planar dual of a \textit{leading Landau diagram}.
Geometrically, the LS degree $\gamma_u$ is the degree of the {\em Landau map}
\begin{equation}
\label{eq:LandauMap}
 \psi_u \,: \, {\rm Gr}(2,4)^{\ell + r} \,\rightarrow \, {\rm Gr}(2,4)^d ,\,\,
({\bf L}, {\bf M}) \,\rightarrow \, {\bf M} \,.
\end{equation}
This map  deletes the $\ell$ internal~lines.
The fiber of $\psi_u$ over a point ${\bf M} = (M_1,\ldots,M_d) \in {\rm Gr}(2,4)^d$
consists of all line configurations ${\bf L} = (L_1,\ldots,L_\ell)$ which satisfy $d$
Schubert conditions. Namely, the $i$th line $L_i$ is required to satisfy
$\langle L_i M_j \rangle = 0$ for the $u_i$ external lines $M_j$ where $j$ is incident to $i$ in $G_u$. Hence, $\gamma_u$ counts the expected number of leading singularities.

\begin{example}[Triangle] \label{ex:triangle}
Let $\ell = 3$ and $G = \{12,13,23\}$.
The variety $V_G$ is the union of
two irreducible components $V_{[3]}$
and $V_{[3]}^*$ of dimension $d=9$ in ${\rm Gr}(2,4)^3$.
Configurations in $V_{[3]}$ are triples of concurrent lines,
while configurations in $V_{[3]}^*$ are triples of coplanar lines.
We have $[V_G] = [V_{[3]}] + [V_{[3]}^*]$, and the
two irreducible components have the same multidegree:
$$ [V_{[3]}] \,=\, [V_{[3]}^*] \,\, = \,\, 4 t_1 t_2 t_3 (t_1 + t_2) (t_1 + t_3)(t_2 + t_3)
\,\,=\,\,
4  \,\biggl(\,\sum_{\pi \in S_3} t_{\pi(1)}^3 t_{\pi(2)}^2 t_{\pi(3)} \biggr)
\,+\, 8 \, t_1^2 t_2^2 t_3^2 \, . $$      
We now fix $u=(3,3,3)$. The graph $G_u$ is shown on the right in
Figure \ref{fig:TP2}. The LS degree $\gamma_u$ equals $16=8+8$.
Indeed, given $9$ general lines ${\bf M}=(M_1,\ldots,M_9)$ in $\PP^3$,
there are precisely $8$ concurrent triples ${\bf L}=(L_1,L_2,L_3)$
which satisfy the $9$ Schubert conditions they impose:
$$  \begin{small} \begin{matrix}
\langle L_1 M_1 \rangle =\langle L_1 M_2\rangle =\langle L_1 M_3\rangle =
\langle L_2 M_4\rangle =\langle L_2 M_5\rangle =\langle L_2 M_6\rangle =
\langle L_3 M_7\rangle =\langle L_3 M_8\rangle=\langle L_3M_9\rangle = 0.
\end{matrix}
\end{small}
$$
Similarly, there are $8$ coplanar triples ${\bf L}$ satisfying these equations.
These ${\bf L}$ form a {\em Cayley octad}, i.e.~they are found by
intersecting three quadratic surfaces in $\PP^3$.
See \cite[Example~3.2]{HMPS1}.
\end{example}

The physical meaning of the Landau map $\psi_u$
concerns the singularities of the Feynman integral 
(\ref{eq:feynmanKP}). In Section \ref{sec:feynman}
we showed that this integral can be written in the form
 (\ref{eq:LMintegral}), where we are integrating
 over ${\bf L}$ in ${\rm Gr}(2,4)^\ell$
 and the resulting function of ${\bf M}$
 has its domain in ${\rm Gr}(2,4)^d$.
The denominator $D({\bf L}; {\bf M})$ is the
product of propagators, and these form 
 an arrangement of hypersurfaces. Higher-codimension singular loci are obtained by intersecting subsets of these hypersurfaces, that is, by setting a subset of propagators to zero, $
D_{i_1}=\cdots=D_{i_k}=0$.
Singularities of maximal codimension are called \emph{leading singularities} (LS).
These are the points in the fibers of $\psi_u$.
Their number is the LS degree $\gamma_u$, which we read off from $[V_G]$.

An important ingredient in our framework is the decomposition 
of the variety $V_G$ into irreducible components.
We denote these irreducible varieties by $V_{G,\sigma}$ where $\sigma$ runs
over an appropriate finite set of labels.
For many graphs of interest,  these labels arise from 
concurrent and coplanar triples, as seen in Example \ref{ex:triangle}.
We  will now make this precise.

In graph theory,
a graph $G$ on $[\ell]$ is called {\em outerplanar} if it is a subgraph of a
triangulated $\ell$-gon.  The vertices are  labeled
in cyclic order $1,2,\ldots,\ell$. Being outerplanar is a strong form of
planarity. Similar to Kuratowski's characterization of planar graphs
by excluded minors, a graph $G$ is outerplanar if and only if it has no subgraph
that is a  subdivision of $K_4$ or $K_{2,3}$.
If $G$ is outerplanar then its triangles are visible in the
planar drawing. We write $\tau(G)$ for the number of triangles.
For instance, if $G$ is a triangulated $\ell$-gon then $\tau(G)= \ell-2$
and the variety $V_G$ has $2^{\ell-2}$ irreducible components
\cite[Theorem 4.10]{HMPS1}.
These components are labeled by the bicolorings $\sigma$ of
the triangles. A black triangle means that three lines
are concurrent, while a white triangle means that they
 are coplanar. See \cite[Figure 1]{HMPS1} for  $\ell=4$.

\begin{theorem} \label{thm:outerplanar}
Let $G$ be an outerplanar graph. Then the incidence variety $V_G$ is
a complete intersection. It has precisely
$2^{\tau(G)}$ irreducible components, as described above,
one for each bicoloring $\sigma$ of the triangles.
In particular, $V_G$ is irreducible when the graph $G$ is triangle-free.
\end{theorem}

\begin{proof} 
Every outerplanar graph $G$ is $(2,3)$-sparse, i.e.~it is independent in the rigidity matroid \cite[Remark 4.4]{HMPS1}.
We claim that $G$ is strictly contraction stable. This notion was defined in \cite[Section 5]{HMPS1}. 
We then conclude,  by \cite[Theorem 5.4]{HMPS1}, that $V_G$ is a complete intersection.

Our claim states that,  for any partition $J = \{J_1, \ldots, J_r\}$ of $[\ell]$ with singletons omitted:
   \[   |G_J| \,+\, 4(|J_1| + |J_2| + \cdots + |J_r| - r) \,\,> \,\, |G|. ~~    \]
   In this formula,
        $G_J$ is the graph on $\ell - s_J$ vertices obtained by identifying all vertices in each $J_i$, merging parallel edges, and deleting resulting loops. Here $s_J = |J_1| + \cdots |J_r| - r$. Note that the graph $(G_{\{J_1\}})_{\{J_2\}}$ obtained by first contracting with respect to the partition $\{J_1\}$ and then $\{J_2\}$ is the same as that obtained by contracting with respect to $\{J_1, J_2\}$.
    Therefore, it suffices to prove the inequality for partitions with only one block. 
    
    Let $J = \{J_1\}$ be a partition of $[\ell]$ with  one non-singleton block.
    We  must show that $|G| - |G_J| < 4 |J_1| - 4$. Let $J_1^c = [\ell] \setminus [J_1]$, and $H = \{ij \in G ~|~ i \in J_1,~ j \in J_1^c \}$,
     and $K = \{i \in J_1^c ~|~ ij \in G \,\,\hbox{for some} \,\,j \in J_1\}$.
The numbers of edges in our graphs are
    \[
    |G| \,=\, |G[J_1]| + |G[J_1^c]| + |H| \quad {\rm and} \quad |G_J| \,=\, |G[J_1^c]| + |K| ,
    \]
    where $G[S]$ is the subgraph of $G$ induced on $S$. Hence,
        $|G| - |G_J| = |G[J_1]| + |H| - |K|$. Since $G$ is (2,3)-sparse, $|G[J_1]| \leq 2|J_1| - 3$, so it remains to show that $|H| - |K| < 2|J_1| - 1$. Note that $H$ is bipartite with parts $J_1$ and $K$.
        As a minor of $G$, it is outerplanar.         It is easy
        to show, by induction,
         that $|H| \leq 2|J_1| + |K| -2$ for any outerplanar bipartite graph $H$.

    We next prove that $V_G$ has $2^{\tau(G)}$ irreducible components indexed by bicolorings of $G$'s triangles.
      For $\ell \leq 3$ this is clear. Let $\ell >3$. Since $G$ is outerplanar, there exists a vertex $v$  of degree $r \leq 2$. Consider the       map $\pi : {\rm Gr}(2,4)^{\ell} \rightarrow {\rm Gr}(2,4)^{\ell-1}$ which forgets the line associated to $v$. The image of 
      $\pi$ is $V_{\widetilde{G}}$, where $\widetilde{G}$ is the graph on $\ell-1$ vertices obtained from $G$ by deleting $v$. By our induction hypothesis, $V_{\widetilde{G}}= X_{\widetilde{G}}$ has $2^{\tau(\widetilde{G})}$ irreducible components. Let $V_{\tilde{\sigma}} \subset V_{\widetilde{G}}$ be any irreducible component, given by a bicoloring $\tilde{\sigma}$ of the $\tau(\widetilde{G})$ triangles in $\widetilde{G}$.

    Suppose that $r=1$. The fiber is irreducible and $\pi$ defines a fiber bundle over $V_{\tilde{\sigma}}$. Hence there exists an irreducible component $V_{\sigma} \subset V_G$ associated to $V_{\tilde{\sigma}}$. In this case, $\tau(G) = \tau(\widetilde{G})$ and thus $V_G$ has $2^{\tau(G)}$ irreducible components indexed by the bicolorings of triangles of $G$. 
    
    Now suppose that $r=2$ and denote by $u$ and $w$ the vertices adjacent to $v$. If $uw \notin G$, then $\pi^{-1}(V_{\tilde{\sigma}})$ is again irreducible and the desired result follows as in the $r = 1$ case. Lastly, if $r = 2$ and $uw \in G$, $\tau(G) = \tau(\tilde{G}) + 1$ and the fiber $\pi^{-1}(V_{\tilde{\sigma}}) = V_{\sigma_w} \cup V_{\sigma_b}$ is the union of two irreducible varieties corresponding to the bicolorings of the new triangle $\{u,v,w\}$ in $G$. The concurrent component $V_{\sigma_b}$ can be parametrized by taking the line corresponding to $v$ to be spanned by the intersection point of $L_u$ and $L_w$ and a new generic point in $\PP^3$. The coplanar component $V_{\sigma_w}$ can be parametrized similarly. This parametrization yields a Zariski-open subset of $V_{\sigma_b}$ where $L_u \neq L_w$, however since $V_G$ is strictly contraction-stable, the subvariety $X_G$ where all $\ell$ lines are distinct is equal to $V_G$. Thus we see that $\pi^{-1}(V_{\tilde{G}}) = V_G = X_G$ has $2^{\tau(G)}$ irreducible components indexed by the bicolorings of the triangles of $G$ as required. 
    \end{proof}

A planar graph can have more or fewer irreducible components than $2^{\tau(G)}$, where $\tau(G)$ is the number its triangles in a planar drawing.

\begin{example}
A first example is the kite graph $
G=\{12,13,14,23,34,25,45\}$. The variety $V_G$ is a complete intersection of codimension $7$, but it has six irreducible components, whereas $G$ has only two triangles:
$6>2^{\tau(G)}=4$. Note that $G$ is a subgraph of the wheel~$W_5$.
\end{example}

Understanding the irreducible decomposition for planar graphs beyond the outerplanar case remains an interesting open problem. A first family for which such a decomposition is known is given by the wheel graphs $G=W_\ell$: by \cite[Theorem~4.9]{HMPS1}, the number of irreducible components is $
2^{\tau(G)}-2\tau(G)$,
which is smaller than the expected count $2^{\tau(G)}$.



\section{Discriminants and Resultants}
\label{sec:discriminants}

Fix a graph $G$ on $[\ell]$,  and let $d$
be the dimension of its incidence variety $V_G$.
Consider $u \in \NN^\ell$ such that $|u| = d$ and
$\gamma_u \geq 2$. 
For generic external data ${\bf M} \in {\rm Gr}(2,4)^d$,
the fiber $\psi_u^{-1}({\bf M})$ contains precisely
$\gamma_u $ line configurations ${\bf L}$.
We are interested in the condition  on ${\bf M}$ such that two points 
in $\psi_u^{-1}({\bf M})$ come together. This condition gives a hypersurface
in ${\rm Gr}(2,4)^d$, defined by a unique (up to scaling)
polynomial $\Delta_{G,u}({\bf M})$, which we call
the {\em LS discriminant}.

First suppose that $V_G$ is irreducible. In this case,
$\Delta_{G,u}({\bf M})$ is an irreducible polynomial.
Namely, it is precisely the {\em multigraded Hurwitz form}
of $V_G$. This is defined in \cite{PSS} and it 
represents a hypersurface in $(\PP^5)^d$.
This hypersurface lives in ${\rm Gr}(u_1,6) \times \cdots \times
{\rm Gr}(u_d,6)$, and we write it in primal Stiefel coordinates.
This means that ${\bf M}$ is represented by a $6 \times d$ matrix
whose columns are Pl\"ucker coordinates on the $d$ factors of $(\PP^5)^d$.
Furthermore, in our application to Landau analysis, we always restrict
$\Delta_{G,u}({\bf M}) $ to  ${\rm Gr}(2,4)^d$.
This means that each column of the 
$6 \times d$ matrix representing ${\bf M}$ satisfies the quadratic Pl\"ucker relation.

\begin{example}[$\ell=1$]
We have $u=(4)$, as in \cite[Example 3.1]{HMPS1},  and $\psi_u^{-1}({\bf M})$ consists of the
two lines that are incident to four  lines $M_1,M_2,M_3,M_4$.
Here the LS discriminant equals
 \begin{equation}
\label{eq:discriminant_4_lines} \qquad
\Delta_{K_1,(4)} \,\, = \,\,
 {\rm det} \begin{small} \begin{bmatrix} 
0 & M_1 M_2 &  M_1 M_3 &  M_1 M_4 \\
 M_1 M_2 & 0 & M_2 M_3 & M_2 M_4 \\
   M_1 M_3 &  M_2 M_3 & 0 & M_3 M_4 \\
     M_1 M_4 & M_2 M_4& M_3 M_4 & 0
\end{bmatrix} ,
\quad {\rm where} \,\,M_i M_j := \langle M_i M_j \rangle.
\end{small}
\end{equation}
This Gram determinant vanishes when the two lines coincide.
See \cite[Example 6.1]{PSS}.
\end{example}

We next analyze the graph  $G = \{12\}$ with $\ell = 2$ vertices.
  Set $u = (4,3)$.
 For simplicity of notation we write
${\bf L} = (X,Y)$ and ${\bf M} = (A,B,C,D; E,F,G)$.
The variety $V_G$ is the hypersurface in 
${\rm Gr}(2,4)^2$ defined by $XY = 0$. The equations for our Schubert problem are
\begin{equation}
\label{eq:sevenlines} AX  \,=\, BX \,=\, CX \,=\, DX \,\,=\,\,
EY \,=\, FY \,= \,GY \,\, = \,\, 0.
\end{equation}
These are inner product like
$\, AX \, = \,
a_{12} x_{34} - a_{13} x_{24} + a_{14} x_{23} + a_{23} x_{14}
- a_{24} x_{13} + a_{34} x_{12}$.
 
 The LS degree equals $\gamma_u = 4$, i.e.~there are four 
 pairs ${\bf L} \in V_G$ such that (\ref{eq:sevenlines}) holds.
 The LS discriminant $\Delta_{G,u}$ has degree $(4,4,4,4,4,4,4)$ in the
entries of the seven given Pl\"ucker vectors $A,B,C,D,E,F,G$.
We will present  an explicit formula in Propositions \ref{prop:eins} and \ref{prop:zwei}.

\smallskip

Assuming that $A,B,C,D,E,G$ span $\RR^6$, 
and introducing variables $x_1,x_2,\ldots,x_6$, we write 
$$ X \,\, = \,\, x_1 A + x_2 B + x_3 C + x_4 D + x_5 E + x_6 G . $$
This basic choice is recorded in the following Gram matrix, which is symmetric of size $6 \times 6$:
\begin{equation}
\label{eq:sixbysix} \begin{footnotesize}
\mathcal{G} \quad = \quad
\begin{bmatrix}
0 & AB & AC & AD & AE & AG \\
AB & 0 & BC & BD & BE & BG \\
AC & BC & 0 & CD & CE & CG \\
AD & BD & CD & 0 & DE & DG \\
AE & BE & CE & DE & 0 & EG \\
AG & BG & CG & DG & EG & 0 \\
\end{bmatrix}. \end{footnotesize}
\end{equation}
Our assumption that $X$ intersects $A,B,C,D$ yields four linear equations in $x_1,\ldots,x_6$, namely
\begin{equation}
\label{eq:fourlinear}
 \begin{matrix}
 AX & = & AB x_2 + AC x_3 + AD x_4 + AE x_5 + AG x_6 & = & 0 ,\\
 BX & = & AB x_1 + BC x_3 + BD x_4 + BE x_5 + BG x_6 & = & 0 , \\
 CX & = & AC x_1 + BC x_2 + CD x_4 + CE x_5 + CG x_6 & = & 0, \\
 DX & =  & AD x_1 + BD x_2 + CD x_3 + DE x_5 + DG x_6 & = & 0.
\end{matrix}
\end{equation}
The Pl\"ucker relation $XX = 2 (x_{12} x_{34} - x_{13} x_{24} + x_{14} x_{23}) = 0$
implies the quadratic constraint
\begin{equation}
\label{eq:quad1} XX \,\, = \,\,AB x_1 x_2+AC x_1 x_3+A D x_1 x_4+
\cdots
     +DE x_4 x_5+DG x_4 x_6+EG x_5 x_6 \,\, = \,\,  0.
\end{equation}
For the LS discriminant, we require that the two transversal lines
of $E,F,G $ and $X$ coincide:
\begin{equation}
\label{eq:quad2}
\!\!{\rm det} \! \begin{footnotesize} \begin{bmatrix}
0 & EF & EG &EX \\
 EF & 0 & FG & FX \\
 EG & FG & 0 & GX \\
 EX & FX & GX & 0 
\end{bmatrix} =  \begin{matrix}
(AE^2FG^2-2AE\,AF\,EG\,FG-\cdots+AG^2EF^2)\,x_1^2 \,+\,\, \end{matrix} \cdots \,\, = \,\, 0.\quad \end{footnotesize}
\end{equation}

\begin{proposition} \label{prop:eins}
The LS discriminant equals $ 1/{\rm det}(\mathcal{G})^2$ times the resultant of
six equations in six unknowns, namely
the four linear forms in (\ref{eq:fourlinear}) 
and the two quadrics in (\ref{eq:quad1}) and~(\ref{eq:quad2}).
\end{proposition}

\begin{proof}
We consider four linear equations ${\bf T} {\bf x} = 0$
and two quadratic equations ${\bf x}^t {\bf U} {\bf x} = {\bf x}^t {\bf V} {\bf x} = 0$,
where ${\bf T}$ is a $4 \times 6$ matrix and
${\bf U}$ and ${\bf V}$ are symmetric $6 \times 6$ matrices.
The resultant is a polynomial of degree $20$ in the $66 = 4 \cdot 6 + 2 \cdot \binom{7}{2}$ 
unknown entries of these matrices. Namely, it has degree $4$ in each row of ${\bf T}$
and it is quadratic in each of ${\bf U}$ and ${\bf V}$.

We now substitute (\ref{eq:fourlinear}) for ${\bf T}$,
(\ref{eq:quad1}) for ${\bf U}$, and (\ref{eq:quad2}) for ${\bf V}$.
A computation shows that the resulting polynomial is nonzero,
and it  has degree $26 = 4 +4+4+4+2+8$ in the $28$
quantities $AB,AC,\ldots,FG$. 
More precisely,  it is multihomogeneous of degree $(8,8,8,8,8,4,8)$ in the
seven Pl\"ucker vectors $A,B,C,D,E,F,G$.
 The extraneous factor has degree $(4,4,4,4,4,0,4)$.
We find that it is the square of the determinant
of the Gram matrix $\mathcal{G}$.
\end{proof}

It remains to give a formula for the resultant
${\rm Res}_{1,1,1,1,2,2}$ of six homogeneous equations in six unknowns, of
the indicated degrees.  Let $f_1,f_2,f_3,f_4$ denote the linear forms in
${\bf T}{\bf x}$, and set $f_5 = {\bf x} {\bf U} {\bf x}^t$ and
$f_6 = {\bf x} {\bf V} {\bf x}^t$. Our formula is an explicit polynomial
in the $66$ coefficients.

For $i =0,1,\ldots,6$, write  $\sigma_i$ for the operator which replaces 
the unknowns $x_1,\ldots,x_i$ with new unknowns $y_1,\ldots,y_i$.
Let $\Delta$ be the $ 6 \times 6$ matrix whose entries are
the divided differences 
$$ \qquad \Delta_{ij} \,\,=\,\,  \frac{\sigma_{i-1}(f_j) - \sigma_i(f_j)}{x_i-y_i}
\quad \qquad {\rm for}\,\, \,\, 1 \leq i,j \leq 6. $$
The  {\em B\'ezoutian} ${\bf B}$ is the $6 \times 6$-matrix whose
entries are the mixed partials
$ \,\partial^2 {\rm det}(\Delta)/ \partial x_i \partial y_j$.
These entries are multilinear forms in the coefficients of 
our polynomials $f_1,f_2,f_3,f_4,f_5,f_6$.

\begin{proposition} \label{prop:zwei}
The resultant for four linear forms and two quadrics in six unknowns is
\begin{equation} \label{eq:tenbyten}
 {\rm Res}_{1,1,1,1,2,2} ({\bf T}, {\bf U},{\bf V}) \quad  = \quad 
{\rm det} \begin{bmatrix}\, {\bf B} & \,{\bf T}^t \\ {\bf T} \,& 0\, \end{bmatrix}. \end{equation}
This $10 \times 10$ determinant is a polynomial of degree
$(4,4,4,4,2,2)$ in the $66$ coefficients.
\end{proposition}

\begin{proof}
This is an instance of the
determinantal formulas for resultants derived by
D'Andrea and Dickenstein in \cite{Carlos}.
The critical degree for our equations is
$t_6 =  1+1+1+1+2+2-6=2$. Setting
$t=1$ in \cite[Lemma 5.3]{Carlos},
we obtain the block matrix of size $10 \times 10$ shown in (\ref{eq:tenbyten}).
\end{proof}

For general graphs, the variety $V_G$ is reducible. In this case,
$\Delta_{G,u}$ is a product whose irreducible factors include the
   LS discriminants $\Delta_{G,\sigma,u}$
 of the  components $V_{G,\sigma}$. 
In addition, there are certain mixed  factors.
We show this for the triangle graph $G = K_3$ with $u = (3,3,3)$.

\begin{theorem} \label{thm:16168}
The LS discriminant $\Delta_{K_3,u}$ has degree $(48,48,48)$. It has three
irreducible factors, corresponding to the three scenarios for two of the $8+8$ solutions
to come together: \vspace{-0.15cm}
\begin{itemize}
\item the LS discriminant $ \Delta_{K_3,u, \mathbf{b}} $ of the concurrent component $V_{[3]}$ has degree $(16,16,16)$, 
\vspace{-0.2cm}
\item the LS discriminant $ \Delta_{K_3,u, \mathbf{w}} $ of the coplanar component $V_{[3]}^*$ has degree $(16,16,16)$,
\vspace{-0.2cm}
\item the square of the irreducible mixed LS discriminant, which has degree $(8,8,8)$. 
\vspace{-0.2cm}
\end{itemize}
The last factor vanishes when a solution of  (\ref{eq:sevenlines}) on $V_{[3]}$ merges with a
solution on $V_{[3]}^*$.
\end{theorem}

This theorem is proved by a symbolic computation with {\tt Macaulay2} \cite{M2}.
In the next section we discuss an independent validation based on
the techniques   in \cite[Section 6]{PSS}.

\smallskip

Another singular scenario  for the integral   (\ref{eq:LMintegral})
arises when the Schubert problem is overconstrained,
namely if  $u \in \NN^\ell$ satisfies $|u| = d+1$.
The Landau map $\psi_u$ is not surjective,~but its image is a hypersurface in
${\rm Gr}(2,4)^{d+1}$. This hypersurface is defined by a polynomial
$R_{G, u}$ which we call the {\em SLS resultant}.
A {\em superleading singularity} (SLS) arises when
$R_{G, u}({\bf M})$ = 0.

 \begin{example}[$\ell=1$] \label{ex:commontrans}
 The SLS resultant is
the condition for the five lines $M_j$ in $\PP^3$ to
have a common transversal. This is the
determinant of 
the Gram matrix for five Pl\"ucker vectors:
\begin{equation}\label{eq:resultant_5_lines}
R_{K_1,(5)}({\bf M}) \,\, = \,\,\,
 {\rm det} \begin{small} \begin{bmatrix} 
0 & M_1 M_2 &  M_1 M_3 &  M_1 M_4 &  M_1 M_5 \\
 M_1 M_2 & 0 & M_2 M_3 & M_2 M_4 & M_2 M_5 \\
  M_1 M_3 &  M_2 M_3 & 0 & M_3 M_4 & M_3 M_5 \\
  M_1 M_4 & M_2 M_4& M_3 M_4 & 0 & M_4 M_5 \\
  M_1 M_5  & M_2 M_5& M_3 M_5 & M_4 M_5 & 0
\end{bmatrix}. \end{small}
\end{equation}
We see that this SLS resultant has degree $(2,2,2,2,2)$
as a hypersurface in $(\PP^5)^5$.
\end{example}

\begin{example}[$\ell=2$] \label{ex:SLS2}
Let $G = \{12\}$ and $u = (4,4)$.
We augment the system (\ref{eq:sevenlines}) by
one more equation $HY = 0$. Here, the SLS resultant
$R_{G,u}$ is a large irreducible polynomial 
which is homogeneous of degree $4$ in each of the eight Pl\"ucker vectors in
${\bf M} = (A,B,\ldots,G,H)$. 
\end{example}

\begin{example}[$\ell=3$] \label{ex:SLS3}
Let $G = K_3$ as in Theorem~\ref{thm:16168}, but now we take
$u = (4,3,3)$.
The SLS resultant is the
product of two irreducible factors, one for each 
irreducible component of $V_G$.
Each of these two factors has degree $(8,8,8,8;\,4,4,4;\,4,4,4)$ in
the $10$ external lines in ${\bf M}$. We conclude that the SLS
  resultant $R_{K_3,u}$ has degree $(16,16,16,16;\,8,8,8;\,8,8,8)$.
  In the next section we shall see that these numbers are 
    the LS degrees $\gamma_{u-e_i}$ for $i=1,2,3$.
\end{example}

\section{Chow and Hurwitz}
\label{sec:hurwitz}

Osserman and Trager \cite{OT} introduced the multigraded Chow forms
for arbitrary subvarieties in a product of projective spaces.
Their definition shows that the multigraded Chow forms
of the incidence variety $V_G$ are precisely the
SLS resultants $R_{G,u}$, for the various vectors $u \in \NN^d$ that satisfy $|u|  = d+1$.
These Chow forms are multiplicative over the irreducible decomposition of $V_G$.
Hence $R_{G,u}$ is the product of the irreducible SLS resultants
$R_{G,u, \sigma}$, where $V_{G,\sigma}$ ranges over the components of $V_G$.
For instance, if $G$ is outerplanar then this product is over the $2^{\tau(G)}$ bicolorings
of the triangles of $G$. From \cite[Theorem 1.2]{OT} we obtain:

\begin{theorem} \label{thm:osserman}
The SLS resultant $R_{G,u}$ is a homogeneous polynomial of  degree $\gamma_{u-e_i}$ in the
six Pl\"ucker coordinates of each of the $u_i$ external lines 
that are attached to vertex $i$ of $G$.
\end{theorem}

It is important to recognize the different coordinate systems that occur in this result.
In the setting of \cite{OT},  the multigraded Chow form is a hypersurface in
${\rm Gr}(u_1,6) \times \cdots \times {\rm Gr}(u_d,6)$, and the
degree $\gamma_{u-e_i}$ refers to Pl\"ucker coordinates on the $i$th factor
${\rm Gr}(u_i,6)$. Each line that is attached to the $i$th vertex of $G$ has
its own six Pl\"ucker coordinates, and we  write these as
the rows of a $u_i \times 6$ matrix.  That matrix furnishes primal
Stiefel coordinates for  ${\rm Gr}(u_i,6)$ in the set-up of
Osserman and Trager.
Since the $u_i \times u_i$ minors of this $u_i \times 6$ Stiefel matrix
are linear in each row,  the degree of $R_{G,u}$ in each of the $u_i$ lines in the $i$-th
block of ${\bf M}$ coincides with the degree on ${\rm Gr}(u_i,6)$ in \cite{OT}.
And this is the
coefficient $\gamma_{u-e_i}$ in the multidegree $[V_{G}]$.

Examples \ref{ex:SLS2} and \ref{ex:SLS3} illustrate
Theorem \ref{thm:osserman}. 
We present a case study for a larger~graph.

\begin{example}[$\ell=5$]
Consider the triangulated pentagon $G = \{12, 13, 14, 15, 23, 34, 45\}$.
Its ideal $I_G$ is radical, and a complete intersection by \cite[Theorem 5.4]{HMPS1}.
We also know from \cite[Theorem 4.10]{HMPS1} that the ideal has the prime decomposition
$$ \begin{matrix} I_G & = &
I_{12345} \, \cap \,I_{12345}^* \, \,\,\cap\,\,\,
(I_{123} + I_{1345}^*) \,\cap \,
(I_{123}^* + I_{1345}) \phantom{dodododododododododododo} \\ & & \cap\,\,\,
(I_{1234} + I_{145}^*) \,\cap \, 
(I_{1234}^* + I_{145})\, \,\,\cap \,\,\,
(I_{123} + I_{134}^* + I_{145} )\,\cap\,
(I_{123}^* + I_{134} + I_{145}^* ). \end{matrix}
$$
Since $I_G$ is a complete intersection, its multidegree  is the product of its generating degrees:
$$ [I_G] \,\,=\,\, 32\, t_1 t_2 t_3 t_4 t_5 \, (t_1+t_2)  (t_1+t_3)  (t_1+t_4)  (t_1+t_5)  (t_2+t_3) (t_3+t_4) (t_4+t_5) . $$
The expansion is a sum of $82$ terms of degree $12$, whose
coefficients add up to $4096$. For each of the eight irreducible components $V_{G,\sigma}$
we compute the multidegree $[V_{G,\sigma}]$, and we check that these add up correctly,
i.e.~$[I_{12345}] + [I_{12345}^*] + \cdots + [I_{123}^* \!+\! I_{134} \!+\! I_{145}^* ] = [I_G]$.
Among the eight summands, there are four distinct pairs of multidegrees. Namely, we find the following:
\begin{itemize}
\item[(a)] $[I_{12345}]  =  [I_{12345}^*]$ has $30$ terms with coefficients in $\{4,8\}$, \vspace{-0.1cm}
\item[(b)] $[I_{123} + I_{1345}^*] = [I_{123}^* + I_{1345}]$ has $ 56$ terms with coefficients in $\{4,8,12,16,24\}$,  \vspace{-0.1cm}
\item[(c)] $[I_{1234} + I_{145}^*] =  [I_{1234}^* + I_{145}]$ has $ 56$ terms with coefficients in $\{4,8,12,16,24\}$,  \vspace{-0.1cm}
\item[(d)] $[I_{123} {+} I_{134}^* {+} I_{145}] =  [I_{123}^* {+} I_{134} {+} I_{145}^*]$
has $ 82$ terms with coeffs in $\{4, 8, 12, 16, 20, 24, 32\}$.
\end{itemize}
Equipped with this data, we can now apply Theorem \ref{thm:osserman} to
derive the expected multidegrees of all irreducible SLS resultants $R_{G,u,\sigma}$.
Here $\sigma$ is a bicoloring that marks one of the eight irreducible components,
and $u$ ranges over all vectors in $\NN^5$ that satisfy $|u| = d+1 = 14$.

We consider the five vectors $u$ that have  four entries $2$ and one entry $3$.
Each vector specifies a graph $G_u$ with $19$ vertices and $21$ edges,
and hence a planar Landau diagram with $19$ regions.
The degrees of the associated SLS resultants are listed in the following table:
$$ \begin{small}
\begin{matrix}
 \,\,u \,=\!\! & (2, 3, 3, 3, 3) &  (3, 2, 3, 3, 3) &  (3, 3, 2, 3, 3) &  (3, 3, 3, 2, 3) &  (3, 3, 3, 3, 2) \smallskip \\
{\rm (a)} &   (0, 8, 8, 8, 8) &  (8, 0, 8, 8, 8) &  (8, 8, 0, 8, 8) &  (8, 8, 8, 0, 8) &  (8, 8, 8, 8, 0) \\
{\rm (b)} & (16, 8, 24, 16, 16) &  (8, 0, 8, 8, 8) &  (24, 8, 16, 16, 16) &  (16, 8, 16, 0, 8) &  (16, 8, 16, 8, 0) \\
{\rm (c)} & (16, 16, 16, 24, 8) &  (16, 0, 8, 16, 8) &  (16, 8, 0, 16, 8) &  (24, 16, 16, 16, 8) &  (8, 8, 8, 8, 0) \\
{\rm (d)} & (32, 16, 32, 32, 16) &  (16, 0, 8, 16, 8) &  (32, 8, 16, 24, 16) &  (32, 16, 24, 16, 8) &  (16, 8, 16, 8, 0) \\
\end{matrix}
\end{small}
$$
Observe that the degree in the Pl\"ucker coordinates of each external line $M_j$ attached to the same internal line $L_i$ is the same. For instance, the first entry $(0,8,8,8,8)$ in the table above reflects that this SLS resultant has degree $0$ in the 2 external lines attached to vertex $1$, and degree $8$ in the external lines attached to all other vertices. 
The last row  (d) concerns the two most interesting components $V_{G,\sigma}$,
where $\sigma$ assigns alternating colors to the three triangles in $G$.
The SLS resultant in the middle is an irreducible polynomial
of degree $(32,8,16,24,16)$ in ${\bf M} = (M_1,M_2,M_3,M_4,M_5)$.
To obtain the degree of $R_{G, u}$, we add the degrees of $R_{G, u, \sigma}$ for each of the eight components given by the bicolorings $\sigma$. Since the degrees are preserved under Hodge star duality, we add
twice the four degrees in the third column.
We see that the reducible
 SLS resultant $R_{G,u} = \prod_\sigma R_{G,u,\sigma}$
has degree $(160, 64, 64, 128, 96)$ in~{\bf M}.
\end{example}

Pratt, Sodomaco and Sturmfels \cite{PSS}
introduced the multigraded Hurwitz forms
for an arbitrary subvariety in a product of projective spaces.
Their definition shows that the multigraded Chow forms
of the incidence variety $V_G$ are precisely the
LS discriminants $\Delta_{G,u}$,
for the various vectors $u \in \NN^d$ that satisfy $|u|  = d$.
These Hurwitz forms are reducible when $V_G$ has
multiple components. But the factorization is more
complicated than that for Chow forms.
We saw this in Theorem \ref{thm:16168}
for the incidence variety of the triangle,
$V_{K_3} \subset (\PP^5)^3$.

A complete understanding of the factorization properties
of LS discriminants would require a multigraded extension of
the results  in \cite[Section 4]{Stu} on degenerations of Hurwitz forms.
We shall not pursue this here, but instead we often work with irreducible $V_G$ or restrict to an irreducible component $V_{G, \sigma}$. For instance, we can take $G$
to be a triangle-free outerplanar graph. Even the special case
when $G$ is a tree is of interest for our application.

Consider the $\ell+|G|$ equations that define $V_G$ in $(\PP^5)^\ell$.
First, we have the $\ell $ Pl\"ucker relations, and these have degrees
$2 e_i$ for $i=1,\ldots,\ell$. Second, we have the equations
$\langle L_i L_j \rangle = 0$  for all edges $ij \in G$, and these have degrees
$e_i+e_j$. Since $V_G$ is a complete intersection, 
the following formula holds for its multidegree from which one easily extracts the LS degrees~$\gamma_u$: 
\begin{equation}\label{eq:LS_CI}
    [V_G] \,\,\, = \,\,\, 2^\ell \cdot t_1 t_2 \cdots t_\ell\, \cdot \prod_{ij \in G} ( t_i + t_j ).
\end{equation}
In \cite[Section 6]{PSS}, the authors 
compare $V_G$ with the complete intersection defined by
generic equations of the same degrees.
The following result is proved in \cite[Corollary 6.3]{PSS}.

\begin{proposition} \label{prop:bbexpected}
The LS discriminant $\Delta_{G,u}$ is a polynomial of degree $(b_1,\ldots,b_\ell)$ where
\begin{equation}
\label{eq:discdegree}
     \quad b_i \,\,\,\leq \,\,\,2 \gamma_u \, \, +\!
     \sum_{\substack{j\in[\ell]\\ u_j+\delta_{ij}>0}}
     \!\! \gamma_{u+e_i-e_j} \cdot \biggl(1-u_j- \delta_{ij}+{\rm degree}_G(i)
           \biggr)\,
\quad \hbox{for all}\,\,\, \,i \in [\ell] .
\end{equation}
\end{proposition}

The right hand side in (\ref{eq:discdegree}) is the expected 
degree of the LS discriminant.
Here $\delta_{ij}$ is Kronecker's delta, and ${\rm degree}_G(i)$ is the
number of edges in $G$ that are incident to vertex~$i$.
 We conjecture that equality holds in (\ref{eq:discdegree})
 if we take extraneous factors into consideration.
 
 Let us explain what this conjecture means for $\ell=2$. The LS discriminant 
 $\Delta_{K_2,(4,3)}$ was featured in Proposition \ref{prop:eins},
 where we saw that its degree  is $(b_1,b_2) = (4,4)$. The
 upper bound in Proposition \ref{prop:bbexpected}
 equals $(8,4)$. This is the true degree of the
 multigraded Hurwitz form for the complete intersection
of degrees $(2,0),(0,2),(1,1)$ in $\PP^5 \times \PP^5$.
When specializing to $V_G$, we obtain
$\Delta_{K_2,(4,3)}$ times an extraneous factor
of degree $(4,0)$. That extraneous factor is the
LS discriminant (\ref{eq:discriminant_4_lines}) for
the external lines $A,B,C,D$ that are incident to $X$ in (\ref{eq:sevenlines}).

 Pratt et al.~\cite[Section 6]{PSS} developed software in  the computer algebra system {\tt Macaulay2}
 for computing the expected degrees of all  LS discriminants $\Delta_{G,u}$ for a fixed graph $G$. 
 Here $u$ ranges over all terms in the multidegree $[V_G]$.
 In particular, the software also outputs all the LS degrees $\gamma_u$.
  Illustrations of how to use this {\tt Macaulay2} code are given in \cite[Figure~5]{PSS}
for $\ell =2$ and  in \cite[Figure 6]{PSS} for $\ell = 3$. In what follows we discuss a larger example.

\begin{example}[$5$-cycle] \label{ex:5cycle}
Fix $\ell = 5$ and $G = \{12,23,34,45,15\}$.
The expansion of $[V_G] = 32 t_1 t_2 t_3 t_4 t_5 (t_1 + t_2) (t_2+t_3)(t_3+t_4)(t_4+t_5)(t_1+t_5)$
has $31 $ monomials $\gamma_u t_1^{5-u_1} \cdots t_5^{5-u_5}$. 
Each of them satisfies $|u| = 15 = {\rm dim}(V_G)$.
We compute the
expected degrees of all LS discriminants $\Delta_{G,u}$. The command is
${\tt multidegHurwitz(5,\{\{1,2\},\{2,3\},\{3,4\},\{4,5\},\{1,5\}\}})$.
For instance, for $u = (2,3,3,3,4)$, the output reveals that the expected degree of $\Delta_{G,u}$ equals
$(32, 64, 96, 128, 192)$. The unique monomial in $[V_G]$ with $ \gamma_u = 64$
comes from $u = (3,3,3,3,3)$. Here, the
expected degree of the LS discriminant $\Delta_{G,u}$ equals $(256,256,256,256,256)$.
\end{example}

A key ingredient for computing the expected degrees of LS discriminants is the multisectional genus 
\cite[eqn (12)]{PSS} of the incidence variety $V_G$. We introduce this in the next~section.

\section{NLS Geometry}
\label{sec:nls}

 Fix a graph $G \subseteq \binom{[\ell]}{2}$ and let $d$
be the dimension of its incidence variety $V_G \subseteq {\rm Gr}(2,4)^\ell$.
Every vector $u \in \NN^\ell$ specifies
 a new graph $G_u$ with $e := \sum_{i=1}^\ell u_i$
pendant edges.  The Landau map
 $\psi_u :  V_{G_u} \rightarrow {\rm Gr}(2,4)^e$
 takes configurations to its external lines.
The fibers of $\psi_u$  have expected dimension $d-e$,
which we now assume to be nonnegative.
If $d=e$ then the fiber is finite, its points are the
{\em leading singularities} (LS), and  its cardinality is the {\em LS degree}.

Landau analysis is concerned with the discriminant of the Landau map $\psi_u$.
This discriminant is the hypersurface in ${\rm Gr}(2,4)^e$ whose points correspond to
singular fibers.  For $e=d$, we get the {\em LS discriminant}, which vanishes when two of
the leading singularities collide.

We are here interested in the next case, when the fibers are curves.
We thus set $e=d-1$. The fiber of  the Landau map $\psi_u$ 
over a general point ${\bf M} \in {\rm Gr}(2,4)^e$ is a smooth projective curve.  
When this curve is singular, then our integral has a 
{\em next-to-leading singularity (NLS)}.
Our section title NLS geometry refers to
the study of these curves and their degenerations.

Depending on  $u$, it can happen that the general fiber
$\psi^{-1}({\bf M})$ is a reducible curve.
The geometric genus of this curve is called the
 {\em NLS genus} of the graph $G_u$.
 If the curve is smooth and irreducible then this is the
 familiar genus of the corresponding Riemann surface.
The hypersurface in ${\rm Gr}(2,4)^e$ whose fibers 
are singular curves is the {\em NLS discriminant} $\Delta_{G,u}$.

Of course, we could go on and
examine the case $e=d-2$ where the
fibers are surfaces and we get an NNLS discriminant.
For $e=d-3$, the fibers would be threefolds and 
we get  N${}^3$LS discriminants, etc.
The ultimate Landau analysis would involve all of
these discriminants. For starters, we here focus on the
case of NLS discriminants where the fibers are curves.
Hence, for the rest of this section, we fix $e=d-1$
and we consider $u \in \NN^\ell$ with $|u| = e$.

\begin{example}[$\ell=1$]\label{ex:ell1_NSL}
Here $e = 3$ and $u=(3)$.
  The fiber of the Landau map $\psi_u$
consists of all lines $X$ which intersect three given lines
$M_1,M_2,M_3$ in $\PP^3$.
This fiber is a plane conic, so the NLS genus is~$0$.
The NLS discriminant $\Delta_{G,u}$
vanishes when the conic degenerates to two lines. It is  a polynomial with 
$978$ monomials of degree $(3,3,3)$
in the unknowns $(M_1,M_2,M_3)$.
To compute $\Delta_{G,u}$, we consider the linear equations
$M_1 X = M_2 X  =  M_3 X = 0$
in  unknowns~$x_{ij}$.
We solve it for three of the six unknowns, say $x_{23},x_{24},x_{34}$,
and we plug into the Pl\"ucker quadric $XX$.
The result is a quadratic form in $x_{12},x_{13},x_{14}$
whose coefficients are trilinear in $(M_1,M_2,M_3)$.
The Hessian determinant of this quadratic form equals
$\Delta_{G,u}$.
\end{example}

We define the {\em multisectional genus} of the graph $G$ to be the generating function
$$ g(G) \,\,  = \,\,  \sum_{|u| = e} g_u \cdot t_1^{5-u_1} \cdots \,t_\ell^{5-u_\ell} , $$
where $g_u$ is the NLS genus of $G_u$. This is the genus of the generic
fiber of the Landau map~$\psi_u$.
One can compute $g(G)$ with the {\tt Macaulay2}
implementation described in \cite[Section~6]{PSS}.
For instance, for the graph with $\ell=2$, the command on the left in \cite[Figure 2]{PSS} shows that
\begin{equation}
\label{eq:gK2}
 g(K_2) \,\, = \,\,  t_1^4 - t_1^3 t_2 + t_1^2 t_2^2 - t_1 t_2^3 + t_2^4. 
\end{equation}
The expected degree of the LS discriminant on the right in
(\ref{eq:discdegree}) is determined from the
multisectional genus via the formula 
$ 2 (g_{u-e_i} + \gamma_u - 1)$, which is
found in \cite[Theorem 3.4]{PSS}.
For instance, for the two-edge graph $K_2$ with $u = (4,3)$, we recover the expected degree
$$ \bigl( \,2(g_{3,3} + \gamma_{4,3} - 1) \,,\, 2 (g_{4,2} + \gamma_{4,3} - 1) \,\bigr)
\,\, = \,\, \bigl(\,2(1 + 4-1)\,,\, 2(-1+4-1) \,\bigr) \,\,=\,\, (8,4). $$
This equals $(4,0)$ plus the true degree $(4,4)$, as explained in
\cite[Example 6.2]{PSS} and after Proposition~\ref{prop:bbexpected}.
We next compute the NLS genera for the larger graph from Example \ref{ex:5cycle}.

\begin{example}[$5$-cycle] \label{ex:multigenera}
Fix $\ell = 5$ and $G = \{12,23,34,45,15\}$.
The multisectional genus $g(G)$ is a polynomial
with $735$ terms of degree $11$  in $t_1,t_2,t_3,t_4,t_5$. 
We compute this in {\tt Macaulay2} by typing
\ {\tt multiGenera(5,\{\{1,2\},\{2,3\},\{3,4\},\{4,5\},\{1,5\}\})}.
The output shows that the largest  NLS genus is $g_u = 65$. This
arises for $u = (3,2,2,2,2)$ and four others. Also interesting 
is $u = (3,2,2,3,1)$, for which the fiber $\psi^{-1}({\bf M})$ is
a curve of genus $g_u = 33$.
\end{example}

A graph with five vertices is rather large
when it comes to finding explicit formulas
for the NLS discriminant. We therefore return to 
$\ell=2$, with the aim of extending
Propositions \ref{prop:eins} and \ref{prop:zwei}
to the NLS setting, where now the Landau diagram is a {\em double box}.
Namely, we fix $G = K_2$ and $u=(3,3)$. We drop the external line $D$ in
(\ref{eq:sevenlines}). The fiber
$\psi^{-1}({\bf M})$ consists of all pairs $(X,Y) \in V_G$ such that
$AX = BX = CX = 0$ and $EY=FY=GY = 0$.
This is an elliptic curve. We saw in (\ref{eq:gK2}) that
the NLS genus is $g_{(3,3)} = 1$. Geometrically,
our elliptic curve is obtained by intersecting two quadratic surfaces in 
$\PP^3$.  The first quadric contains the lines $A,B,C$, and the
second quadric contains the lines $E,F,G$. This defines them uniquely.

\begin{theorem} \label{thm:NLSdoublebox}
The NLS discriminant for the double box is a polynomial $\Delta_{K_2,(3,3)}$ of degree $12$ in 
the Pl\"ucker coordinates of each of the six external lines
$A,B,C; E,F,G$. It is the mixed discriminant of two symmetric $4 \times 4$ matrices
$\,U$ and $\,V$ whose entries are given in (\ref{eq:umatrix}).
\end{theorem}

\begin{proof}
The mixed discriminant of $U$ and $V$ is the discriminant of
the univariate quartic $f(t) = {\rm det}(U+tV)$. According to
\cite[Lemma 2.1]{BFS}, this has $67753552 $ monomials of degree $(12,12)$.
It vanishes precisely when the quartic curve
 defined by $U$ and $V$ is singular in~$\PP^3$.
 
 The unique quadratic surface containing the lines $A,B,C$ in $\PP^3$ is the zero set of  the quadratic form  $ x^t U x =  \langle Ax| B | Cx \rangle $.  This is the compact notation 
 to be
 defined below (\ref{eq:remarkablycluster}). It translates into the
 following explicit formulas for the entries $u_{ij}$ of 
 the $ 4 \times 4$ matrix $U$:
 \begin{equation}
\label{eq:umatrix}
\begin{small}
\begin{matrix}
u_{11} & = & 2 a_{23} b_{24} c_{34}-2 a_{23} b_{34} c_{24}-2 a_{24} b_{23} c_{34}
            +2 a_{24} b_{34} c_{23}+2 a_{34} b_{23} c_{24}-2 a_{34} b_{24} c_{23} \\
u_{12} & = & (-b_{13} c_{24}+b_{14} c_{23}-b_{23} c_{14}+b_{24} c_{13}) a_{34}
+(a_{13} c_{24}-a_{14} c_{23}+a_{23} c_{14}-a_{24} c_{13}) b_{34} \\  & & 
      +(-a_{13} b_{24}+a_{14} b_{23}-a_{23} b_{14}+a_{24} b_{13}) c_{34} \\
u_{13} & = & (-b_{12} c_{34}-b_{14} c_{23}+b_{23} c_{14}+b_{34} c_{12}) a_{24}
    +(a_{12} c_{34}+a_{14} c_{23}-a_{23} c_{14}-a_{34} c_{12}) b_{24} \\  & & 
      +(-a_{12} b_{34}-a_{14} b_{23}+a_{23} b_{14}+a_{34} b_{12}) c_{24} \\
u_{14} & = & (b_{12} c_{34}-b_{13} c_{24}+b_{24} c_{13}-b_{34} c_{12}) a_{23}
 +(-a_{12} c_{34}+a_{13} c_{24}-a_{24} c_{13}+a_{34} c_{12}) b_{23} \\ & & 
      +(a_{12} b_{34}-a_{13} b_{24}+a_{24} b_{13}-a_{34} b_{12}) c_{23} \\
u_{22} & = &  2 a_{13} b_{14} c_{34}-2 a_{13} b_{34} c_{14}-2 a_{14} b_{13} c_{34}
 +2 a_{14} b_{34} c_{13}+2 a_{34} b_{13} c_{14}-2 a_{34} b_{14} c_{13} \\
u_{23} & = & (b_{12} c_{34}+b_{13} c_{24}-b_{24} c_{13}-b_{34} c_{12}) a_{14}
+(-a_{12} c_{34}-a_{13} c_{24}+a_{24} c_{13}+a_{34} c_{12}) b_{14} \\  & & 
      +(a_{12} b_{34}+a_{13} b_{24}-a_{24} b_{13}-a_{34} b_{12}) c_{14} \\
u_{24} & = & (-b_{12} c_{34}+b_{14} c_{23}-b_{23} c_{14}+b_{34} c_{12}) a_{13}
  +(a_{12} c_{34}-a_{14} c_{23}+a_{23} c_{14}-a_{34} c_{12}) b_{13} \\  & & 
      +(-a_{12} b_{34}+a_{14} b_{23}-a_{23} b_{14}+a_{34} b_{12}) c_{13} \\
      u_{33} & = & 2 a_{12} b_{14} c_{24}-2 a_{12} b_{24} c_{14}-2 a_{14} b_{12} c_{24}
      +2 a_{14} b_{24} c_{12}+2 a_{24} b_{12} c_{14}-2 a_{24} b_{14} c_{12} \\
u_{34} & = & (-b_{13} c_{24}-b_{14} c_{23}+b_{23} c_{14}+b_{24} c_{13}) a_{12}
    +(a_{13} c_{24}+a_{14} c_{23}-a_{23} c_{14}-a_{24} c_{13}) b_{12} \\ & & 
     +(-a_{13} b_{24}-a_{14} b_{23}+a_{23} b_{14}+a_{24} b_{13}) c_{12} \\
     u_{44} & = & 2 a_{12} b_{13} c_{23}-2 a_{12} b_{23} c_{13}-2 a_{13} b_{12} c_{23}
      +2 a_{13} b_{23} c_{12}+2 a_{23} b_{12} c_{13}-2 a_{23} b_{13} c_{12}
\end{matrix}
\end{small}
\end{equation}
The unique quadratic surface $V$ containing the lines $E,F,G$ in $\PP^3$ is given by the 
same formulas, but now with $u,a,b,c$ replaced by $v,e,f,g$. If we perform these
substitutions in the discriminant of $f(t) = {\rm det}(U + t V)$ then we obtain
the  NLS discriminant $\Delta_{K_2,(3,3)}$.
\end{proof}

\begin{remark}
We also consider $u = (4,2)$ for the graph $G = K_2$.
Here the Landau diagram is a pentagon attached to a triangle.
We see from (\ref{eq:gK2}) that the NLS genus is negative, namely
$g_{(4,2)}=-1$. This reflects the fact that the
general fiber of  the Landau map $\psi_u$ 
consists of two disjoint conics. Indeed, the
disjoint union of two curves of genus $0$ has genus $-1$.
\end{remark}

We note that the formulas in (\ref{eq:umatrix})
can also be used for the LS discriminant for the triangle 
graph $G = K_3$.  Namely, 
the Schubert problem for the concurrent component $V_{[3]}$
amounts to intersecting three quadratic surfaces in $\PP^3$.
This is the Cayley octad in \cite[Example 8.3]{HMPS1}.
Each of the three quadrics is determined by a triple of lines, and
the explicit formula for its  symmetric $4 \times 4$ matrix is given in
(\ref{eq:umatrix}). The LS discriminant for $V_{[3]}$ is
the discriminant of these three quadrics.
It is known that this discriminant has degree $16$
in the entries of each of the three
symmetric $4 \times 4$ matrices.
This explains the degree of $\Delta_{K_3, \mathbf{b},u}$ in
Theorem \ref{thm:16168}.

We conclude this section with a discussion of the
NLS geometry arising from three lines.

\begin{example}[$\ell=3$] \label{ex:NLSthreelines}
The chain $G = \{12,23\}$ has NLS genus $g_u = 5$ for $u=(3,3,3)$.
Indeed, the  three Schubert conditions at each vertex
define a $\PP^2 $ inside $\PP^5$. The general fiber of $\psi_u$ is
a curve in $(\PP^2)^3$ defined by five equations, of degrees
$(2,0,0),(0,2,0),(0,0,2)$ for the Grassmannians ${\rm Gr}(2,4)$,
and of degrees $(1,1,0)$ and $(0,1,1)$ for the line incidences.
 The general complete intersection in $(\PP^2)^3$ defined by these five degrees
is a smooth and irreducible curve of genus $5$. We see this by typing
\ {\tt multiGenera(3,\{\{1,2\},\{2,3\}\})}.

Next consider the triangle $G =K_3= \{12,13,23\}$,
and let $u=(3,3,2)$. Here the NLS genus is $9$.
The general fiber is the disjoint union of two irreducible curves
of genus~$5$, one on $V_{[3]}$ and one on $V_{[3]}^*$.
The genus of such a disjoint union equals $5+5-1=9$,
and this is the NLS genus. Note that $9$ is also the 
genus for a general complete intersection curve 
of degrees $(2,0,0),(0,2,0),(0,0,2),
(1,1,0),(1,0,1),(0,1,1)$ inside
the ambient space $\PP^2 \times \PP^2 \times \PP^3$.
\end{example}

\begin{remark}
Our terminology is organized by the dimension of the generic fiber of the Landau map: zero-dimensional fibers are \emph{leading}, one-dimensional fibers \emph{next-to-leading}, two-dimensional fibers \emph{next-to-next-to-leading}, and so on. This differs from some physics conventions, where \emph{leading Landau singularities} correspond to strata of highest codimension in the on-shell space, and \emph{subleading} ones to strata of lower codimension \cite{Dennen:2015cqa}. The two conventions agree when the minimal fiber dimension is $0$, but not in general. For instance, the highest-codimension stratum of the \emph{$3$-mass triangle} in Example \ref{ex:ell1_NSL} is \emph{next-to-leading} in our terminology, since its generic fiber is a curve, but \emph{leading} in the usual physics sense.
\end{remark}

\section{Recursive Formulas}
\label{sec:recursive}

We now present a recursive method for evaluating
LS discriminants and SLS resultants. 
In both cases, we start with a context-free example and then apply it to  the pentabox ($\ell = 2$). 

\begin{example}\label{ex:rec_quadr}
Consider the following system of two polynomials in two unknowns $x$ and $y$:
\begin{equation} \label{eq:fgsys}
\begin{matrix}
f (x)& = &a_2 x^2+a_1 x+a_0 \, , \qquad \qquad \qquad \qquad \qquad
\qquad \qquad \qquad \qquad \qquad \qquad \qquad \\
g(x,y) & =& b_{22}x^2y^2 + b_{21}x^2y + b_{12}xy^2 + b_{20}x^2+b_{11} x y+b_{02} y^2+b_{10}x+b_{01}y+b_{00} \, .
\end{matrix}
\end{equation}
In the recursive method, we solve $f(x) = 0$ with the formula
$ x^{(\pm)}  =   \frac{-a_1 \pm \sqrt{ a_1^2 - 4 a_0 a_2}}{2 a_2}$.
Plugging the roots into $g(x,y)$ and taking the product yields a rational function
of degree $0$ in~$a_0,a_1,a_2$:
\begin{equation}
\label{eq:EDelta}
\left(\mathrm{Disc}_y(g)\big|_{x=x^{(+)}} \right) \cdot \left(\mathrm{Disc}_y(g)\big|_{x=x^{(-)}} \right) 
\,=\, a_2^{-4} \cdot{\rm Res}_x (f,{\rm Disc}_y(g))
\,=\,  a_2^{-4} \cdot {\rm Disc}(f,g) \, .
\end{equation}
The discriminant of (\ref{eq:fgsys}) is an irreducible polynomial with $150$ terms of degree $(4,4)$:
\begin{equation}
\label{eq:disccontextfree}
{\rm Disc}(f,g)  = a_2^4b_{01}^4 - 8a_2^4b_{00}b_{01}^2b_{02} + 16a_2^4b_{00}^2b_{02}^2 + \cdots  - 16a_0^3a_1b_{10}b_{20}b_{22}^2 +  16a_{0}^4b_{20}^2b_{22}^2 \, .
\end{equation}
\end{example}

This example is relevant for the LS discriminant of
the pentabox ($\ell = 2$). 
Let ${\bf L} = (X,Y)$,  ${\bf M}^{(1)} = (A,B,C,D)$,
and ${\bf M}^{(2)} = (E,F,G)$, as in Proposition \ref{prop:eins}. In the following, we write intersections
and joins of linear subspaces in projective space $\PP^3$
 using the Grassmann-Cayley algebra~\cite[Section 2.4]{VRC}.
 Here we use the symbol $\star$ for intersection $\wedge$, and concatenation for join~$\vee$.
We set $p(x) = d_1 + x d_2$, where $D = d_1 d_2$, and $q(y) = g_1 + y g_2 $, where $G = g_1 g_2$. Note that $D = d_1d_2 = d_1 \vee d_2$ means that $D$ is the line spanned by the points $d_1$ and $d_2$. 
Then $X := (p(x) B) \star (p(x)C)$ satisfies $\langle BX \rangle =\langle CX \rangle= \langle DX \rangle= 0$,
and $ Y := (q(y) E)  \star (q(y) F)$ satisfies ${\langle EY \rangle=\langle FY \rangle=\langle GY \rangle= 0}$.
We have two more constraints, defined by $f(x) = \langle AX\rangle $ and $g(x,y) = \langle XY \rangle$.
These are now written precisely  as in~(\ref{eq:fgsys}), with
  coefficients 
 \begin{equation}
 \label{eq:remarkablycluster}
 a_2 =  \langle d_2 A | B | d_2 C \rangle \,,\,\,\,\,
 a_1 =  \langle d_1 A | B | d_2 C \rangle +  \langle d_2 A | B | d_1 C \rangle
\quad  {\rm and} \quad
 a_0 =  \langle d_1 A | B | d_1 C \rangle \, ,
\end{equation}
and similarly for the $b_{ij}$. The quantities $\langle abc | de | fgh \rangle
 := \langle abcd \rangle \langle efgh \rangle
   - \langle abce \rangle \langle dfgh \rangle$ are the \emph{chain polynomials}. 
   Their expansions appeared in (\ref{eq:umatrix}).
   We obtain the following~result.

\begin{proposition}
\label{prop:box-pentabox-recursion}
The LS discriminant of the box ($\ell=1$) and the pentabox $(\ell=2)$ are
    \begin{itemize}
        \item ${\rm Disc}_x(f) \,\,=\,\, a_1^2- 4 a_0 a_2 = \Delta_{K_1,(4)}(\mathbf{M}^{(1)})  $, \vspace{-0.1cm}
        \item $ {\rm Disc}(f,g) \,=\, \langle B C \rangle^2 \langle B D \rangle^2 \langle CD \rangle^2 \cdot  \Delta_{K_2,(4,3)}(\mathbf{M}) $.
This is exactly the formula in (\ref{eq:disccontextfree}).
    \end{itemize}
\end{proposition}
\begin{proof}
    The first statement is proven by symbolic computation in \texttt{Macaulay2} \cite{M2}, and the second was verified
    on specializations of ${\bf M}$.
    We now prove that the two sides have the same degree.
        Since the degree of $\Delta_{K_2, (4,3)}$ is $(4,4,4,4,4,4,4)$ in the Pl\"ucker coordinates of the external lines, the right hand side has degree $(4,8,8,8,4,4,4)$. On the other hand, the $a_i$ have degree $(1,1,1,1,0,0,0)$ while the $b_{ij}$ have degree $(0,1,1,1,1,1,1)$.  Since (\ref{eq:disccontextfree}) has degree $(4,4)$ in the $a_i$ and $b_{ij}$, we see that the degree
        of the left hand side in the Pl\"ucker coordinates equals
         $4(1,1,1,1,0,0,0) + 4(0,1,1,1,1,1,1) = (4,8,8,8,4,4,4)$. 
    \end{proof}

This result enables us to  compute $\Delta_{K_2,(4,3)}$ recursively. From~\eqref{eq:EDelta}, we obtain
$$    \left( \mathrm{Disc}_y(g)\big|_{x=x^{(+)}} \right) \cdot \left(\mathrm{Disc}_y(g)\big|_{x=x^{(-)}} \right) 
    \,\,= \,\, \Delta_{K_1,(4)}(X^{(+)},\mathbf{M}^{(1)}) \cdot \Delta_{K_1,(4)}(X^{(-)},\mathbf{M}^{(1)}) \, ,
    $$
\begin{equation}
\label{eq:Xpm}
{\rm where}   \ X^{(\pm)}(\mathbf{M}^{(1)})   \,=\,  \left(p(x^{(\pm)}) B \right)  \star \left(p(x^{(\pm)})C \right) 
    \ {\rm with} \
 x^{(\pm)}  \,=\, \begin{small}
       \frac{-a_1 \pm \sqrt{ a_1^2 - 4 a_0 a_2}}{2 a_2} \, . \end{small}
\end{equation}
We renormalize the solutions as $L^{\pm}(\mathbf{M}^{(1)})=\varepsilon(\mathbf{M}^{(1)}) \cdot X^{(\pm)}(\mathbf{M}^{(1)})$ by 
\begin{equation}\label{eq:4box_extr_fact}
    \varepsilon(\mathbf{M}^{(1)}) \,\,=\,\,
    \begin{small} \frac{2 a_2}{\sqrt{\langle B C \rangle \langle B D \rangle \langle CD \rangle}} \, , \qquad \text{with $a_2$ as in \eqref{eq:remarkablycluster}} \, . \end{small}
\end{equation}
This ensures that the solutions $L^{\pm}(\mathbf{M}^{(1)})$ have degree $(1,1,1,1)$ in ${\bf M}^{(1)}= (A,B,C,D)$. 
Combining this with~\eqref{eq:EDelta}, we obtain the following recursive formula
for the LS discriminant:
\begin{equation}\label{eq:pentabox_rec}
    \Delta_{K_2,(4,3)}(\mathbf{M}) \,\,= \,\,  \prod_{\sigma \in \{\pm\}} \Delta_{K_1,(4)}(\varphi_{\pm}(\mathbf{M})) \,.
\end{equation}
Here the right hand side uses the two \textit{four-mass box substitution maps}
\begin{equation}\label{eq:four_box_subs_maps}
    \varphi_{\pm} : {\rm Gr}(2,4)^{4+3} \rightarrow {\rm Gr}(2,4) \times {\rm Gr}(2,4)^{3} \, , \ (M_1^{(1)},\dots,M_4^{(1)},\mathbf{M}^{(2)}) \mapsto (L^{\pm}(\mathbf{M}^{(1)}),\mathbf{M}^{(2)}) \, .
\end{equation}
One checks that both sides of~\eqref{eq:pentabox_rec} have indeed degree $(4,4,\dots,4)$ in the external lines. 

\smallskip

We now vastly generalize Proposition \ref{prop:box-pentabox-recursion}.
We partition the vertex set $[\ell]$ of $G$ into two non-empty subsets $V_1$ and $V_2$.  Let $G_1$ and $G_2$ be the induced subgraphs on $V_1$ and~$V_2$, respectively. Then, $G=G_1 \cup G_2 \cup E$, where $E$ is the set of edges
connecting $V_1$ and $V_2$.
Let $u=(u_1,u_2)$ be a term in the multidegree of $V_G$, partitioned according to $V_1$ and $V_2$. 
Set $d:=d_1+d_2 $ with $d_i = |u_i|$. Pick any external degeneration $H_u =H_{u_1} \cup H_{u_2}$ and denote the corresponding Landau diagrams by $\mathcal{L} = G_u \cup H_u$, $\mathcal{L}_1 = G_{1,u_1} \cup H_{u_1}$ and $\mathcal{L}_2 = G_{2,u_2} \cup H_{u_2} \cup E$, and their associated Landau maps by $\psi$ and $\psi_i$. We say that $\mathcal{L}$ is \textit{reducible with respect to} $\mathcal{L}_1$ if $d_1 = {\rm dim}(V_{G_1})=|u_1|$. 
Geometric considerations imply the~following~result.

\begin{proposition}\label{prop:rec_fibers}
Let $\mathcal{L}$ be a Landau diagram that is reducible with respect to $\mathcal{L}_1$. Then, 
\begin{equation}\label{eq:fiber}
    \psi^{-1}({\bf M})\quad =\,\bigcup_{{\bf L}^{(1)} \in 
    \psi_1^{-1}({\bf M}^{(1)})} 
    \left\{({\bf L}^{(1)}, \mathbf{L}^{(2)}) : \mathbf{L}^{(2)} \in \psi_2^{-1}(\varphi_{{\bf L}^{(1)}}({\bf M})) \right\} .
\end{equation}
 The substitution map  in (\ref{eq:fiber}) is the algebraic function
 given by evaluating at a solution:
  \begin{equation}\label{eq:subs_map}
     \varphi_{{\bf L}^{(1)}}: {\rm Gr}(2,4)^d \,\rightarrow\, {\rm Gr}(2,4)^{|E|+d_2} \, , \quad {\bf M}=({\bf M}^{(1)},{\bf M}^{(2)}) \,\mapsto \,({\bf L}^{(1)},{\bf M}^{(2)}) \, .
 \end{equation}
\end{proposition}

 \begin{figure}[htbp]
    \centering
\includegraphics[width=0.95\textwidth]{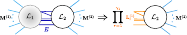}
\caption{A Landau diagram which is reducible with respect to $\mathcal{L}_1 = G_{1,u_1} \cup H_{u_1}$.}\label{fig:recursion}
\end{figure}

In the following we write $\psi_1^{-1}(\mathbf{M}^{(1)}) = \big\{\mathbf{L}_1^{(1)} , \dots, \mathbf{L}_{\gamma_1}^{(1)} \big\}$ with $\mathbf{L}_r^{(1)} = \mathbf{L}_r^{(1)}(\mathbf{M}^{(1)})$, where $\gamma_1$ is the LS degree of $\mathcal{L}_1$, and $\varphi_r = \varphi_{\mathbf{L}^{(1)}_r}$ for the corresponding substitution maps for $r=1,\dots,\gamma_1$.
\begin{theorem}\label{thm:discr_fact}
Let $\mathcal{L}$ be a Landau diagram that is reducible with respect to $\mathcal{L}_1$. Then, 
\begin{equation}\label{eq:discr_fact}
        \Delta_{\mathcal{L}}(\mathbf{M}) \, \, = \, \,  \mathcal{E}(\mathbf{M}^{(1)}) \,\,\cdot \!\!\ \prod_{r=1}^{\gamma_1} \ \!\!
        \Delta_{\mathcal{L}_2}(\varphi_r(\mathbf{M}))
          \,  \, ,
\end{equation}
where the extraneous factor $\mathcal{E}$ is a rational function in the Plücker coordinates of $\mathbf{M}^{(1)}$.
\end{theorem}

\begin{proof}
    The product~(\ref{eq:discr_fact}) is a symmetric function in the roots of the zero-dimensional polynomial system 
    given by $\mathcal{L}_1$. It is  a rational function in the Plücker coordinates in $\mathbf{M}$ which has degree zero in the those of $\mathbf{M}^{(1)}$. By Proposition~\ref{prop:rec_fibers},~the function (\ref{eq:discr_fact}) vanishes when two points in $\psi^{-1}(\mathbf{M})$ coincide. Hence, $\Delta_{\mathcal{L}_1}$ divides the numerator of the product in~(\ref{eq:discr_fact}). The extraneous factor $\mathcal{E}(\mathbf{M}^{(1)})$ arises from the ambiguity in defining the points $\mathbf{L}^{(1)}_r$ in $\psi_1^{-1}(\mathbf{M}^{(1)})$, and consequently the substitution maps $\varphi_r$. It can be set to one by redefinition of~\eqref{eq:subs_map}.
\end{proof}

It follows from~\eqref{eq:pentabox_rec} that the extraneous factor $\mathcal{E}(\mathbf{M}^{(1)})$
in~\eqref{eq:discr_fact}, which is produced by the four-mass box recursion described
in \eqref{eq:Xpm} and \eqref{eq:four_box_subs_maps}, is always equal to one.

\begin{corollary}\label{cor:rec_discr_4box}
In Theorem \ref{thm:discr_fact}, let $\mathcal{L}_1$ be the box and
        $\varphi_{\pm}$ as in~(\ref{eq:four_box_subs_maps}).
    Then, 
    $$
      \Delta_{\mathcal{L}}(\mathbf{M}) \, \,  = \, \,  \prod_{\sigma \in \{\pm \}} \Delta_{\mathcal{L}_2}(\varphi_{\pm}(\mathbf{M})) 
       \, .
$$
\end{corollary}

 Corollary~\ref{cor:rec_discr_4box} is applicable whenever $u_i = 4$ for some $i$.
 This holds for any $u$
 in the multidegree of $V_G$ if $G$ is a tree. It is also true when $G$ is a cycle and $u = (4,3,\dots,3,2,3,\dots,3)$.


\smallskip 
We now turn to the case of SLS resultants,
and again start with a context-free example.

\begin{example}
We augment $f(x)$ and $g(x,y)$ from (\ref{eq:fgsys}) by 
$\,h (y)= c_2 y^2+c_1 y+c_0$.
The resultant  of these equations is an irreducible  polynomial with $1340$ terms of degree $(4,4,4)$:
\begin{equation}
\begin{small}
    {\rm Res}(f,g,h)\,\, = \,\,
a_2^4b_{02}^4c_0^4
- 2a_1a_2^3b_{02}^3b_{12}c_0^4
+ a_1^2a_2^2b_{02}^2b_{12}^2c_0^4
+ \cdots
-2a_0^3a_1b_{10}b_{20}^3c_2^4
+a_0^4b_{20}^4c_2^4 \, .
\end{small}
\end{equation}
We apply a recursive method as in Example~\ref{ex:rec_quadr}, solving for the roots $x^{(\pm)}$ of $f$, and we find
\begin{equation}
\label{eq:Res}
\left(\mathrm{Res}_y(g,h)\big|_{x=x^{(+)}} \right) \cdot \left(\mathrm{Res}_y(g,h)\big|_{x=x^{(-)}} \right) 
\,=\, a_2^{-4} \cdot {\rm Res}_y( {\rm Res}_x(f,g),h) \,=\,
a_2^{-4} \cdot {\rm Res}(f,g,h) \, .
\end{equation}
\end{example}

We adapt (\ref{eq:Res}) to our setting, by introducing parameters $x,y$ for
$X,Y$ as  above. We set $f(x) = \langle AX \rangle$, $g(x,y) = \langle XY \rangle$ and $h(y) = \langle HY \rangle$.
Similarly to Proposition~\ref{prop:box-pentabox-recursion}, we obtain:
\begin{proposition}
\label{prop:penta-dpenta-recursion}
The pentagon ($\ell=1$) and double pentagon $(\ell=2)$ have the SLS resultants
    \begin{itemize}
        \item ${\rm Res}_x(f, \langle EX\rangle ) =  \langle BC \rangle \langle CD \rangle \langle BD \rangle \cdot R_{K_1,(5)}(A,B,C,D,E)  \, ,$ \vspace{-0.1cm}
        \item ${\rm Res}(f,g,h) =  \langle B C \rangle^2 \langle B D \rangle^2 \langle CD \rangle^2  \langle EF \rangle^2 \langle EG \rangle^2 \langle FG \rangle^2 \cdot R_{K_2,(4,4)}(\mathbf{M}) \, .$ 
    \end{itemize}
\end{proposition}
Just like for the LS discriminant, we again obtain a recursive formula:
\begin{equation}\label{eq:dob_penta_rec}
    R_{K_2,(4,4)}(\mathbf{M}) \,\,=\,\,  \prod_{\sigma \in \{\pm\}} R_{K_1,(5)}(\varphi_{\pm}(\mathbf{M})) \, .
\end{equation}
We extend the recursion in Theorem~\ref{thm:discr_fact} 
from LS discriminants to SLS resultants, and obtain:

\begin{theorem}\label{thm:res_fact}
If a super-leading Landau diagram $\mathcal{L}$ is reducible with respect to $\mathcal{L}_1$, then
\begin{equation}\label{eq:res_fact}
       R_{\mathcal{L}}(\mathbf{M}) \, \, =  \, \, 
       \mathcal{E}(\mathbf{M}^{(1)}) \cdot \prod_{r=1}^{\gamma_1} R_{\mathcal{L}_2}(\varphi_r(\mathbf{M}))   \, ,
\end{equation}
where the extraneous factor $\mathcal{E}$ is a rational function in the Plücker coordinates of $\mathbf{M}^{(1)}$.
\end{theorem}

\begin{corollary}\label{cor:rec_res_4box}
    In the set-up of Corollary~\ref{cor:rec_discr_4box}, 
    with $\mathcal{L}_1$ the box and $\varphi_{\pm}$ as in~(\ref{eq:four_box_subs_maps}), we have 
    \begin{equation}
        R_{\mathcal{L}}(\mathbf{M}) \, \, = \, \, \prod_{\sigma \in \{\pm \}} R_{\mathcal{L}_2}(\varphi_{\pm}(\mathbf{M}))   \, .
    \end{equation}
\end{corollary}

We conclude by discussing a scenario when the Landau diagram is \emph{not} reducible.

\begin{proposition}\label{prop:rec_cycles} Fix the Landau diagram $\mathcal{L} \!=\! G_u$ where $G$ is the $\ell$-cycle
    and $u = (3,\ldots,3)$. The LS discriminant $\Delta_{\mathcal{L}} $ is the discriminant of a univariate polynomial $f_\ell(x)$ of degree $2^{\ell+1}$.
\end{proposition}

\begin{proof}
Let $M^{(i)}_1,M^{(i)}_2,M^{(i)}_3$ be 
the external lines at vertex $i \in [\ell]$.
We parametrize the line at vertex $\ell$ as $L_\ell(x) = (p(x)M_2^{(\ell)}) \star (p(x)M_3^{(\ell)})$, 
where $p(x) = u + x v$ with $M^{(\ell)}_1 = uv$. 
        Consider the SLS resultant of a path $\widetilde{G}$ on $[\ell-1]$ with $\tilde{u} = (4,3,\dots,3,4)$, with external data $M^{(i)}_j$ at $i\in [\ell-1]$, and an extra line $L_{\ell}(x)$ at vertices $1$ and $\ell-1$.  We view this SLS resultant as a polynomial
         $f_\ell(x) = R_{\widetilde{G},\tilde{u}}\bigl(L_\ell(x),\mathbf{M}^{(1)},  \ldots, \mathbf{M}^{(\ell-1)},L_{\ell}(x) \bigr)$
           in one variable $x$. The degree of $f_\ell(x)$  equals $2 \cdot 2^\ell =  2^{\ell + 1}$. 
           The
          discriminant of $f_\ell(x)$ in $x$ is equal to the full LS discriminant $\Delta_{\mathcal{L}} $, up to possibly some extraneous factors. We now describe the case of $\ell = 3$.
\end{proof}

\begin{example} 
    Let $G=K_3$ and $u=(3,3,3)$. 
    Here, $f_3(x)$ factors into two irreducible polynomials, each of degree eight.
    Its discriminant has the three factors in Theorem  \ref{thm:16168},
    one  for the concurrent component, one for the coplanar component, and one for the mixed~part.  
        
    We explain how to obtain $f_3(x)$ from our recursive approach. We follow the notation as in the proof of Proposition~\ref{prop:rec_cycles}, and we use~\eqref{eq:pentabox_rec}.
     We find that $f_3(x)$ equals the resultant
    \begin{equation*}
        R_{K_2,(4,4)}\left(L_3(x),\mathbf{M}^{(1)},\mathbf{M}^{(2)} ,L_3(x) \right) \,\,= \,\, \!\! \prod_{\sigma \in \{\pm\}} 
        \! R_{K_1,(5)}\left(L^{(\sigma)}(x), \mathbf{M}^{(1)} , L_3(x) \right) \, ,
    \end{equation*}
    where $L^{(\sigma)}(x) = L^{(\sigma)} \left(L_3(x),\mathbf{M}^{(2)} \right)$ as below~\eqref{eq:Xpm}.
     Each factor is a Gram determinant~\eqref{eq:resultant_5_lines}.
      The product of the two $5 \times 5$ Gram matrices, specialized to this setting, is a rational function of $\mathbf{M}$. 
      Hence $f_3(x) $  is the determinant of a 
      $5 \times 5$ matrix with polynomial entries in $\mathbf{M}$ and~$x$.
\end{example}

\section{Rationality}
\label{sec:rationality}

Fix a graph $G$ on $[\ell]$ and $u \in \NN^\ell$ with $|u| = d = {\rm dim}(V_G)$.
We assume that $G$ is outerplanar.
The Landau diagram $G_u$ is a planar graph with
external vertices labeled cyclically by $[d]$.
The fiber $\psi_u^{-1}({\bf M})$ of the Landau map $\psi_u$ is irreducible over the rational function field $\QQ({\bf M})$. It consists of $\gamma_u$ points that are algebraic functions 
of $\mathbf{M}$. It is generally impossible to express these $\gamma_u$ points
in terms of radicals. Hence, writing down the LS discriminant is non-trivial.

We now degenerate the external lines
${\bf M}$ so that the fiber is given by rational functions.
Following \cite[Section 8]{HMPS1}, we fix a graph $H_u$ on $[d]$, and we
require ${\bf M}$  to lie in $V_{H_u} \subseteq {\rm Gr}(2,4)^d$. We denote the resulting Landau diagram by $\mathcal{L}=G_u \cup H_u$.
The induced Landau map $\psi_\mathcal{L} : V_{G \cup H_u} \rightarrow V_{H_u}$
projects configurations of $\ell + d$ lines onto  configurations of $d$ lines. 
We seek a sparse graph $H_u$ which yields
LS degree many rational formulas for the $\ell$ internal lines ${\bf L}$ in terms of the $d$ external lines ${\bf M}$. If this
rationality property holds then we say that  $\mathcal{L}$ is a \textit{rational} Landau diagram.
The physical meaning of this degeneration is that when two external lines intersect, the momentum flowing in the corresponding vertex of the dual graph is on-shell (massless),
in contrast to the non-intersecting case which is off-shell (massive).

We assume that $u_i \geq 2$ for all $i \in [\ell]$. Let $H^\triangle_u$ be the graph with one edge among the vertices in $[d]$ that are neighbours of any given $i \in [\ell]$.
Thus $H^\triangle_u$ is a graph on $[d]$ with $\ell$ edges.
The graph $\mathcal{L}^\triangle := G_{u}^{\triangle} := G_u \cup H^\triangle_u$ has one pendant triangle attached to each
original vertex $i \in [\ell]$ and $u_i - 2$ pendant edges at each vertex $i$. 
Each triangle gives rise to three constraints on the incidence variety $V_{G}^\triangle $ of the graph $G^\triangle$.
This denotes the graph we obtain by deleting the pendant edges  from $G_u^\triangle$.
 We write the incidences of this graph $G^\triangle$ as
\begin{equation}
\label{eq:trianglerel}
\langle L_i M_1^{(i)} \rangle \,=\, \langle L_i M_2^{(i)} \rangle \,=\, \langle M_1^{(i)} M_2^{(i)} \rangle \,\,=\,\,0 \, 
\qquad {\rm for} \,\, i \in [\ell] \, .
\end{equation}
Here $M_1^{(i)}$ and $ M_2^{(i)}$ are the external lines that are connected by an edge in $H^\triangle_u$.

The equations (\ref{eq:trianglerel}) imply that the  lines $L_i\,, M_1^{(i)}, M_2^{(i)}$ are either concurrent or coplanar.
We assign a bicoloring $\sigma_{\rm ext}$ to the pendant triangles  and we incorporate this into the bicoloring $\sigma=(\sigma_{\rm int},\sigma_{\rm ext})$, where $\sigma_{\rm int}$ colors the triangles of $G$. Then, $\sigma$ labels one irreducible component $V^\triangle_{G,\sigma}$ of the variety $V_{G}^\triangle$.
A priori, the irreducible component $V^\triangle_{G,\sigma}$ is a subvariety of $(\PP^5)^{3\ell}$.

We consider the Landau map $\psi_{\mathcal{L}}^\triangle : V_{G}^\triangle \rightarrow (\PP^5)^{2\ell}$
which takes  a configuration of $3 \ell$ lines $({\bf L}, {\bf M})$ to
the tuple of $\ell$ incident pairs ${\bf M} = (M_1^{(i)}, M_2^{(i)} : i \in [\ell] )$.
We write $\psi_{\mathcal{L},\sigma}^\triangle$ for the restriction of the Landau map $\psi_{\mathcal{L}}^\triangle$
to the irreducible component $V^\triangle_{G,\sigma}$ of $V_{G}^\triangle$.

\begin{proposition}  The generic fiber of $\psi_{\mathcal{L},\sigma}^\triangle$ is
an irreducible variety of dimension $d-2 \ell$. This variety is naturally embedded 
in the product $\PP^2 \times \PP^2 \times \cdots \times \PP^2$ of $\ell$ projective planes.
\end{proposition}

\begin{proof}
The bicoloring $\sigma$ specifies for each $i \in [\ell]$ whether the triple in
(\ref{eq:trianglerel}) is concurrent or coplanar. Since $M_1^{(i)}$ and $M_2^{(i)}$ are
are given, this means that $L_i$ either passes through a given point, or it lies in a given plane.
In each case, the set of such lines $L_i$ forms a  plane $\PP^2$.
\end{proof}

\begin{example}[$\ell = 3$]
Fix the triangle $G = K_3$. Then $G^\triangle$ is an outerplanar graph
with $3 \ell = 9$ vertices and $12$ edges. These form $4$ triangles,
so there are $2^4 = 16$ bicolorings $\sigma$. Each of them specifies 
an irreducible component  $V_{G,\sigma}^\triangle$. For instance, if $\sigma$
assigns black to all $4$ triangles, then the points in $V_{G,\sigma}^\triangle$ 
are trees of $9$ lines in $\PP^3$ that meet at $4$ intersection points.
For fixed external lines $M_j^{(i)}$, the fiber of $\psi_{\mathcal{L},\sigma}^\triangle$ is a threefold of multidegree $t_1t_2t_3$ in~$(\PP^2)^3$.
\end{example}

We use the notation $W_{\sigma}^\triangle$ for the generic fiber of
the Landau map $\psi_{\mathcal{L},\sigma}^\triangle$.
This is a variety in the new ambient space $(\PP^2)^\ell$ whose coordinates are also denoted
${\bf L} = (L_1,L_2,\ldots,L_\ell)$. 
These are now vectors of length three,
so there are no more Pl\"ucker quadrics to be considered.
Indeed, the irreducible components of the $\ell$ systems in
(\ref{eq:trianglerel}) specify the ambient space $(\PP^2)^\ell$.

Let $\sigma = (\sigma_{\rm int},\sigma_{{\rm ext}})$ where
$\sigma_{\rm int}$ colors the internal triangles in $G$, and $\sigma_{{\rm ext}}$ colors the $\ell$ external triangles.
Let $W_{\sigma_{\rm ext}}^\triangle$ denote the variety defined by
the bilinear equation $\langle L_i L_j \rangle$ for the edges $ij \in G$.
This is reducible if $G$ has triangles, and all its irreducible
components $W_{(\sigma_{\rm ext},\sigma_{\rm int})}^\triangle$ have the same
codimension $|G|$ in $(\PP^2)^\ell$.  The multidegrees
of these varieties satisfy
\begin{equation}
\label{eq:WDelta}
 \qquad [W_{\sigma_{\rm ext}}^\triangle] \,\, = \,\, \sum_{\sigma_{\rm int}} 
\, \big[W_{(\sigma_{\rm ext},\sigma_{\rm int})}^\triangle \big]  \quad \in \quad A^*\bigl(\, (\PP^2)^\ell \,\bigr) \,\, = \,\,
\ZZ[t_1,\ldots,t_\ell]/\langle t_1^3,\ldots,t_\ell^3 \rangle \, .
\end{equation}
We shall be interested in scenarios when all coefficients in the multidegree $[W_{\sigma}^\triangle]$ are $0$ or $1$.

To study this, we return to the variety $V_{G}^\triangle$ in its
original definition. Recall that
$u_i \geq 2$ for all $i \in [\ell]$, and hence
 $d = \sum_{i=1}^\ell u_i \geq 2 \ell$.
The graph $H^\triangle_u$ has $d$ vertices but only $\ell$ edges.
The Landau map $\psi_\mathcal{L}^\triangle$ of the Landau diagram $\mathcal{L}^\triangle=G_u^\triangle \cup H_u^\triangle$ takes $V_{G}^\triangle$ onto $V_{H}^\triangle$, and its generic fiber consists of finitely many points. These points are distributed
over the irreducible components $V_{G,\sigma}^\triangle$.
We are interested in the case when each component $V_{G,\sigma}^\triangle$
intersects the fiber of $\psi_\mathcal{L}^\triangle$ in precisely one point.
If this happens then $\mathcal{L}^\triangle$ is a rational Landau diagram. We say that $G^\triangle$ is rational if the Landau diagram $G_u^\triangle \cup H_u^\triangle$ is rational for every~$u$. 

\begin{proposition}
\label{prop:mld-1-implies-rational}
Suppose that, for each summand $[W^\triangle_{\sigma}]$ which appears in the multidegree (\ref{eq:WDelta}),
the monomials have only coefficients $0$ and $1$.
Then $G^\triangle$ is a rational Landau diagram.
\end{proposition}

\begin{proof}
The fiber of $\psi_{\mathcal{L}}^\triangle$ is specified by a tuple ${\bf M}$
where each pair of lines $M_1^{(i)},M_2^{(i)}$ spans a plane $P_i$ in $\PP^3$ and intersects
in a point $p_i \in \PP^3$. Consider the restriction $\psi_{\mathcal{L},\sigma}^\triangle$ to one irreducible component $V_{G,\sigma}^\triangle$.
The bicoloring $\sigma$ determines whether $L_i$ lies in that plane or whether it passes through that point.
In either case, the $L_i$ is represented by a point in $\PP^2$. The number of solutions
${\bf L} \in (\PP^2)^\ell$ to these constraints for $\sigma$ is a coefficient of $[W_{\sigma}^\triangle]$.
Our hypothesis implies that there is only one solution, which is
a rational expression in ${\bf M}$.
\end{proof}

We now discuss the rational graph for our running example, the triangle $G=K_3$. The simpler formulas for the box in the rational setting can be found in Example \ref{ex:fact_cluster_4mb_disc}. 
We will write the rational fibers using the Grassmann-Cayley algebra,
with notation as in Section \ref{sec:recursive}.

\begin{example}[$\ell=3$]\label{ex:rat_trian}
Fix $G = K_3$ and $u=(3,3,3)$. The graph $G^\triangle_u$ has $12$ vertices and $15$ edges. 
Let $\sigma \in \{{\rm {\bf b}lack},{\rm {\bf w}hite}\}^4$
where $\sigma_1$ is the color of the internal triangle.
Let $M_i = M^{(i)}_3$. We also denote by $p_i$ the intersection point of $M^{(i)}_1$ with $M^{(i)}_2$ and by $P_i$ the plane they span.
There are $8$ distinct bicolorings of $G^\triangle$, up to symmetry.
The $8$ points in the fiber are:
\smallskip \begin{small}
\[
\begin{aligned}
\textbf{bbbb}: \quad 
{\bf L} &= \bigl(p_1  q,\, p_2  q,\, p_3  q\bigr) \, ,\quad 
q = (p_1  M_1) \star (p_2  M_2) \star (p_3  M_3) \, , \\
\textbf{wbbb}: \quad 
{\bf L} &= \bigl(p_1  q_1,\, p_2  q_2,\, p_3  q_3\bigr) \, ,\quad 
q_i = M_i \star (p_1  p_2  p_3)\, , \\
\textbf{bbbw}: \quad 
{\bf L} &= \bigl(p_1  q,\, p_2  q,\, (M_3 \star P_3)  q\bigr)\, ,\quad 
q = (p_1  M_1) \star (p_2  M_2) \star P_3 \, , \\
\textbf{wbbw}: \quad 
{\bf L} &= \bigl(p_1  (M_1 \star Q) \, ,\, p_2  (M_2 \star Q),\, P_3 \star Q\bigr) \, ,\quad 
Q = p_1  p_2  (M_3 \star P_3) \, , \\
\textbf{bbww}: \quad 
{\bf L} &= \bigl(p_1  q,\, (M_2 \star P_2)  q,\, (M_3 \star P_3)  q\bigr) \, ,\quad 
q = (p_1  M_1) \star P_2 \star P_3 \, , \\
 \end{aligned} \] \[ \begin{aligned}
\textbf{wbww}: \quad 
{\bf L} &= \bigl(p_1  (M_1 \star Q),\, P_2 \star Q,\, P_3 \star Q\bigr) \, ,\quad 
Q = p_1  (M_2 \star P_2)  (M_3 \star P_3) \, , \\
\textbf{bwww}: \quad 
{\bf L} &= \bigl((M_1 \star P_1)q,\, (M_2 \star P_2)  q,\, (M_3 \star P_3)  q\bigr) \, ,\quad 
q = P_1 \star P_2 \star P_3 \, , \\
\textbf{wwww}: \quad 
{\bf L} &= \bigl(P_1 \star Q,\, P_2 \star Q,\, P_3 \star Q\bigr) \, ,\quad 
Q = (M_1 \star P_1)  (M_2 \star P_2)  (M_3 \star P_3) \, .
\end{aligned}
\] \end{small}
In each case, ${\bf L} = (L_1,L_2,L_3)$ is given by a rational expression in terms of the $9$ lines in ${\bf M}$.
\end{example}

The next example shows that the degeneration $H^\triangle_u$ does not always make $\mathcal{L}^\triangle$ rational.

\begin{example}[Cycles]\label{eg:cycles}
    Let $\mathcal{L}^\triangle$ with $G = C_\ell$ the cycle on $\ell \geq 4$ vertices and $u=(3,3,\dots,3)$. The LS degree
     equals  $2^{\ell + 1}$. We claim that the monomial $t_1^2 \cdots t_\ell^2$ has coefficient $2$ in
      $[W_{\sigma}^\triangle]$.
       The idea is the following. 
       To compute  the fiber of $\psi_\mathcal{L}^\triangle$ for a given bicoloring $\sigma \in \{\mathbf{b},\mathbf{w}\}^{\ell}$, we open the cycle at the vertex $\ell$,
       as in Proposition  \ref{prop:rec_cycles}.
        We solve the equations associated to $\mathcal{L}^\triangle$ along the path from $1$ to $\ell-1$, and we are left with one degree of freedom. By closing the system at vertex $\ell$ we are left to solve the resultant of five lines, 
        which by equation (\ref{eq:resultant_5_lines}) yields a quadratic 
        equation in one variable. Hence, the Plücker coordinates 
        of the lines in the fiber of $\psi_\mathcal{L}^\triangle$ belong to a 
        degree-two field extension of $\mathbb{Q}(\mathbf{M})$. In particular, $\mathcal{L}^\triangle$ is not~rational.
\end{example}

We next show that trivalent trees give rise to rational fibers.

\begin{theorem} \label{thm:trivalenttree}
Let $G$ be a trivalent tree on $[\ell]$. For each of the $2^\ell$ bicolorings
$\sigma$, the variety $W_{\sigma}^\triangle$ is irreducible and its multidegree has
coefficients $0$ or $1$. Hence $G^\triangle$ is a rational graph.
\end{theorem}

\begin{proof}
Since $V_G$ is a complete intersection by \cite[Theorem 5.4]{HMPS1}, so is 
 $V_{G}^\triangle$ in $(\PP^5)^{3\ell}$.
This implies that the generic fiber of $\psi_\mathcal{L}^\triangle$ is a complete intersection in $(\PP^5)^\ell$.
Its multidegree equals
\begin{equation}
\label{eq:treemultidegree}
2^\ell \prod_{i = 1}^\ell t_i^{1+u_i} \prod_{ij \in G}(t_i + t_j) \, .
\end{equation}
There are $2^\ell$ possible bicolorings  $\sigma$ of the pendant triangles, each of which gives an irreducible component
$V_{G,\sigma}^\triangle$.  The only possible monomials which can appear in the multidegree of $V_{G,\sigma}^\triangle$ are those appearing in the multidegree of $V_G^\triangle$. If we divide these monomials by $t_1^3 t_2^3 \cdots t_\ell^3$
then we get the monomials that can occur in the multidegree of 
$W_{\sigma}^\triangle$, which lives in $(\PP^2)^\ell$.

 We now prove by induction that each of these monomials appear with coefficient $1$. If $G$ is the graph with $1$ 
 vertex and no edges, then $G^\triangle$ has $4$ pendant edges and one triangle, which is either black or white. In either case,
 $W_{\sigma}^\triangle = \PP^2$, which  has multidegree $1$.
 If $\ell \geq 2$ then we consider any vertex $i$ of $G$ which has at least $4$ neighbors in $G^\triangle$.
Since the bicoloring $\sigma$ is fixed, the Schubert problem has a unique solution $L_i$ which depends
rationally on ${\bf M}$.
By propagating the solution through the tree, and we find that there is a unique
 solution on each component $V_{G,\sigma}^\triangle$ to the Schubert problem for any monomial in the multidegree
 (\ref{eq:treemultidegree}). 
 \end{proof}

\begin{remark}
    The rational degeneration in Theorem~\ref{thm:trivalenttree} is minimal in the following sense. If $G$ is a trivalent tree and $u$ is such that $u_i \geq 2$ for  all $i$, then $G_u \cup H$ is rational if and only if $H=H^\triangle_u$. Moreover, the assumption on $G$ being trivalent is necessary. To see this, let $G$
         be the one-vertex graph and $u = (0,4,4,4,4)$. Then the Landau diagram $G^\triangle_u$ is not rational.
\end{remark}

\begin{theorem}\label{thm:rat_tr_of_l_gon}
Let $G$ be a triangulation of the $\ell$-gon whose dual graph is a path with pendant edges. Then $G^\triangle$ is rational. However, for other triangulations,
   $G^\triangle$ need not be rational.
\end{theorem}

\begin{proof} The proof of the rationality statement
 is analogous to that of Theorem~\ref{thm:trivalenttree}.
 For the last claim, let $\ell = 6$ and
 $G =\{12,13,23,34,35,45,56,15,16\}$ with $u = (2,3,2,3,2,3)$. We consider the irreducible component $V_{G,\sigma}^\triangle$ of $V_{G,u}^\triangle$ where the triangles $123, 156, 345$ are colored black, and $135$ is white. 
 A computation shows that  the LS degree of this component is $58$.
 Among the solutions,
  $32$ are rational, while the other $16$ require a quadratic extension.
\end{proof}

We now briefly share an example with a remarkable combinatorial structure. Let $T_\ell$ be the triangulation of the cyclically labeled $\ell$-gon with diagonals $(2,\ell), (\ell,3),(3,\ell-1), \ldots, (v_\ell-1,v_\ell+1)$, where $v_\ell=\lceil \frac{\ell}{2}+1 \rceil$. Let $F_\ell = F_{\ell-1}+F_{\ell-2}$ with $F_0 = 0$ and $F_1=1$ be the \emph{Fibonacci numbers}. The following result 
is derived from Theorem~\ref{thm:rat_tr_of_l_gon}, using an application of~\eqref{eq:LS_CI}. 

\begin{corollary}\label{cor:fibonacci}
    Let $u \in \mathbb{N}^\ell$ be such that $u_i = 3$ for $i \in \{1,2,v_\ell \}$ and $u_i = 2$ otherwise. 
    The LS degree of $ (T_\ell)_u $ is equal to $2^\ell F_{\ell+1}$. Moreover, 
    the Landau diagram $(T_\ell)_u \cup H^\triangle_u$ is rational.
\end{corollary}

Rational Landau diagrams are important because the factors of their
  LS discriminants are expected to be
    cluster variables, as we shall see in Section~\ref{sec:cluster}.
    This is related to the  positivity structures in Section \ref{sec:positivity}.
     The following example illustrates this in the simplest~scenario.

\begin{example}[$\ell=1$]\label{ex:fact_cluster_4mb_disc}
The  LS discriminant $\Delta_{K_1,(4)}$ is the Gram determinant of the four lines $M_1, M_2, M_3, M_4$. If two of these lines are incident, so $\langle M_1 M_2 \rangle = 0$, then this simplifies~to
\begin{equation}\label{eqn:rational-box}
 {\rm det} \begin{small} \begin{bmatrix} 
0 & 0 &  M_1 M_3 &  M_1 M_4  \\
 0 & 0 & M_2 M_3 & M_2 M_4 \\
  M_1 M_3 &  M_2 M_3 & 0 & M_3 M_4  \\
  M_1 M_4 & M_2 M_4& M_3 M_4 & 0\\
\end{bmatrix}  \end{small}
 \,\,=\,\, \bigl(\,\langle M_2M_3 \rangle \langle M_1M_4 \rangle - \langle M_1 M_3 \rangle \langle M_2 M_4 \rangle\,\bigr)^2. 
\end{equation}
Using notation as in Example \ref{ex:rat_trian}
we can write this expression as $\Delta_{K_1,(4)} = \langle L^{(\mathbf{b})} L^{(\mathbf{w})} \rangle$, where 
\begin{equation}\label{eq:3_mass_sol}
    L^{(\mathbf{b})}  = (p  \, M_3) \star (p \, M_4) \, , \quad L^{(\mathbf{w})}  = (P \star M_3)  (P \star M_4) .
\end{equation}
\end{example}

We now discuss in detail the rational factorization of the LS discriminant for the triangle.
We use the following notation for the three factors in the generic regime of Theorem~\ref{thm:16168}:
\begin{equation}\label{eq:discr_tr}
            \Delta_{\mathcal{L}} \,\,=\,\, \Delta_{\mathcal{L}, \mathbf{b}} \cdot \Delta_{\mathcal{L}, \mathbf{w}} \cdot \Delta_{\mathcal{L}, \mathrm{mix}}\, ,
\end{equation}
Each of the three discriminants will now break into a product of many smaller factors.

\begin{theorem}\label{thm:tr_rat}
Let $\mathcal{L} = G_u \cup H^\triangle_u$ with $G = K_3$ and $u = (3,3,3)$.
The fiber of $\psi_\mathcal{L}$ 
is given in Example \ref{ex:rat_trian}. The labels
 $\{{\bf b},{\bf w}\}^4$ index the vertices of the $4$-cube.
The full LS discriminant $\Delta_{\mathcal{L}}$ has $32$ irreducible factors, one for each edge of the $4$-cube. The discriminant on the irreducible component of $V_G$ associated to $\sigma_{\rm int} \in \{{\bf b},{\bf w}\}$ factors into $12$ perfect squares as: \begin{equation}\label{eq:discr_tr_fact}
            \Delta_{\mathcal{L}, \sigma_{\rm int}} \,\,=\,\, \prod_{\sigma_1,\, \sigma_2} \langle L_i^{(\sigma_1)} L_i^{(\sigma_2)} \rangle \, .
        \end{equation}
     The product is over pairs $\sigma_1, \sigma_2$ of bicolorings whose first component is $\sigma_{\rm int}$ and which differ by only one color at position $i$ in $\sigma_{\rm ext}$. 
        The mixed part factors into $8$ perfect squares as: \begin{equation}\label{eq:discr_tr_fact_mixed}
            \Delta_{\mathcal{L}, \mathrm{mix}} \,\,=\,\, \prod_{\sigma} \langle p^{(\sigma)} P^{(\sigma)} \rangle \, .
        \end{equation}
        The point $p^{(\sigma)}$ and plane $P^{(\sigma)}$ come from
         the bicoloring~$\sigma \in \{\mathbf{b},\mathbf{w}\}^3$ of the external~triangles.
\end{theorem}

    \begin{proof}
    We first prove~\eqref{eq:discr_tr_fact}. The concurrent discriminant $\Delta_{\mathcal{L}, \mathbf{b}}$ vanishes when the lines $ L_i^{(\sigma_1)} $ and $L_i^{(\sigma_2)}$ in Example~\ref{ex:rat_trian} coincide.
    Hence $\langle L_i^{(\sigma_1)} L_i^{(\sigma_2)} \rangle $ divides 
$    \Delta_{\mathcal{L}, \mathbf{b}}$.
    Equation~\eqref{eq:discr_tr_fact} 
     follows by counting the degrees in the Plücker coordinates of ${\bf M}$ and comparing with Theorem~\ref{thm:16168}.
    We next prove~\eqref{eq:discr_tr_fact_mixed}. It follows from the constructions in
     Example \ref{ex:rat_trian} that $L_i^{(\mathbf{b}, \sigma)}=L_i^{(\mathbf{w},\sigma)}$ 
     if and only if  $p^{(\sigma)}$ lies in $P^{(\sigma)}$.
      It follows that $\langle p^{(\sigma)} P^{(\sigma)} \rangle$ divides the mixed LS discriminant
       $\Delta_{\mathcal{L}, \mathrm{mix}}$. We again reach the conclusion by counting degrees and comparing with Theorem~\ref{thm:16168}.
\end{proof}

We expect~Theorem~\ref{thm:tr_rat} to generalize to other rational Landau diagrams $\mathcal{L}$. The full LS discriminant $\Delta_{\mathcal{L}}$ factorizes into many factors, one for each irreducible component of $V_G$, and one for the mixed LS discriminants. 
All factors should be analogous to  those for the~triangle.

\begin{conjecture}\label{conj:rat_formula}
Given any rational Landau diagram $\mathcal{L}$, the LS discriminant on each irreducible component of $V_G$ can be expressed as in~\eqref{eq:discr_tr_fact}, and similarly the mixed LS discriminant as in~\eqref{eq:discr_tr_fact_mixed}. Both products run over all external colorings that are admissible for $\mathcal{L}$.
\end{conjecture}

The following example explains what we mean by ``admissible external colorings''.
 
\begin{example}[Triangulated Square]
    Let $\mathcal{L}=G_u \cup H^\triangle_u$ with $G=\{12,23,13,34,14\}$ and $u = (2,3,3,3)$. The LS degree equals $48 $, and $H^\triangle_u$ is a minimal rational degeneration of $G_u$.
    The fiber of $\psi_\mathcal{L}$ is in bijection with the following bicolorings 
    $\sigma \subset \{\mathbf{b},\mathbf{w}\}^6$ of the triangles in $G^\triangle$.
    If the colors $\sigma_1$ and $\sigma_2$ of the
           two internal triangles are distinct, then all external colorings $(\sigma_3,\sigma_4, \sigma_5, \sigma_6)$ are allowed.
          If they have the same color, then the external triangle at vertex $1$ must have that same color.
          The number of such $\sigma$ is indeed $2 \cdot 2^4 + 2 \cdot 2^3 = 48$.
          
           By Proposition~\ref{prop:bbexpected}, the
expected degree of $\Delta_{\mathcal{L}}$ is $(160,176,256,176)$.
We verified that this is consistent with   Conjecture~\ref{conj:rat_formula}.
The degrees in the proposed factorization add up to $(160,176,256,176)$.
    As an example, let $\sigma= (*\mathbf{bbbbw})$ and let $L_1^{(*)}$ denote the line at vertex $1$ in the fiber of $\psi_\mathcal{L}$.
    These two lines are given explicitly by the following geometric formulas:
    \begin{equation*}
    L_1^{(\mathbf{b})} = p_1  \big((p_1  p_2  p_3) \star (p_3  P_3) \star (p_4  P_4) \big) \, , \quad L_1^{(\mathbf{w})} = P_1 \star \Big( \big( P_1 \star(p_2  P_2) \star (p_3  P_3) \big)  p_3  p_4 \Big) \, .
    \end{equation*}
    The term $\langle L_1^{(\mathbf{b})} L_1^{(\mathbf{w})} \rangle $ is a factor of degree $(2,2,2,2)$ in the LS discriminant of
   $V_{(\mathbf{b},\mathbf{w})} \subseteq V_G$. Similarly, we can construct all $48 $ rational lines,
  and hence all factors~\eqref{eq:discr_tr_fact} and~\eqref{eq:discr_tr_fact_mixed} of $\Delta_{\mathcal{L}}$.
  \end{example}

\section{Positivity and Reality}
\label{sec:positivity}

In Section~\ref{sec:rationality} we examined when the fiber of the Landau map $\psi_{\mathcal{L}}$ is rational in the external lines $\mathbf{M}$.  We now study when the fiber $\psi_{\mathcal{L}}^{-1}({\bf M})$ is real. We seek
configurations $\mathbf{M}$ of real lines $M_1,\ldots,M_d$ with the property that
all $\gamma_u$ complex points in $\psi_{\mathcal{L}}^{-1}({\bf M})$  have real coordinates.
Our answer will produce large families of fully real Schubert problems.
From the physics perspective, reality is tightly linked to positivity of Landau discriminants and
to the expectation that planar \(\mathcal{N}=4\) SYM amplitudes are regular functions
on the \textit{positive configuration space}, which is the quotient of the positive Grassmannian by the torus action.

Recall that a real $k \times n$ matrix $Z$ is {\em totally positive} if all its $k \times k$ minors are positive.
We write ${\rm Mat}^{>0}_{k \times n}$ for the semi-algebraic set of these matrices. Its image in the
Grassmannian ${\rm Gr}(k,n)$ is denoted ${\rm Gr}_{>0}(k,n)$ and called the {\em positive Grassmannian}.
An element $f$ in the function field $\RR({\rm Gr}(k,n))$ is called {\em Grassmann copositive} if it is
regular on ${\rm Gr}_{>0}(k,n)$ and has no zeros there.
Thus $f$ represents a rational function that is strictly positive on ${\rm Gr}_{>0}(k,n)$.

 We introduce a notion of positivity for the external lines of a planar Landau diagram.
 The tuple $\mathbf{M}=(M_1,\dots,M_d)$ in $\Gr(2,4)^d$ is \emph{positive} if each line $M_i$ is spanned by a pair $(z_{p_i},z_{p_i+1})$ of consecutive columns of a totally positive matrix $Z=[z_1|\cdots|z_n]\in \Mat^{>0}_{4\times n}(\mathbb{R})$
 Here $1=p_1<\cdots<p_d\leq n \leq 2d$, and, if the lines $M_i, M_{i+1}$ are incident, then $p_{i+1}=p_i+1$.

\begin{conjecture}[Reality Conjecture]\label{conj:reality_conjecture}
Let $G$ be an outerplanar graph, with cyclically labeled vertices in $[\ell]$, and fix
any of its planar Landau diagrams $\mathcal{L} = G_u \cup H_u$. 
 If $\,\mathbf{M}$ is positive then the number of points in the
  fiber $\psi_\mathcal{L}^{-1}(M)$ equals the LS degree, and these points are all~real.
  \end{conjecture}

This section presents strong evidence for this result, both theoretical and computational.
We argue that Conjecture \ref{conj:reality_conjecture} can be approached
by recursively decomposing the graph $G$. A key player is the substitution map~\eqref{eq:subs_map},
which we shall view as a cluster promotion map~\cite{VRC}.
Our next result offers a first glimpse on the connection to positroids and cluster algebras.

\begin{proposition}\label{prop:copositivity_resultant}
The discriminant (\ref{eq:discriminant_4_lines}) 
and the resultant  (\ref{eq:resultant_5_lines}) are Grassmann copositive.
\end{proposition}

\begin{proof}
For the LS discriminant $\Delta_{K_1,(4)}({\bf M})$, this result was established in
\cite[Section 4]{ALS}.
We here present a proof for the SLS resultant $R_{K_1,(5)}({\bf M})$.
It rests on the birational parametrization 
$\RR_{>0}^{24} \rightarrow {\rm Gr}_{>0}(4,10)$ of the positive Grassmannian that is given by
the $4 \times 10$ matrix
$$ \! Z =\,\begin{tiny} \begin{bmatrix}
\,\,1 & 0 &  0 &  0 \smallskip \,\, \,\, \\ \,\,
g+k+o+s+v+x &  1 &  0 &  0 \medskip \,\, \,\,  \\ \,
gh{+}(g{+}k)l{+}(g{+}k{+}o)p{+}(g{+}k{+}o{+}s)t{+}(g{+}k{+}o{+}s{+}v)w &
d+h+l+p+t+w &  1 &  0 \smallskip \,\, \,\, \\ \,
{ghi{+}(gh{+}(g{+}k)l)m{+}(gh{+}(g{+}k)l{+}(g{+}k{+}o)p)q{+}  \atop
(gh{+}(g{+}k)l{+}(g{+}k{+}o)p{+}(g{+}k{+}o{+}s)t)u }&
{de{+}(d{+}h)i{+}(d{+}h{+}l)m{+} \atop (d{+}h{+}l{+}p)q{+}(d{+}h{+}l{+}p{+}t)u }&
b+e+i+m+q+u &  1 \medskip \,\, \,\, \\ \,\,
(ghi+(gh{+}(g{+}k)l)m+(gh{+}(g{+}k)l+(g{+}k{+}o)p)q)r &
(de{+}(d{+}h)i{+}(d{+}h{+}l)m{+}(d{+}h{+}l{+}p)q)r &  (b+e+i+m+q)r &
r \smallskip \,\, \,\, \\ \,\, (ghi{+}(gh{+}(g{+}k)l)m)n &
(de{+}(d{+}h)i{+}(d{+}h{+}l)m)n &  (b+e+i+m)n &
n  \medskip \,\, \,\, \\ \,\,
ghij &  (de+(d+h)i)j &  (b+e+i)j &  j \smallskip \,\, \,\, \\ \,\,
0 &  def &  (b+e)f &  f \medskip \,\, \,\, \\ \,\,
0 &  0 &  bc &  c \smallskip \,\, \,\, \\ \,\, 0 & 0 &  0 &  a \,\, \\
\end{bmatrix}^t\!\! . \end{tiny} $$
This parametrization is associated with the affine permutation $(5, 6, 7, 8, 9, 10, 11, 12, 13, 14)$
in the theory of positroids. We found the matrix $Z$ with Bourjaily's {\tt Mathematica} package~\cite{mathematica_plabic}.
Writing $M_i$ for the line spanned by columns $z_{2i-1}$ and $z_{2i}$, we now
substitute $Z$ into the Gram determinant (\ref{eq:resultant_5_lines}).
The determinant evaluates to a polynomial of degree $30$ in the $24$ parameters $a,b,\ldots,x$.
We see that each of its $5917$ monomials has a positive coefficient.
\end{proof}

Consider a substitution map $\bar{\varphi}: \Mat_{4,2d} \rightarrow \Mat_{4,2d'}$, with $d > d' \geq 2$. 
Then $\bar{\varphi}$ extends~to:
\begin{itemize}
    \item a map $\tilde{\varphi}: \Gr(4,2d) \rightarrow \Gr(4,2d')$, between Grassmannians $\Gr(4,2s) \simeq \Mat_{4,2s}/{\rm GL}(4)$;
    \item a map $\varphi: \Gr(2,4)^d \rightarrow \Gr(2,4)^{d'}$, between spaces of lines $\Gr(2,4)^{s} \simeq \Mat_{4,2s}/{\rm GL}(2)^s$.
\end{itemize}
If the map $\bar{\varphi}$ sends totally positive matrices to totally positive matrices, then $\tilde{\varphi}$ sends the positive Grassmannian to the positive Grassmannian, and $\varphi$ sends positive line
configurations to positive line configurations.
If this holds, then we say that these maps are \emph{copositive}. 

\begin{lemma}\label{lem:4mb_pos}
The four-mass box substitution maps $\varphi_{\pm}$ in~\eqref{eq:four_box_subs_maps} are copositive.  
\end{lemma}

\begin{proof}
Let $[z_1|z_2|\cdots|z_{2d}] \in {\rm Mat}^{>0}_{4 \times 2d}$ and  $M_i = z_{2i-1} z_{2i}$.
Then $\varphi_\pm({\bf M})$ is the configuration of
 lines spanned by consecutive pairs of columns in the matrix $\widetilde{\mathbf{z}}^\pm = [v^{\pm}|w^{\pm}|z_{9}|z_{10}|\ldots|z_{2d-1}|z_{2d}]$, 
 with $X^{\pm} = v^{\pm} w^{\pm}$ as in \eqref{eq:Xpm}
are two transversals to the four lines $M_1, \ldots, M_4$. On these lines we selected the points
  $v^\pm = X^\pm \cap M_1$ and $w^\pm = X^\pm \cap M_2$.
  Note that $X^{\pm} = L^{\pm}$ as lines, from above~\eqref{eq:4box_extr_fact}, and~\eqref{eq:4box_extr_fact} is non-vanishing on the positive Grassmannian as its square is a Laurent monomial in cluster variables on ${\rm Gr}(4,8)$. It follows that $\varphi_{\pm}$ are copositive if and only if the $4 \times (2d-6)$ matrices $\widetilde{\mathbf{z}}^\pm$ are both totally positive. The latter point follows from the results in \cite[Section 10]{VRC}. The
maps associated to $\widetilde{\mathbf{z}}^\pm$ are essentially the same~as the \emph{$4$-mass promotion maps} $\tilde{\varphi}_\pm: {\rm Gr}(4,n) \rightarrow {\rm Gr}(4,n')$, with $n=2d$ and $n'=n-4$, mapping $[z_1|z_2|\cdots|z_n]$ to
$[z_1|w^\pm|v^\pm|z_8|z_9|\cdots|z_n]$. By \cite[Corollary 10.12]{VRC},
the maps $\tilde{\varphi}_\pm$ are copositive, and hence so are $\varphi_{\pm}$.
We will discuss promotion maps extensively in Section \ref{sec:promotion_maps}.
\end{proof}

We can combine the copositivity of $\varphi_{\pm}$
with Proposition~\ref{prop:rec_fibers} to obtain the following.

\begin{theorem}\label{th:real_trees}
The Reality Conjecture~\ref{conj:reality_conjecture} is true when the graph $G$ is a tree.   
\end{theorem}

\begin{proof}
    We prove this by induction on $\ell$.
        Fix $u$ so that $|u|={\rm dim}(V_G)$. Since $G$ is a tree, $V_G$ is a complete intersection and its multidegree is given by~\eqref{eq:LS_CI}. Moreover, the tree $G$ has a vertex $i$ with $u_i=4$. Let $e_1,\ldots,e_s$ be the edges of $G$ connected to $i$. We apply Proposition~\ref{prop:rec_fibers} to $G_1 \cup E \cup G_2$, where $G_1$ is the one-vertex graph, $E=\{e_1,\ldots,e_s\}$, and $G_2=\cup_{i=1}^s\tilde{G_i}$ with the tree $\tilde{G_i}$ being the subgraph of $G$ connected to $i$ via the edge $e_i$. The claim follows by induction, where we combine
         Proposition~\ref{prop:rec_fibers} with Lemma~\ref{lem:4mb_pos}.
          In particular, Lemma~\ref{lem:4mb_pos} guarantees inductively t
          hat the  maps $\varphi_{\pm}$ are well-defined on
        the set of positive        configurations ${\bf M}$,
         so that the recursive construction of Proposition~\ref{prop:rec_fibers} works.
\end{proof}

\begin{remark} \label{rmk:allowfor}
In the Reality Conjecture \ref{conj:reality_conjecture}, we allow for the
possibility that two real points in $\psi_{\mathcal{L}}^{-1}({\bf M})$ collide if they come
from different irreducible components of $V_G$. Our conjecture concerns
each individual component $V_{G,\sigma}$ separately. Namely, whenever ${\bf M}$ is positive,
we want the fiber of the restricted Landau map $\psi_{\mathcal{L}}|_{V_{G,\sigma}}$
over ${\bf M}$ to be reduced and fully real.
\end{remark}

\begin{theorem} \label{thm:copositive}
Suppose that the Reality Conjecture  \ref{conj:reality_conjecture}
holds for a planar Landau diagram $\mathcal{L}$.
Then the LS discriminant for each irreducible component of $\,V_G$
is Grassmann copositive.
\end{theorem}

\begin{proof}
Let $\Delta_{\mathcal L}$ be the irreducible LS discriminant associated with the component $V_{G,\sigma}$. By Conjecture~\ref{conj:reality_conjecture}, for every positive ${\bf M}\in {\rm Gr}_{>0}(4,2d)$ the fiber $\psi_{\mathcal L}^{-1}({\bf M})$ consists of exactly the LS degree many real points. In particular, no such ${\bf M}$ lies in the branch locus of $\psi_{\mathcal L}$ for the component $V_{G,\sigma}$.
Hence $\Delta_{\mathcal L}({\bf M})\neq 0$ for all
${\bf M}\in {\rm Gr}_{>0}(4,2d)$.
Since ${\rm Gr}_{>0}(4,2d)$ is connected (indeed, it is homeomorphic to a ball~\cite{postnikov}), the continuous function $\Delta_{\mathcal L}$ has constant sign there. Fixing its overall sign to be positive at one point of ${\rm Gr}_{>0}(4,2d)$, we conclude that $\Delta_{\mathcal L}({\bf M})>0$ for all positive ${\bf M}$. Hence $\Delta_{\mathcal L}$ is Grassmann copositive.
\end{proof}


\begin{corollary} \label{cor:LLDtree}
 The LS discriminant is Grassmann copositive whenever the graph $G$ is a~tree.
\end{corollary}

We now turn to Grassmann copositivity for SLS resultants.
This was shown for the pentagon diagram ($\ell=1$) in Proposition \ref{prop:copositivity_resultant}.
We conjecture that Grassmann copositivity holds for
the multigraded Chow form of
every irreducible variety $V_{G,\sigma}$, provided $G$ is outerplanar.
 In other words, 
 the conclusion of Theorem \ref{thm:copositive} is expected to be valid
for all super-leading Landau diagrams. 
What we can show here is the following analogue to 
Corollary \ref{cor:LLDtree}.

\begin{proposition}
 The SLS resultant is Grassmann copositive whenever the graph $G$ is a~tree.
\end{proposition}

\begin{proof}
    The proof is again by induction on $\ell$.
    Since $\mathcal{L}$ is super-leading and $G$ is a tree, we can assume  that 
    there exists a vertex $i$ of $G$ with $u_i = 4$. In fact, if this is not the case, 
    there exists a super-leading subdiagram $\mathcal{L}^{'}$ of $\mathcal{L}$, 
    and $R_{\mathcal{L}}  = R_{\mathcal{L}^{'}}$.
    The proof is now similar to that of Theorem~\ref{th:real_trees}.
We use Corollary~\ref{cor:rec_res_4box} to reduce the tree by removing the vertex $i$.
\end{proof}

It is our aim  to extend the results above to graphs which are not trees.
One approach we pursued is computer experiments.
Namely, we undertook a thorough study of  the reality of fibers
using methods from numerical algebraic geometry.
In what follows, we
report on the computations we carried out
with the software {\tt HomotopyContinuation.jl}  \cite{BT}.

We first checked that the Reality Conjecture \ref{conj:reality_conjecture} holds for all outerplanar graphs 
with $\ell =3, 4$ vertices. For each graph $G$ and vector $u$ such that $\gamma_u \neq 0$, we solved the square system of equations which cut out $V_{G, u}$ for $10^5$ random choices of positive 
line configurations $\mathbf{M}$.
These  are spanned by a pair of consecutive columns of a totally positive matrix $Z$. We then verified that each of the LS degree many solutions is real, meaning that the imaginary part of the coordinates of each solution is less than some tolerance. For all such outerplanar graphs, we found only real solutions across our $10^5$ random positive external lines $\mathbf{M}$. 

We then checked numerically that the substitution maps for both $G = K_2$ and $u = (4,3)$ as well as $G = K_3$ and $u = (3,3,3)$ are copositive.
This positivity result means that we can use the graphs as `seeds' for the recursions as in Proposition \ref{prop:rec_fibers} to prove that certain
infinite families of graphs satisfies the Reality Conjecture  \ref{conj:reality_conjecture}.
For example, $G=K_3$ with $u=(4,3,2)$ can be obtained from a recursion from the pentabox $G = K_2$ and $u = (4,3)$.

To do this, we again solve the system of equations which cut out $V_{G, u}$ numerically for external lines $\mathbf{M}$ coming from a totally positive matrix $Z_1$ which is a submatrix of a $4 \times n$ totally positive matrix  $Z = [Z_1 | Z_2]$. In the case of $G = K_3$, the matrix $Z_1$ has size $4 \times 18$ and $Z_2$ has size $4 \times (n-18)$. We then solve the system for the lines $L_1, L_2, L_3$ and form the points $x = L_1 \cap M_1, ~ p = L_1 \cap L_2,~ y = L_2 \cap M_9$,
where $M_1, \ldots, M_9$ are the external lines attached to $G = K_3$ arranged in cyclic order. Lastly, we check that for each solution, the matrix $[x|p|y|Z_2]$ is totally positive, though it suffices to check for $n=21$, i.e. that the $4 \times 6$ matrix $[x|p|y|Z_2]$ is totally positive. This implies that the substitution maps associated to $G=K_3$ are copositive and suggests that Corollary \ref{cor:LLDtree} also holds for diagrams made by iteratively gluing together triangles. An analogous statement would hold for the pentabox~$G=K_2$. 

\section{Positroid Varieties}
\label{sec:positroid}

In this section we establish a connection between
incidences of lines in $\PP^3$
and Grassmannians in higher dimensions. 
LS degrees, discriminants and resultants have natural interpretations
in terms of positroid varieties \cite{positive_geom, mathematica_plabic}
and their amplituhedron maps \cite{amplituhdr, m4tiling}. Moreover, lines in the fibers of Landau maps are obtained from 
vector configurations
associated with positroids.

We fix a graph $G$ on $[\ell]$, and we assume for simplicity
that $G$ is outerplanar. 
The Landau diagram $ G_u$ is planar for all $u \in \NN^\ell$ with $|u| = d = {\rm dim}(V_G)$.
  We fix an irreducible component $V_{G,\sigma}$ of the variety $V_G$. 
  By Theorem \ref{thm:outerplanar},   this is indexed
  by a bicoloring $\sigma$ of the triangles in $G$.
 A black triangle means that three lines are concurrent. A white triangle means 
that they are coplanar. We denote by $V_{G,\sigma,u}$ the component of $V_{G_u}$ corresponding to $\sigma$, and by $\psi_{G,\sigma,u}$ the associated Landau map. 
The LS degree $\gamma_{G,\sigma,u}$ is the coefficient in the multidegree
$[V_{G,\sigma}]$ which counts the solutions to our Schubert problem.
The LS discriminant $\Delta_{G,\sigma,u}$ vanishes
when two solutions come together.
In this section,
we write $\Delta_{G,\sigma,u}$ as a polynomial in the entries
of a $4 \times n$ matrix ${\bf M}$ where $n=2d$ and the $d$
lines are spanned by consecutive~columns.

Let $r = r(G,\sigma)$ be the largest integer such that
all line incidences from $V_{G,\sigma}$ are realized by lines spanning the full projective space $\PP^{r-1}$.
Fix such a realization and pick $u_i$ generic points on its $i$-th line.
Let $M_{G,\sigma,u}$ be the rank $r$ matroid given by
these $d = |u|$ points in $\PP^{r-1}$.
Planarity ensures that, after relabeling, the matroid $M_{G,\sigma,u}$ is a positroid.
The dual of this positroid has rank $k := d-r$ on the same ground set $[d]$.
We extend this dual by $d$ new elements which are parallel to the given $d$ elements.
This yields a positroid $\Pi_{G,\sigma,u}$ of rank $k$ on $[n] = [2d]$.
The ranks $r$ and $k$ depend on choice of
coloring $\sigma$ of $G$. Namely, $r$ grows with the number of
black triangles. If all triangles are white then $r$ is small and $k$ is large.
We identify $\Pi_{G,\sigma,u}$  with its positroid variety in ${\rm Gr}(k,n)$.
The dimension of this variety is $4k$.

The ($m=4$) amplituhedron map $Z$ from
${\rm Gr}(k,n)$ to ${\rm Gr}(k,k+4)$ 
restricts to a finite-to-one map on the variety
$\Pi_{G,\sigma,u}$. Here $Z$ is any matrix whose
kernel is the row space of ${\bf M}$.
The degree of the map $Z : \Pi_{G,\sigma,u} \rightarrow {\rm Gr}(k,k+4)$ is the
{\em intersection number} of the positroid.
Our main result states that this
is a faithful  representation of our Schubert problem in $3$-space.

\begin{theorem} \label{thm:positroid}
The LS degree $\gamma_{G,\sigma,u}$ equals the intersection number
of the positroid $\Pi_{G,\sigma,u}$. Solutions to the Schubert problem
given by ${\bf M}$ correspond to subspaces
$\mathfrak{L} \in \Pi_{G,\sigma,u}$ such that
$\mathfrak{L} \subseteq {\rm kernel}({\bf M})$.
Hence, the LS discriminant $\Delta_{G,\sigma,u}$
equals the Hurwitz-Lam form of $\Pi_{G,\sigma,u}$.
\end{theorem}

Before we get to the proof of Theorem \ref{thm:positroid}, 
we need to explain the ingredients in the statement.
The {\em Hurwitz-Lam form} was defined in \cite[Section 5]{PS}
for any subvariety $\mathcal{V}$ of dimension~$4k$ in ${\rm Gr}(k,n)$.
For a general point $P \in {\rm Gr}(n-4,n)$, the
Grassmannian ${\rm Gr}(k,P)$ is a subvariety
of codimension $4k$ in ${\rm Gr}(k,n)$. The intersection
${\rm Gr}(k,P) \cap \mathcal{V}$ consists of finitely many 
points. Lemma 5.1 in \cite{PS} refers to this number as the
{\em Hurwitz-Lam degree} of $\mathcal{V}$. It appears as
a coefficient in the Schubert expansion of the
 class $[\mathcal{V}] \in A^*({\rm Gr}(k,n))$.
The Hurwitz-Lam form ${\rm HL}_\mathcal{V}$
is the locus of points $P$ where the intersection is not transverse
\cite[eqn (21)]{PS}.
Theorem \ref{thm:positroid} applies this to
the special case when $\mathcal{V}$ is the positroid variety $\Pi_{G,\sigma,u}$.
The branch locus of the amplituhedron map
$Z : \mathcal{V} \rightarrow {\rm Gr}(k,k+4)$ is obtained from
the Chow-Lam form ${\rm HL}_{\mathcal{V}}$
by substituting twistor coordinates, as explained in \cite[Theorem~5.2]{PS}.

\begin{example}[$\ell=3$]\label{ex:positroid_triangle}
Fix the triangle $G = K_3$ with $u=(3,3,3)$,
$d= {\rm dim}(V_G) = 9$, and $n=18$.
  Here $\sigma$ is either black ({\bf b}) or white ({\bf w}),
indicating whether the three lines are concurrent or coplanar.
We begin with $\sigma = {\bf w}$, so $V_{G,{\bf w},u} =  V_{[3]}^*$.
Then $r = r(G,{\bf w}) = 3$ because the three lines must lie in a plane $\PP^2$.
The rank $3$ matroid $M_{G,{\bf w},u}$ has three non-bases
$123$, $456$, $789$, namely the points on the three internal lines
that lie on the nine external lines. The dual $M_{G,{\bf w},u}^*$ has rank $k=6$,
with three non-bases $456789$, $123789$ and $123456$.
The common realization space of the two matroids has dimension $15$. 
Indeed, the three non-bases impose independent constraints on the
Grassmannian ${\rm Gr}(3,9) \simeq {\rm Gr}(6,9)$ which has dimension $18$.
We obtain the rank $6$ positroid $\Pi_{G,{\bf w},u}$ 
from $M_{G,{\bf w},u}^*$ by simply  duplicating each point.
In matroid language, we create $9$ pairs of parallel elements.
The realization space of $\Pi_{G,{\bf w},u}$
is the positroid variety in ${\rm Gr}(6,18)$. It has
dimension $4k=24 = 15+9$.
We now fix nine general lines in $\PP^3$. These external data are written as
a $4 \times 18$ matrix~${\bf M}$.
The subspace $P = {\rm kernel} ({\bf M})$ is an element in $ {\rm Gr}(14,18)$.
The sub-Grassmannian ${\rm Gr}(4,P)$ has dimension
$48$ in the $72$-dimensional ambient Grassmannian ${\rm Gr}(4,18)$,
and it intersects the positroid variety $\Pi_{G,\sigma,u}$ in eight points.
Indeed, the LS degree is $8$. This is the intersection number of
$\Pi_{G,\sigma,u}$, and hence Hurwitz-Lam degree in \cite[Lemma 5.1]{PS}.
Two of the eight points come together for special matrices ${\bf M}$.
The condition for this is the LS discriminant $\Delta_{G,\sigma,u}$,
which has degree $16$ in each column of ${\bf M}$. 
Theorem \ref{thm:positroid} identifies $\Delta_{G,\sigma,u}$
 with the Hurwitz-Lam form of $\Pi_{G,\sigma,u}$, which has degree $16$ in the Pl\"ucker
coordinates on ${\rm Gr}(14,18)$.

We next consider $\sigma = {\bf b}$, so $V_{G,{\bf b},u} =  V_{[3]}$.
Here, $r = r(G,{\bf b}) = 4$ because the three concurrent lines span $\PP^3$.
Now, $M_{G,{\bf b},u}$ is the transversal matroid of rank $4$ given by a matrix
$$ \begin{footnotesize} \begin{bmatrix}
\star & \star & \star & \star & \star & \star & \star & \star & \star \\
\star & \star & \star & 0 &0 &0 &0 &0 &0 \\
0 & 0 & 0 & \star & \star & \star & 0 &0 &0 \\
0 & 0 & 0 & 0 & 0 & 0 & \star & \star & \star 
\end{bmatrix}. \end{footnotesize} $$
Its dual has rank $k=5$. The common realization space has dimension $11$ in
${\rm Gr}(4,9) \simeq {\rm Gr}(5,9)$. We get the rank $5$ positroid $\Pi_{G,{\bf b},u}$ 
by duplicating each point. The variety 
$\Pi_{G,{\bf b},u}$ lives in ${\rm Gr}(5,18)$ and its
 dimension is $4k=20=11+9$. 
 We again fix nine lines
 in~$\PP^3$, written as a $4 \times 18$ matrix ${\bf M}$,
 and representing a point $P = {\rm kernel} ({\bf M})$ in $ {\rm Gr}(14,18)$.
 The sub-Grassmannian ${\rm Gr}(5,P)$ has dimension
$45$ in the $65$-dimensional Grassmannian ${\rm Gr}(5,18)$,
and it intersects the positroid variety $\Pi_{G,{\bf b},u}$ also in eight points.
The LS degree~$8$ is  the intersection number of $\Pi_{G,{\bf b},u}$.
Two points come together when ${\bf M}$ is
in the LS discriminant $\Delta_{G,{\bf b},u}$. This
has degree $16$ in each column of ${\bf M}$. 
Theorem \ref{thm:positroid} identifies $\Delta_{G,{\bf b},u}$
 with the Hurwitz-Lam form of $\Pi_{G,{\bf b},u}$, which has degree $16$ in the Pl\"ucker
coordinates on ${\rm Gr}(14,18)$.
\end{example}

We used geometric language to describe the passage from line incidences to positroid varieties.
A more precise combinatorial version is needed for the proof of  Theorem \ref{thm:positroid}.
For this, we employ  the \emph{vector-relation configurations} (VRCs) from \cite{AGPR, VRC},
and the representation of positroids by   \emph{plabic graphs} \cite{postnikov}.
  These are {\bf pla}nar {\bf bic}olored graphs with ordered boundary vertices.
In our construction, plabic graphs associated with uniform matroids $U_{k',n'}$ appear naturally as subgraphs. The case of primary interest is $k'=n'-2$, shown in Figure \ref{fig:grassmann_lines}. We depict the plabic graph for $U_{k',n'}$ by a degree-$n'$ vertex $v$ labeled by $k'$, and we refer to $h(v)=k'$ as its \emph{helicity}. In the special cases $U_{1,n'}$ and $U_{n'-1,n'}$, the vertex is drawn white and black, respectively. The resulting plabic graphs are also called \emph{Grassmann graphs} \cite{GrassGraph,postnikov}.

Starting from the graph $G$ on $[\ell]$, we construct
the plabic graph $\mathcal{G}:=\mathcal{G}_{G,\sigma,u}$ 
which represents the positroid $\Pi_{G,\sigma,u}$, as follows.
 On each vertex $i$ of $G$, place a vertex $v_i$ of $\mathcal{G}$.
For $j\in [\ell]$ belonging to the same black region of $\sigma$, connect every vertex $v_j$ to a common black vertex.
For every edge $ij$ of $G$ that does not lie in a black region, connect $v_i$ and $v_j$ by an edge.
For each boundary vertex $i$ of $G$, attach $u_i$ white tripods to $v_i$.
Finally, assign to each vertex $v_i$ the helicity $
h(v_i)={\rm deg}(v_i)-2$. Here the degree is counted after the previous step.

The positroid  $\Pi_\mathcal{G}$ is read off from the plabic graph $\mathcal{G}$ as follows.
 Its matroid bases are the source sets of {\em perfect orientations} of $\mathcal{G}$.
These are orientations of $\mathcal{G}$ such that every internal vertex $v$ has exactly $h(v)$ incoming edges.
The source set consists of those elements $i$ for which the edge adjacent to the boundary vertex $i$ is oriented away from $i$.
Each source set has cardinality $k$, and this is the rank of the positroid.
The positroid variety $\Pi_\mathcal{G}$ lives in the Grassmannian $\Gr(k,n)$.
Figure~\ref{fig:grassmann_graph_K_3} shows the plabic graphs for the positroids 
 in Example~\ref{ex:positroid_triangle}.

 \begin{figure}[htbp]
    \centering
\includegraphics[width=0.9\textwidth]{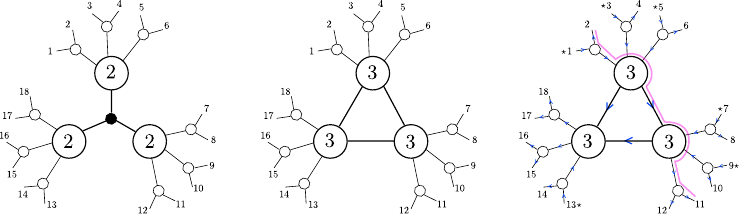}
    \caption{Plabic graphs $\mathcal{G}_{G,\sigma,u}$ for $\sigma={\bf b}$ on the left and
for    $\sigma = {\bf w}$ in the center.
     The perfect orientation for $\mathcal{G}_{G,{\bf w},u}$  on the right has
          the source set $\{1,3,5,7,9,13\}$. This is a basis of $\Pi_\mathcal{G}$. }
    \label{fig:grassmann_graph_K_3}
\end{figure} 

Another convenient way to give a  positroid $\Pi_\mathcal{G}$ is by a permutation $\pi_\mathcal{G}$ 
of the set~$[n]$.
Starting at a boundary vertex $i$ and travelling into $\mathcal{G}$, one turns $h(v)$ times clockwise whenever encountering a vertex $v$.
This path eventually terminates at a boundary vertex $\pi_\mathcal{G}(i)$ of $\mathcal{G}$.
The positroid $\Pi_{G,{\bf w},u}$ in Example~\ref{ex:positroid_triangle} is
thus encoded by the permutation $\pi_{G,{\bf w},u}=(2,11,4,15,6,1,8,17,10,3,12,7,14,5,16,9,18,13)$.
The path $2 \rightarrow 11$ shows $\pi_{G,{\bf w},u}(2)=11$.

With the positroid $\Pi_\mathcal{G}$, we associate a \emph{vector-relation configuration}
\cite[Definition 3.2]{VRC}. 
A VRC is  an assignment of vectors $v_b \in \mathbb{C}^4$ to the black vertices $b$ of $\mathcal{G}$ and coefficients $r_e \in \mathbb{C}^*$ to the edges $e$ such that, for each white vertex $w$, there is a linear relation $
\sum_{e=(wb)} r_e v_b=0$.
We consider VRCs up to \emph{gauge transformations},
by the natural  ${\rm GL}(4)$-action and rescalings that preserve the linear relations.
Let $z_1,\ldots,z_n$ be the vectors assigned to the boundary vertices of $\mathcal{G}$.
Then $
{\bf z}=[z_1|\cdots|z_n] \in \Gr(4,n)$,
and the vectors $v_b$ at the internal vertices determine points in $\PP^3$.
We denote by $\mathcal{C}^{\bf z}(\mathcal{G})$ the set of VRCs of $\Pi_\mathcal{G}$ with fixed boundary ${\bf z}$.

Given an irreducible component $V_{G,\sigma}$ of $V_G$ and $u \in \NN^\ell$, we can construct points in the fiber of its Landau map $\psi_{G,\sigma,u}$ from the VRCs of the positroid $\Pi_{G,\sigma,u}$.
This correspondence will also play an important role later on, in connection with positivity and cluster algebras.

\begin{proposition}\label{prop:lines_VRC}
Let ${\bf z}$ be a $4 \times n$ matrix, and let ${\bf M}_{\bf z}=(M_1,\ldots,M_d)$ be the external lines 
which are spanned by pairs of consecutive columns of ${\bf z}$. Then
$\,
\psi^{-1}_{G,\sigma,u}({\bf M}_{\bf z}) \simeq \mathcal{C}^{\bf z}(\mathcal{G}_{G,\sigma,u}) $.
\end{proposition}

\begin{proof}
We first construct a map from VRCs to line configurations in $V_{G,\sigma,u}$.
   For each vertex $i$ of $G$, there is a corresponding vertex $v_i$ of $\mathcal{G}:=\mathcal{G}_{G,\sigma,u}$.
This vertex determines a copy of the plabic graph $\mathcal{G}_{n'}$ for the uniform matroid $U_{n'-2,n'}$,
where $n'$ is the degree of $v_i$, see Figure~\ref{fig:grassmann_lines}.
In any VRC on $\mathcal{G}_{n'}$, the vectors $w_1,\ldots,w_{n'}$ on the boundary vertices are collinear. Hence they span a line $L(v_i)\subset \PP^3$.
These lines satisfy the incidence relations prescribed by $G_u$ and $\sigma$. Indeed,
if two vertices $v_i$ and $v_j$ form an edge in $\mathcal{G}$, then the lines $L(v_i)$ and $L(v_j)$ are incident.  
They intersect in the point $w_{ij}=L(v_i)\cap L(v_j)$ assigned in the VRC to the black vertex lying on the edge $(ij)$.
If vertices $i,j,s$ of $G$ form a white triangle in $\sigma$, then the lines $L(v_i),L(v_j),L(v_s)$ are coplanar, since their intersection points $w_{ij},w_{js},w_{is}$ are distinct.
If  $i,j,s$ form a black triangle in $\sigma$, then $L(v_i),L(v_j),L(v_s)$ are concurrent, since they share the point $w$ corresponding to the black vertex connected to $v_i,v_j,v_s$ in $\mathcal{G}$.
Finally, let $u_i$ be a white vertex of $\mathcal{G}$ belonging to the tripod attached to the boundary vertices labeled $p_i,p_{i+1}$. Then $L(u_i)=M_i$, since $L(u_i)$ is spanned by the vectors $z_{p_i}$ and $z_{p_{i+1}}$.
Thus every VRC with boundary ${\bf z}$ determines a configuration of lines ${\bf L}=(L(v_1),\ldots,L(v_\ell))$ satisfying the incidence relations prescribed by $G_u$ and $\sigma$, with fixed external lines ${\bf M}_{\bf z}$. Hence ${\bf L}\in\psi^{-1}_{G,\sigma,u}({\bf M}_{\bf z})$.

Conversely, given  ${\bf L}=(L(v_1),\ldots,L(v_\ell))$ in 
$\psi^{-1}_{G,\sigma,u}({\bf M}_{\bf z})$, we can reconstruct the VRC.
For every internal black vertex $b$ lying on an edge  $\{v_i,v_j\}$, 
set $w_b=L(v_i)\cap L(v_j)$. These  points determine the vectors 
assigned to the black vertices. The boundary vectors are fixed by ${\bf z}$, 
and the linear relations at white vertices are then uniquely fixed (up to 
gauge transformations) by the linear relations among these vectors. 
This reconstructs a VRC with boundary~${\bf z}$.

The ${\rm GL}(4)$-action on ${\bf z}$ acts by projective transformations on the  lines ${\bf L}$ in $\PP^3$. Therefore, if ${\bf z}$ and ${\bf z'}$ are related by the ${\rm GL}(4)$-action, then $\psi^{-1}_{G,\sigma,u}({\bf M}_{\bf z})\simeq\psi^{-1}_{G,\sigma,u}({\bf M}_{\bf z'})$.
Conversely, suppose ${\bf z}$ and ${\bf z'}$ satisfy ${\bf M}_{\bf z}={\bf M}_{\bf z'}$, i.e.\ they differ by a ${\rm GL}(2)^d$-action acting independently on each external line. Then $\mathcal{C}^{\bf z}(\mathcal{G})\simeq\mathcal{C}^{\bf z'}(\mathcal{G})$, since every internal black vertex $b$ of $\mathcal{G}$ lies on an edge connecting two vertices $v_i,v_j$. The associated vector $w_b=L(v_i)\cap L(v_j)$ is therefore determined by the lines $L(v_i)$ and $L(v_j)$, which depend only on the external lines $M_i=(z_{p_i}z_{p_{i+1}})=(z'_{p_i}z'_{p_{i+1}})$.
This establishes the asserted identification $\psi^{-1}_{G,\sigma,u}({\bf M}_{\bf z})\simeq\mathcal{C}^{\bf z}(\mathcal{G})$.
\end{proof}

\begin{figure}[htbp]
    \centering
    \includegraphics[width=1.0\textwidth]{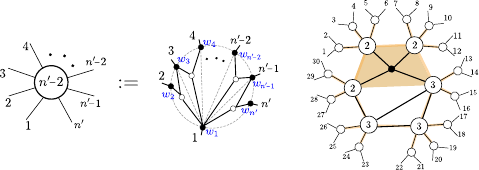}
    \caption{Left: Vertex of a Grassmann graph of degree $n'$ and helicity $h(v)=n'-2$. Center: Plabic graphs for the positroid $U_{n'-2,n'}$. Right: Grassmann graph $\mathcal{G}_{G,\sigma,u}$ in Example \ref{ex:positroid_Fibonacci}, with $\sigma=\{{\bf bbww}\}$.
    The positroid variety $\Pi_{G,\sigma,u}$ in $\Gr(10,30)$ has dimension $40$
    of intersection number $24$. Shown in orange is the bicoloring $\sigma$ of
    the subdivided hexagon $G_u$.}
    \label{fig:grassmann_lines}
\end{figure}

Recent work of Even-Zohar, Parisi, Sherman-Bennett, Tessler and Williams in \cite{VRC}
established the following relation between
 VRCs of positroids and the amplituhedron map.

\begin{theorem}[\cite{VRC}, Theorem 6.8]\label{th:VRC_Amp}
Let $\Pi_\mathcal{G}$ be a positroid variety in $\Gr(k,n)$, and let
${\bf z}$ be any point in the Grassmannian $\Gr(4,n)$.  Then
$\,\mathcal{C}^{\bf z}(\mathcal{G}) \simeq \Gr(k,{\bf z}^\perp) \cap \Pi_\mathcal{G}$, with ${\bf z}^\perp = {\rm kernel}({\bf z})$.
\end{theorem}

When $\Pi_\mathcal{G}$ has dimension $4k$, the set $\mathcal{C}^{\bf z}(\mathcal{G})$ is finite, and its cardinality equals the intersection number of $\Pi_\mathcal{G}$, which is the degree of the amplituhedron map restricted to $\Pi_\mathcal{G}$.

\begin{proof}[Proof of Theorem \ref{thm:positroid}]
Combining Theorem \ref{th:VRC_Amp} with Proposition \ref{prop:lines_VRC} identifies the Landau fiber $\psi^{-1}_{G,\sigma,u}({\bf M}_{\bf z})$ with the amplituhedron fiber $\Gr(k,{\bf z}^\perp) \cap \Pi_{G,\sigma,u}$. In particular, the degree of the Landau map $\psi_{G,\sigma,u}$ on the irreducible component $V_{G,\sigma,u}$ coincides with the degree of the amplituhedron map on the positroid variety $\Pi_{G,\sigma,u}$. 
From this we get Theorem~\ref{thm:positroid}.
\end{proof}

\begin{example}[Fibonacci]\label{ex:positroid_Fibonacci}
We revisit Corollary \ref{cor:fibonacci} with $\ell=6$, diagonals $26, 35, 36$ and $u=(3,3,2,3,2,2)$. For each 
of the $2^4=16$ bicolorings $\sigma$, we have a positroid variety $\Pi_{G,\sigma,u}$, see Figure \ref{fig:grassmann_lines} for $\sigma=({\bf bbww})$.
  Up to Hodge duality, there are $8$ different bicolorings. The sum of the intersection numbers of 
  the $\Pi_{G,\sigma,u}$ gives  $2 \cdot (16+24+64+48+24+40+32+8)=512=2^6 \cdot 8=2^\ell F_{\ell+1}$.
  This is the LS degree of $V_{G_u}$, as expected from Corollary \ref{cor:fibonacci}.

\end{example}

Theorem \ref{thm:positroid} extends naturally to vector labels $u \in \NN^e$
for values of $e$ other than $d$. For $e = d-1$ we can view
the curve $\psi_{G,\sigma,u}^{-1}({\bf M})$ inside the positroid variety associated
with $G,\sigma,u$, and we get similar positroid models for
NLS geometry, N${}^2$LS geometry and beyond.
We conclude this section with SLS singularities,
so we make this explicit for~$e=d+1$.

The positroid variety $\Pi_{G,\sigma,u}$ is defined exactly as above, but now for $|u| = e = d+1$.
We here consider its Chow-Lam form, as defined in \cite{PS}.
See Example 1.2 and Section 4 in \cite{PS}.

\begin{theorem} \label{thm:positroid2}
The Chow-Lam form of the positroid variety $\,\Pi_{G,\sigma,u}$ equals
the SLS resultant $R_{G,\sigma,u}$ of $G_u$ with bicoloring $\sigma$.
External line configurations ${\bf M} \in {\rm Gr}(2,4)^{d+1}$ that lie in the image of 
its Landau map $\psi_{G,\sigma,u}$ correspond to subspaces
$\mathfrak{L} \in \Pi_{G,\sigma,u}$ such that~$\mathfrak{L} \subseteq {\rm kernel}({\bf M})$.
\end{theorem}

\begin{example}[$\ell=1$]
The positroid for $u=(5)$ is denoted $\beta = (2,2,2,2,2)$ 
in \cite[Example 4.7]{PS}. It is the uniform rank $2$ matroid $U_{2,5}$
with each element duplicated. Its Chow-Lam form is displayed in
\cite[Equation (18)]{PS}. This matches (\ref{eq:resultant_5_lines}) in Example~\ref{ex:commontrans}.
\end{example}

\section{Promotion Maps}\label{sec:promotion_maps}

In our recursive approach to LS discriminants in Section \ref{sec:recursive},
   internal lines are replaced by new lines determined by smaller subgraphs via \emph{substitution maps}. By the results of 
   Section~\ref{sec:positroid}, these lines are elements in the fiber of the Landau map of the corresponding subgraph, which can be constructed via the VRCs of the positroid associated to that subgraph.
When interpreted as maps on Grassmannians, these substitution maps coincide with the \emph{promotion maps} introduced in \cite{VRC}. In that work, promotion maps were conjectured to be positive. If true, this would imply that our substitution maps preserve copositivity, i.e.\ they send positive configurations of lines to positive configurations.
We now illustrate this in the simplest case.

\begin{example}[$4$-mass box promotion]\label{ex:4mb_promotion}
 Given four generic lines ${\bf M}^{(1)}=(M_1,M_2,M_3,M_4)$, the fiber $\psi_{K_1,(4)}^{-1}({\bf M}^{(1)})$ consists of two lines
  $L^+({\bf M}^{(1)})$ and $L^-({\bf M}^{(1)})$. These are algebraic functions of ${\bf M}^{(1)}$,
  defined locally on $\Gr(2,4)^4 \setminus \{\Delta_{K_1, (4)}=0\}$.
Given any positive line
configuration ${\bf M}=({\bf M}^{(1)},M_5,M_6,M_7)$, 
the copositivity of the substitution map $\varphi_{\pm}$ in \eqref{eq:four_box_subs_maps} means
that the configurations $(L^\pm({\bf M}^{(1)}),M_5,M_6,M_7)$ are positive.
We now reformulate this condition on Grassmannians. Write $M_i=z_{2i-1} z_{2i}$ and $L^\pm= v^\pm w^\pm$, where the points are defined by $v^\pm=L^\pm\cap M_1$ and $w^\pm=L^\pm\cap M_4$. The lines $L^{\pm}(\mathbf{M}^{(1)})$ are the same as those in~\eqref{eq:Xpm} up to rescaling and relabeling of the lines $\mathbf{M}^{(1)}$. 
Then, for $[z_1|\cdots|z_{14}] \in \Gr_{>0}(4,14)$,
 we have $[v^\pm|w^\pm|z_9|\cdots|z_{14}] \in \Gr_{>0}(4,8)$.
 This is Lemma \ref{lem:4mb_pos}.
The vectors $v^\pm,w^\pm$ we use to represent the line $L^\pm$ 
appear  in the VRC of the positroid $\Pi_{G,\sigma,u}$, see Figure~\ref{fig:promotion_4mb_K_3}. 
\end{example}

The map $\,\tilde{\varphi}:\Gr(4,14)\to\Gr(4,8),\, [z_1|\cdots|z_{14}] \mapsto [v^\pm|w^\pm|z_9|\cdots|z_{14}]\,$
from Example \ref{ex:4mb_promotion} is a \emph{promotion map}.
We recall the general definition of promotion maps from \cite[Section~4.2]{VRC}.

Let $\Pi_\mathcal{G}$ be a positroid variety in $\Gr(k,n)$ of dimension $4k$ and
 intersection number $\gamma>0$. Fix a generic point ${\bf z}\in\Gr(4,n)$ and consider the
 corresponding set $\mathcal{C}^{\bf z}(\mathcal{G})$ 
of VRCs for~$\Pi_\mathcal{G}$.
 By Theorem~\ref{th:VRC_Amp}, this set has $\gamma$ elements,
 corresponding to the $\gamma$ vector-relation configurations with  boundary ${\bf z}$. 
 For a black vertex $b$ of $\mathcal{G}$, let
$w_b^{(1)}({\bf z}),\ldots,w_b^{(\gamma)}({\bf z})$ be
the vectors associated to $b$ in these $\gamma$ configurations.
Now add a disk $D$ (a \emph{blob}) with $n'$ boundary vertices inside one of the faces of $\mathcal{G}$, which we call the \emph{core}. Connect the $j$-th boundary vertex of $D$ to a black vertex $b_j$ of $\mathcal{G}$ so that the resulting graph remains planar. Fix $r\in[\gamma]$. The associated \emph{promotion map}
$\tilde{\varphi}_r:\Gr(4,n)\to\Gr(4,n')$ sends ${\bf z}$ to the matrix $[w_{b_1}^{(r)}({\bf z})|\cdots|w_{b_{n'}}^{(r)}({\bf z})]$
whose $j$-th column is the vector assigned to the black vertex $b_j$ in the $r$-th VRC.

\begin{example}[Triangle]
We consider the promotion map whose core is the positroid $\Pi_{K_3,{\bf w},(3,3,3)}$ 
in Example \ref{ex:positroid_triangle}. Its intersection number is $\gamma=8$.
See Figures~\ref{fig:grassmann_graph_K_3} and \ref{fig:promotion_4mb_K_3}
for the blob attachment.  For a fixed ${\bf z}\in\Gr(4,n)$, there are $8$ VRCs with boundary ${\bf z}$. Let $x^{(r)},p^{(r)},y^{(r)}$ be the vectors appearing in these VRCs as shown in 
Figure~\ref{fig:promotion_4mb_K_3},
with $r\in[8]$. Consider the lines $L_1^{(r)}=x^{(r)}p^{(r)}$ and $L_3^{(r)}=p^{(r)}y^{(r)}$, with $x^{(r)}=L_1^{(r)}\cap (z_1z_2)$ and $y^{(r)}=L_3^{(r)}\cap (z_{17}z_{18})$.
For a fixed $r\in[8]$, the promotion map $\tilde{\varphi}_r:\Gr(4,n)\rightarrow\Gr(4,n-15)$ sends ${\bf z}$ to the point $[x^{(r)}|p^{(r)}|y^{(r)}|z_{19}|\cdots|z_n]$, where $x^{(r)},p^{(r)},y^{(r)}$ are algebraic functions of $z_1,\ldots,z_{18}$.  This is the same matrix discussed at the end of Section \ref{sec:recursive}, whose positivity we verified numerically.
In particular, if $\tilde{\varphi}_r$ is copositive, meaning that it restricts to a map $\tilde{\varphi}_r:\Gr_{>0}(4,n)\rightarrow\Gr_{>0}(4,n-15)$, then the substitution map for $G=K_3$ is copositive. 
\end{example}

Generalizing the above example, the copositivity of substitution maps $\varphi$ on lines 
seen in Section~\ref{sec:recursive} follows from the copositivity of the corresponding promotion maps $\tilde{\varphi}$ on the Grassmannian
in this section. The latter is the subject of the main conjecture of \cite{VRC}.

\begin{conjecture}[\cite{VRC}]\label{conj:promotion_maps}
Let $\tilde{\varphi}:\Gr(4,n)\rightarrow\Gr(4,n')$ be a promotion map. Then $\tilde{\varphi}$ is copositive, i.e.\ it restricts to a map $\tilde{\varphi}:\Gr_{>0}(4,n)\rightarrow\Gr_{>0}(4,n')$.
\end{conjecture}

We now explain how promotion maps arise in the recursion of Section~\ref{sec:recursive}. 
Our setting is that of Proposition~\ref{prop:rec_fibers} where $\mathcal{L}$ is  a planar Landau diagram.
The following construction generalizes the examples where $\mathcal{L}_1$ is the $4$-mass-box and the triangle shown in Figure~\ref{fig:promotion_4mb_K_3}.

Suppose  $V_{\mathcal{L}_1}$ is irreducible (or consider one of its irreducible components) and let $\Pi_{\mathcal{G}_1}$ be the associated positroid variety with intersection number $\gamma$. 
Let ${\bf L}^{(1)}$ be one of the $\gamma$ points in the fiber $\psi_1^{-1}({\bf M}^{(1)})$.
The lines in ${\bf L}^{(1)}$ are spanned by pairs of vectors in a collection ${\bf w}_1$ coming from one of the $\gamma$ VRCs in $\mathcal{C}^{{\bf z}_1}(\mathcal{G}_1)$. The substitution map $\varphi$ sends ${\bf M}=({\bf M}^{(1)},{\bf M}^{(2)})$ to $({\bf L}^{(1)},{\bf M}^{(2)})$. 
On Grassmannians, this corresponds to a map $\tilde{\varphi}$ sending ${\bf z}=({\bf z}_1,{\bf z}_2)$ to $({\bf w}_1,{\bf z}_2)$. 
In this way we obtain $\gamma$ such maps, corresponding to the $\gamma$ VRCs of $\mathcal{C}^{{\bf z}_1}(\mathcal{G}_1)$.
The map $\tilde{\varphi}$ is a promotion map with core $\mathcal{G}_1$ and a blob $D$ attached to black vertices of $\mathcal{G}_1$. 
The attachment is chosen so that if a line $L_1^{(i)}$ is spanned by $w_1^{(p_i)}$ and $w_1^{(p_i+1)}$, then the boundary vertices $p_i$ and $p_i+1$ of $D$ are connected to black vertices adjacent to the vertex $v_i$ of $\mathcal{G}_1$, ensuring that $L(v_i)=L_1^{(i)}$. 
As $\mathcal{L}$ is planar, this attachment can be performed in a planar~way.

\begin{figure}[htbp]
    \centering
\includegraphics[width=1.0\textwidth]{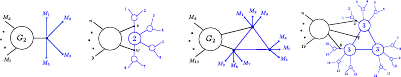}
    \caption{Illustrations of the substitution maps for the $4$-mass box (left) and the triangle (right).
    We display their promotion maps, and some relevant vectors of their VRCs.}
    \label{fig:promotion_4mb_K_3}
\end{figure}

We now turn to {\em rational promotion maps}. This refers to
 Section~\ref{sec:rationality}, where the points in the fibers were given by rational formulas. In particular, the recursive formulas of Section~\ref{sec:recursive} for the LS discriminants yield substitution maps that are rational maps. What can we say about their copositivity properties?
To address this question, we first extend the relations between positroids, VRCs, and line incidence varieties from Section~\ref{sec:positroid} to the rational case.

Let $G^\triangle$ be a rational graph and fix $u$ together with an irreducible component $V^\triangle_{G,\sigma,u}$. We associate to it a positroid $\Pi^\triangle_{G,\sigma,u}$ with plabic graph $\mathcal{G}=\mathcal{G}^\triangle_{G,\sigma,u}$ as follows.
Start from the plabic graph $\mathcal{G}'=\mathcal{G}_{G,\sigma,u}$ corresponding to the positroid variety $\Pi_{G,\sigma,u}$ defined in Section~\ref{sec:positroid}. If the lines $M_i,M_{i+1},L_s=L(v_s)$ are concurrent (that is, if $i,j,s$ form a black triangle in $\sigma$), then we remove the two tripods in $\mathcal{G}'$ attached to $v_s$ and connect $v_s$ directly to the boundary vertex labeled $p_i+1$. If the lines $M_i,M_{i+1},L_s=L(v_s)$ are coplanar (that is, if $i,j,s$ form a white triangle in $\sigma$),
then we merge the two tripods in $\mathcal{G}'$ attached to $v_s$ as shown in Figure~\ref{fig:vrc_lines_3mb}.

\begin{corollary}
The LS degree of the Landau map on the irreducible component $V^\triangle_{G,\sigma,u}$ and the intersection number of the positroid variety $\Pi^\triangle_{G,\sigma,u}$ are both equal to $1$. Moreover, the point in the fiber of the Landau map can be constructed 
rationally from the VRC of $\Pi^\triangle_{G,\sigma,u}$. 
\end{corollary}

Consider the vertex $v_i$ of $\mathcal{G}$ that corresponds to a vertex $i$ of $G$.
If $w_i,u_i$ are two vectors in the VRC assigned to distinct edges incident to $v_i$, then the line $L_i=w_i u_i$ is well defined. The resulting configuration $(L_1,\ldots,L_\ell)$ is in the fiber of the Landau map, as in Proposition~\ref{prop:lines_VRC}.

\begin{figure}[htbp]
    \centering
    \includegraphics[width=1.0\textwidth]{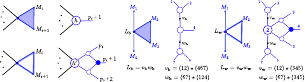}
    \caption{Left: The positroids when the external lines lie on a white or black triangle. Center and right: Line configurations, their positroids, and VRCs for the $3$-mass box.}
    \label{fig:vrc_lines_3mb}
\end{figure}

\begin{example}[$\ell=1$]\label{ex:3mb_vrc}
Consider the \emph{$3$-mass box}. This corresponds to $G^\triangle_{\sigma,u}$ where $G_u$ is the $4$-mass box and $\sigma \in \{{\bf b},{\bf w}\}$. In Figure~\ref{fig:vrc_lines_3mb} we draw the plabic graph $\mathcal{G}:=\mathcal{G}_{G,\sigma,u}^\triangle$, some vectors in the VRC of $\mathcal{C}^{\bf z}(\mathcal{G})$, and the point ${\bf L}_\sigma=(L_\sigma)$ in the fiber of the Landau map for external lines ${\bf M}_{\bf z}$. Here ${\bf z}=[z_1|\cdots|z_7]$ is a fixed $4\times 7$ matrix, the boundary vertices of $\mathcal{G}$ are labeled $1,\ldots,7$, and the external lines are $M_1=(z_1,z_2)$, $M_2=(z_3,z_4)$, $M_3=(z_4,z_5)$, $M_4=(z_6,z_7)$.
\end{example}

Also in the rational regime, substitution maps are promotion maps. However, these maps are more special, as they are rational. 
These rational promotion maps are believed to be copositive by Conjecture~\ref{conj:promotion_maps},
and they are further conjectured to define \emph{cluster quasi-homomorphisms}. In particular, they send cluster variables to (products of) cluster variables. 
    
\section{Cluster Variables}
\label{sec:cluster}

Cluster algebras are remarkable commutative rings, introduced by
Fomin and Zelevinsky~\cite{FZ02} in the study of total positivity and
representation theory. They now appear in many
areas of mathematics, including Teichmüller theory, mirror symmetry,
Poisson geometry, and mathematical physics. See~\cite{FWZ} for an
introduction.
A projective variety $X$ has a \emph{cluster structure} if its
 coordinate ring $\mathbb{C}[X]$ is a cluster
algebra. A \emph{cluster} ${\bf x}=\{x_1,\ldots,x_n\}$ is a collection of
coordinates defining a torus chart on $X$. One passes from
${\bf x}$ to another cluster
${\bf x'}={\bf x}\backslash\{x_i\}\cup\{x'_i\}$ by performing a
\emph{mutation}. This replaces a variable $x_i$ by a new variable $x'_i$
via an \emph{exchange relation}
$x_i x'_i = P_i + Q_i$,
where $P_i$ and $Q_i$ are monomials in 
${\bf x}\backslash\{x_i\}$.
The elements of a cluster are 
\emph{cluster variables}. Starting from an initial cluster and performing
all possible mutations produces the  set of all cluster variables.
This set can be finite or~infinite.

Grassmannians admit cluster structures \cite{Scott}. For example,
$\Gr(2,n)$ has a finite cluster structure. Its
cluster variables are the Plücker coordinates $\langle ij \rangle$ indexed by
arcs of an $n$-gon, and clusters correspond to collections of arcs
forming triangulations of the $n$-gon. Each cluster contains $2n-3$
independent Plücker coordinates, which equals $\dim \Gr(2,n)+1$.
There are Catalan-many clusters.
The cluster structure interacts beautifully with positivity. Each
cluster provides a \emph{positivity test}: 
$\Gr_{>0}(2,n)$ is the locus where the $2n-3$ cluster variables
of any fixed cluster are positive. This  forces all
$\binom{n}{2}$ Plücker coordinates to be positive.

For $k \geq 3$ and
$n$ large enough, the cluster algebra of $\Gr(k,n)$ has
infinite type, and its combinatorics is much richer.
Cluster variables are no longer
only Pl\"ucker coordinates, but they are homogeneous polynomials in
the Pl\"ucker coordinates of arbitrarily high degree. No general
classification of these cluster variables is known.
Nevertheless, clusters still provide a positivity test for
$\Gr_{>0}(k,n)$. Every cluster variable is Grassmann copositive.
For example, 
$\langle2451\rangle\langle3672\rangle-\langle2453\rangle\langle1672\rangle$
is a cluster variable for $\Gr(4,7)$. 
It is positive on $\Gr_{>0}(4,7)$.

Cluster variables  provide an infinite family of Grassmann copositive polynomials.
Positivity can be made manifest by expressing a cluster variable in
terms of a given cluster: it becomes a Laurent
polynomial with positive coefficients. This is the celebrated
\emph{Laurent phenomenon} \cite{FWZ}.
 In Section~\ref{sec:positivity} we also introduced a family of
Grassmann copositive functions, namely the LS
discriminants and SLS resultants, which are multi-graded Hurwitz and Chow
forms \cite{PSS} for line incident varieties \cite{HMPS1}.
It is natural to ask
whether they have any relation to cluster variables. This question is
also motivated by the physics of scattering amplitudes.

The first connection between particle physics and cluster algebras was
made by Golden--Goncharov--Spradlin--Vergu--Volovich~\cite{GGSVV, GGSVV2},
who described singularities of scattering amplitudes of planar
$\mathcal{N}=4$ SYM  using cluster variables for 
$\Gr(4,n)$. Remarkable
\emph{cluster adjacencies} were later discovered by
Drummond--Foster--G\"urdo\u{g}an~\cite{DFG}. For related versions
involving Landau singularities see
in~\cite{GP23}. These advances led to the
cluster bootstrap program~\cite{CHDDDFGvHMP}, which enabled
cutting-edge computations. Nevertheless, a proof or explanation
{\em why}
cluster structures should arise in particle physics has since been beyond~reach.

When amplitudes are rational functions and no integration is needed 
(at tree level), the recent advances  in~\cite{EZLPTSW} 
explain the cluster structure governing the
singularities of scattering amplitudes
in terms of the positive geometry of the amplituhedron.
A crucial tool was the \emph{BCFW recursion} and the introduction
of \emph{BCFW promotion maps} at the level of the
cluster algebra $\mathbb{C}[\Gr(4,n)]$ that mimic the
recursive structures of amplitudes and preserve the cluster
properties. These maps are examples of \emph{cluster
quasi-homomorphisms}~\cite{Fraser}. Given a map
$\tilde{\varphi}: \Gr(4,n) \rightarrow \Gr(4,n')$ with $n' \leq n$, we
can consider its pullback
$\tilde{\varphi}^*: \mathbb{C}[\Gr(4,n')] \rightarrow
\mathbb{C}[\Gr(4,n)]$. If $\tilde{\varphi}$ is a cluster
quasi-homomorphism, then it maps cluster variables in $\Gr(4,n')$ to
Laurent monomials in cluster variables in $\Gr(4,n)$, and it preserves
cluster structures.
Going further,
Even-Zohar {\em et al.}~\cite{VRC} 
  introduced a general
framework conjectured to generate infinite families of cluster
quasi-homomorphisms from plabic graphs.
Their \emph{cluster promotion maps} are exactly our
promotion maps in Section~\ref{sec:promotion_maps} in the special case
when they are rational.

In this work, we introduced Landau analysis on the Grassmannian to understand
the singularities of scattering amplitudes at any loop order. Our
(N)LS discriminants and SLS resultants are Landau
singularities that can arise in amplitudes and Feynman integrals, even
if one is not able to compute such integrals explicitly.
In Section~\ref{sec:positivity} we provided structural results for some
of these singularities, proving their copositivity and conjecturing a
general mechanism for copositivity. In what follows we reveal how
the cluster structures emerge.

If the external data are positive and the Landau diagram is planar, we
expect LS discriminants to factor into cluster variables in the
rational regime. Remarkably, the substitution maps
$\varphi_{r}$ in Theorem~\ref{thm:discr_fact} are expected to
be cluster promotion maps. This yields a vast class of graphs whose LS
discriminants factor into products of cluster variables. We start by
illustrating this phenomenon with two examples in the rational regime, one from Section \ref{sec:rationality}:

\begin{example}[Three-Mass Box]  \label{ex:3massbox}
   Let $\mathcal{L}=G_u \cup H^\triangle_u$ where $G = K_1$ and $u=4$.
    If the external lines attached to $G$ are     $12,23,45,67$ then
     the LS discriminant in (\ref{eqn:rational-box})~equals
    \begin{equation}\label{eq:discr_box}
        \Delta_{K_1,(4)} \,\,=\,\,
         (\langle 2451\rangle \langle 3672\rangle - \langle 2453\rangle \langle 1672\rangle)^2
         \,\,=\,\,\langle 245 | 13 | 267  \rangle^2.
    \end{equation}
       This is the square of a  cluster variable for ${\rm Gr}(4,7)$.
\end{example}

\begin{example}[$\ell=2$] \label{ex:boxin12}
Let $\mathcal{L}=G_u \cup H^\triangle_u$ with $G = K_2$ and $u = (4,3)$. Let $a_1 a_2$, 
$a_3 a_4$, $a_4 a_5$, $a_6 a_7$ be the external lines attached to vertex $1$ and $b_1 b_2$, 
$b_2 b_3$, $b_4 b_5$ those attached to vertex $2$. By Theorem~\ref{thm:discr_fact}, the LS discriminant $\Delta_{K_2,(4,3)}$ can be written in terms of LS discriminants of boxes $\Delta_{K_1,(4)}$. In particular, it is divisible by $\varphi^*_\textbf{b}\Delta_{K_1,(4)} \cdot  \varphi^*_\textbf{w} \Delta_{K_1,(4)}$.
That product equals 
\begin{equation}\label{eq:discr_pentabox_rat}
        \langle b_2 b_4 b_5 | b_1 b_3 | b_2 , (a_4 a_1 a_2) \star (a_4 a_6 a_7) \rangle^2   \cdot   \langle b_2 b_4 b_5 | b_1 b_3 | b_2 , (a_3 a_4 a_5) \star (a_1 a_2) , (a_3 a_4 a_5) \star (a_6 a_7) \rangle^2,
    \end{equation}
    where each factor is $\varphi_\sigma^* \langle b_2 b_4 b_5 | b_1 b_3 | b_2 b_6 b_7 \rangle^2$, and $\varphi_\sigma^*$ substitutes the line $(b_6b_7)$ with $L^{(\sigma)}$ as in \eqref{eq:3_mass_sol} for $\sigma \in \{\mathbf{b},\mathbf{w}\}$. 
    Since the expression in~\eqref{eq:discr_pentabox_rat} has degree $4$ in the Plücker coordinates of each external line, we conclude there is no extraneous factor, and~\eqref{eq:discr_pentabox_rat} equals $\Delta_{K_2,(4,3)}$.
\end{example}

The considerations in Section~\ref{sec:recursive}, together with Example \ref{ex:boxin12}, imply
   that the substitution maps $\varphi_{\mathbf{b}}$ and $\varphi_{\mathbf{w}}$, defined by substituting the lines $L^{(\mathbf{b})}$ and $L^{(\mathbf{w})}$ as in~\eqref{eq:3_mass_sol}, never produce extraneous factors in~\eqref{eq:discr_fact}. 
 The three-mass box is the Landau diagram in Example  \ref{ex:3massbox}.
 
\begin{proposition}\label{prop:fact_3mb}
Let $\mathcal{L}$ be a Landau diagram reducible by the three-mass box $\mathcal{L}_1$. Then
\begin{equation} \label{eq:3mb_rec}
\Delta_\mathcal{L} \,\,=\,\, \varphi_\textbf{b}^*\Delta_{\mathcal{L}_2} \cdot \varphi_\textbf{w}^*\Delta_{\mathcal{L}_2}.
\end{equation}
\end{proposition}

\begin{conjecture}
Fix a rational Landau diagram $\mathcal{L}_1$ 
 and recall the setting of Theorem~\ref{thm:discr_fact}. The points in the fiber 
 $\psi_1^{-1}(\mathbf{M}^{(1)})$ can be expressed rationally  in $\mathbf{M}^{(1)}$, namely
as polynomials in the Grassmann--Cayley algebra, such that the extraneous factor in~\eqref{eq:discr_fact} is equal to one.
\end{conjecture}

\begin{figure}[htbp]
\centering
\includegraphics[width=1.0\textwidth]{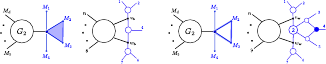}
\caption{Recursion on lines for $\sigma={\bf b}$ and $\sigma = {\bf w}$, with the corresponding promotion map.}
\label{fig:recursion_3mb}
\end{figure}

Returning to cluster algebras, we now define
rational promotion maps between Grassmannians,
as shown in Figure~\ref{fig:recursion_3mb}.
 Let $\sigma \in \{\textbf{b},\textbf{w}\}$ and consider the maps
$\tilde{\varphi}_{\sigma}: \Gr(4,n) \rightarrow \Gr(4,n-5)$
that send $[z_1|z_2|\cdots|z_n]$ to $[v_\sigma|w_\sigma|z_8|\cdots|z_n]$.
For $\sigma=\textbf{b}$ we set
$v_\textbf{b}=(12) \star (467)$ and $w_\textbf{b}=(67) \star (124)$,
while for $\sigma=\textbf{w}$ we define
$v_\textbf{w}=(12) \star (345)$ and $w_\textbf{w}=(67) \star (345)$.
These vectors arise from the VRCs in Example~\ref{ex:3mb_vrc} and Figure \ref{fig:vrc_lines_3mb}.
 They determine the lines
$L^{({\bf w})}=v_{{\bf w}}w_{{\bf w}}$
and
$L^{({\bf b})}=\langle1267\rangle v_{{\bf b}}w_{{\bf b}}$ in~\eqref{eq:3_mass_sol}.
The $\tilde{\varphi}_{\sigma}$ belong  to the general families of rational promotion maps
   in~\cite{VRC}, which are conjectured to be cluster quasi-homomorphisms. If this holds,
   then it provides an inductive mechanism ensuring that the LS discriminant $\Delta_\mathcal{L}$ in~\eqref{eq:3mb_rec} factors into a product of cluster variables. We prove this statement in Theorem \ref{th:cluster_fact_trees}.

\begin{lemma}\label{lem:3mb_prom_cluster}
The maps $\tilde{\varphi}_\textbf{b}$ and $\tilde{\varphi}_\textbf{w}$ are cluster quasi-homomorphisms.
\end{lemma}

\begin{proof}
Extending $\tilde{\varphi}_\sigma$, let
$\overline{\varphi}_\sigma: \Gr(4,n) \rightarrow \Gr(4,n-3)$ 
be the map which sends $[z_1|z_2|\cdots|z_n]$ to $[z_1|v_\sigma|w_\sigma|z_7|z_8|\cdots|z_n]$. We label the $n-3$ columns in the image by $1,2,6,7,8,\ldots,n$.
The initial cluster $\Sigma$ for $\Gr(4,n-3)$ has the following $4(n-7)+1$ Plücker coordinates~as cluster variables:
$\{\langle126i\rangle : i\in[8,n]\}$,
$\{\langle12i,i+1\rangle : i\in[7,n-1]\}$,
$\{\langle1i,i+1,i+2\rangle : i\in[6,n-2]\}$,
together with the variables
$\langle1267\rangle$, $\langle2678\rangle$,
and $\{\langle i,i+1,i+2,i+3\rangle : i\in[6,n-3]\}$.

Each of these is mapped by
$\overline{\varphi}_\textbf{b}^*$ to a product of Plücker coordinates; for instance, we have
$\overline{\varphi}_\textbf{b}^*(\langle1267\rangle)
=
\langle1267\rangle\langle1467\rangle\langle1247\rangle$,
$\overline{\varphi}_\textbf{b}^*(\langle126i\rangle)
=
\langle124i\rangle\langle1467\rangle\langle1267\rangle$,
and
$\overline{\varphi}_\textbf{b}^*(\langle2678\rangle)
=
\langle4678\rangle\langle1247\rangle\langle1267\rangle$.
Likewise, $\overline{\varphi}_\textbf{w}^*$ maps each element of $\Sigma$ to a product of cluster variables. For example
$\overline{\varphi}_\textbf{w}^*(\langle1267\rangle)
=
\langle1267\rangle\langle1345\rangle\langle3457\rangle$
and
$\overline{\varphi}_\textbf{w}^*(\langle126i\rangle)
=
\langle345|76|i12\rangle\langle1345\rangle$,
$\overline{\varphi}_\textbf{w}^*(\langle2678\rangle)
=
\langle345|21|678\rangle\langle3457\rangle$.
Here some of the factors are quadratic cluster variables.

Using the techniques of~\cite[Section~8]{VRC}, one verifies that $\overline{\varphi}_\textbf{b}$ and $\overline{\varphi}_\textbf{w}$  are cluster quasi-homomorphisms. In particular, $\overline{\varphi}_\textbf{b}$ is closely related to the \emph{BCFW promotion map} in~\cite[Section~4]{m4tiling}. 
Since $\overline{\varphi}_\sigma^*$ maps cluster variables into products of these, so do the maps $\tilde{\varphi}_\sigma^*$.
\end{proof}

\begin{theorem}\label{th:cluster_fact_trees}
Consider a planar Landau diagram $\mathcal{L} = G_u \cup H^\triangle_u$ 
where $G$ is a tree, as in Theorem \ref {th:real_trees}.
 Then the LS discriminant of $\mathcal{L}$ factorizes into a product of cluster variables.
\end{theorem}

\begin{proof}
We use induction on $[\ell]$.   The base case is
   Example~\ref{ex:fact_cluster_4mb_disc}:
       $\Delta_{K_1, (4)}(M_1,M_2,M_3,M_4)$ is a cluster variable when $M_1$ and $M_2$ are
       incident. For the inductive step, Proposition~\ref{prop:fact_3mb} expresses the LS discriminant of $\mathcal{L}$ in terms of LS discriminants of smaller diagrams via the substitution maps $\varphi_\sigma$. By Lemma~\ref{lem:3mb_prom_cluster}, the corresponding Grassmannian maps $\tilde{\varphi}_\sigma$ are cluster quasi-homomorphisms. Since $\tilde{\varphi}_\sigma$ and $\varphi_\sigma$ differ only by the multiplicative factor $\langle1267\rangle$, which is itself a cluster variable, it follows that if $\tilde{\varphi}_\sigma$ maps cluster variables to products of cluster variables, then so does $\varphi_\sigma$. This completes the induction, and it proves the theorem.
\end{proof}

The three-mass box recursion in Figure~\ref{fig:recursion_3mb} can be applied whenever the Landau diagram contains a vertex with $u_i=4$ and $H$ contains an edge between two external lines attached to $i$.
We expect Theorem \ref{th:cluster_fact_trees} to generalize to a vast class of graphs  $G$.
The following conjecture is important because it offers
an ab-initio explanation for
 the appearance of cluster structures in the Landau analysis
of singularities arising in planar $\mathcal{N}=4$ SYM scattering amplitudes.

\begin{conjecture}[Cluster Factorization Conjecture]\label{conj:CFC}
Let $\mathcal{L} = G_u \cup H^\triangle_u$ where $G$ 
is an outerplanar graph. Then the LS discriminant of $\mathcal{L}$ factors into a product of cluster variables.
\end{conjecture}

We now turn to Landau diagrams which are not reducible. Let $\mathcal{L}=G_u \cup H^\triangle_u$ with $G=K_3$ and $u=(3,3,3)$.
We write all factors in~\eqref{eq:discr_tr_fact} explicitly, and we argue that each of them is a cluster variable.
Let $a_1a_2$, $a_2a_3$, $a_4a_5$ denote the external lines at vertex $1$, 
and similarly with $b_i$ and $c_i$ for vertices $2$ and $3$. The $15$ vectors giving the lines form a totally positive $4 \times 15$ matrix ${\bf z}$ whose columns are ordered as $a_1 a_2 \ldots a_5 b_1 b_2 \ldots b_5 c_1 c_2 \ldots c_5$. Then, the concurrent discriminant $
\Delta_{\mathcal{L}, \mathbf{b}}$ is a polynomial in the Pl\"uckers of ${\bf z}$ given by the $12$ factors:
\begin{equation}\label{eq:discr_tr_fact_explicit}
    \begin{aligned}
       \Delta_{\mathcal{L}, \mathbf{b}} \,= \, \, \,& \langle a_2 a_4 a_5 | a_1 a_3 | a_2 (b_2 b_4 b_5) \star (c_2 c_4 c_5) \rangle^2 \, 
        \cdot   \langle a_2 a_4 a_5 | a_1 a_3 | a_2 (b_2 b_4 b_5) \star (c_1 c_2 c_3) \rangle^2 \, \\ \, \,\cdot  \,& \langle a_2 a_4 a_5 | a_1 a_3 | a_2 (b_1 b_2 b_3) \star (c_1 c_2 c_3) \rangle^2 \cdot (\text{permutations of $a,b,c$}) \, .
    \end{aligned}
    \end{equation} 
 The mixed LS discriminant $\Delta_{\mathcal{L}, \mathrm{mix}}$ is the product of the following eight factors:
 $$ \begin{small}
    \begin{matrix}
       &   \langle a_2 b_2 c_2 ,  (a_2a_4a_5) \star (b_2 b_4 b_5) \star (c_2 c_4 c_5) \rangle \, 
         \cdot \langle (a_1 a_2 a_3) \star (a_4a_5) , b_2 c_2 , (a_1a_2a_3) \star (b_2 b_4 b_5) \star (c_2 c_4 c_5) \rangle \,       
         \\ & \cdot \langle (a_1 a_2 a_3) \star (a_4a_5), (b_1 b_2 b_3) \star (b_4b_5) ,(c_1 c_2 c_3) \star (c_4c_5) , (a_1a_2a_3) \star (b_1 b_2 b_3) \star (c_1 c_2 c_3) \rangle    \\ 
         & \cdot  \langle (a_1 a_2 a_3) \star (a_4a_5) , (b_1 b_2 b_3) \star (b_4b_5)  ,c_2 , (a_1a_2a_3) \star (b_1 b_2 b_3) \star (c_2 c_4 c_5) \rangle \, 
         \cdot (\text{permutations of $a,b,c$}) \, .
    \end{matrix} \end{small}
$$    
    Both of the products above have degree $16$ in the Plücker coordinates of each external line. We
    verified that each factor of $\Delta_{\mathcal{L},\mathbf{b}}$ and $\Delta_{\mathcal{L},\mathbf{w}}$ is a
    cubic cluster variable for $\Gr(4,15)$. By contrast, the factors in
    $\Delta_{\mathcal{L},\mathrm{mix}}$
     are not cluster variables, since they are not Grassmann copositive.
     They change sign even when the external lines are positive; see  Remark  \ref{rmk:allowfor}.

We conclude the section with
the rational promotion maps associated with \emph{chain-trees} \cite[Definition~5.7]{VRC}.
Let $G$  be a chain of $s$ triangles touching pairwise at a common vertex, with
$u=(3,3,2,3,2,\ldots,2,3,2,3,3)$, together with $H^\triangle_u$ as in Figure~\ref{fig:chain_triangles}. 
 Take all triangles to be black: $\sigma=({\bf b},\ldots,{\bf b})$. When $s=1$ this is the triangle above. It gives rise to the~\emph{spurion promotion} in \cite[Section 8.2]{VRC}.
We would like to use this configuration as a seed for the recursion.
The substitution map $\varphi_\sigma$ sends ${\bf M}=({\bf M}^{(1)},{\bf M}^{(2)})$ to
$(L_1,L_3,\ldots,L_{2s+1},{\bf M}^{(2)})$. 
The map $\tilde{\varphi}_\sigma$ on the Grassmannian sends
$[z_1|\cdots|z_n]$ to
$[z_2|p_1|p_2|\cdots|p_s|z_{r}|z_{r+1}|\cdots|z_n]$,
where $r=8s+6$ and $p_i$ is the intersection point of the three concurrent lines of the $i$-th triangle.
For instance, $p_1=L_1 \cap L_2 \cap L_3$, while
$z_2=M_1 \cap M_2$ and $z_r=M_{4s+4} \cap M_{4s+5}$.
The lines appearing in the substitution are 
$L_1=z_2p_1$, $L_{2a+1}=p_a p_{a+1}$ for $a\in \{1,\ldots,s-1 \}$,
and $L_{2s+1}=p_s z_r$.
The map $\tilde{\varphi}_\sigma$ is precisely the promotion map of Figure~\ref{fig:chain_triangles}, whose core is a \emph{chain-tree} of type
$(3,3,1,3,1,3,\ldots,1,3,1,3,3)$ with $k=s+1$, as defined in~\cite[Definition~5.7]{VRC}.

The map  $\tilde{\varphi}_\sigma$ should be a cluster quasi-homomorphism, 
so $\varphi_\sigma$ sends cluster~variables to products of cluster variables.
This was proved for $s=1$ in \cite[Section~8.2]{VRC} and for
$s=2$ in \cite[Section~8.3]{VRC}.
For $s \geq 3$, it is covered by
 \cite[Conjecture~4.16]{VRC}. These examples suggest that kissing-triangle configurations provide natural
seeds for recursion in Landau analysis and give rise to infinite families
of promotion maps compatible with cluster structures.

\begin{figure}[htbp]
    \centering
    \includegraphics[width=1.0\textwidth]{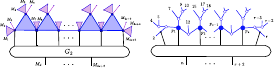}
    \caption{Left: Recursion on lines whose seed is a chain of triangles. Right: The corresponding cluster promotion map whose core is a chain-tree of type $(3,3,1,3,1,3,\ldots,1,3,3)$.}
    \label{fig:chain_triangles}
\end{figure}


\section{Future Directions}
\label{sec:future}

Throughout this paper we have seen that Landau analysis in the Grassmannian allows us to exploit
configurations of lines in 3-space 
\cite{HMPS1} to give concrete proofs of established conjectures in physics and to compute novel examples. 
Notably, we  applied our new geometric approach to prove the Reality Conjecture \ref{conj:reality_conjecture} for trees,
and we  showed that the LS discriminant factors into a product of cluster variables in the rational regime 
(Theorem \ref{th:cluster_fact_trees}). This new perspective on Landau analysis opens many doors for future work which we now describe.

\begin{enumerate}

    \item \textbf{Momentum twistors vs. Lee-Pomeransky.} Fevola, Mizera, and Telen \cite{FMT1, MT} also used methods from algebraic geometry to conduct Landau analysis. Their work uses the Lee-Pomeransky representation of a Feynman integral whereas we use momentum twistors. We presently do not understand the connection between these two approaches 
    at the level of the underlying algebraic geometry. This is an
        important open question.

    \item \textbf{More Landau singularities.} Our work focuses on \textit{leading} and \textit{super-leading} Landau singularities, 
    except for the discussion on the \textit{next-to-leading} case in Section~\ref{sec:nls}. 
    For full Landau analysis,
    one has to consider the branching loci of the maps $\psi_\mathcal{L}$ restricted to each stratum of a (\textit{Whitney}) stratification of the domain \cite{helmer2024landau}. It is therefore desirable to extend our tools and analysis to the case in which the fiber of $\psi_\mathcal{L}$ has dimension $\geq 2$. 
    What should be the analogue to rationality and reality for sub-leading singularities?

    \item \textbf{Beyond outerplanar graphs.} Most of our Landau diagrams came from
        \emph{outerplanar} graphs $G$. The vertices of $G$
                correspond to the loop variables in the Landau diagram. We expect that Conjectures \ref{conj:reality_conjecture} and \ref{conj:CFC} hold for any planar Landau diagram. A first interesting example arises from the \emph{wheel} $W_\ell$ with central vertex $\ell$ and $u \in \NN^\ell$ such that $u_\ell = 0$. Extending \cite[Theorem 4.9]{HMPS1}, we must
          determine the irreducible decomposition of $V_G$ and  compute the LS discriminant $\Delta_{G,\sigma,u}$ for each component $V_\sigma$. For planar graphs which are not outerplanar, we need an analogue to Theorem \ref{thm:outerplanar}. 
          Our problem is to determine the irreducible decomposition of $V_G$ for any planar graph~$G$. 
    
    \item \textbf{Loop amplituhedra and other semialgebraic sets.} 
    The $\ell$-loop amplituhedron~\cite{positive_geom} is a semialgebraic set in ${\rm Gr}(2,4)^\ell$. 
    It is
    believed to be a (weighted) positive geometry whose canonical form is the integrand of scattering amplitudes in planar $\mathcal{N}=4$ SYM. Our incidence varieties $V_G$ are part of the algebraic boundaries of loop amplituhedra. It is interesting to study semialgebraic sets inside ${\rm Gr}(2,4)^\ell$, or our incidence varieties, by considering a fixed sign on the polynomials $\langle L_i L_j \rangle$ for various pairs of indices $i,j$. These are generalizations of loop amplituhedra and elements in their boundary stratification.

\item \textbf{A computer algebra challenge.}
The full expansion of  an LS discriminant or an SLS resultant is very large.
Can it be computed for one irreducible component $V_{K_3,\sigma}$ of
the triangle graph?
Determine the number of Pl\"ucker monomials in
${\bf M}$ appearing in the expansion of $\Delta_{K_3,\sigma, (3,3,3)}$ and in the expansion of
$R_{K_3,\sigma,(3,3,4)}$. For those who prefer theoretical questions,
we reiterate the conjecture stated in \cite[Section 6]{PSS} and after Proposition \ref{prop:bbexpected}:
equality holds in (\ref{eq:discdegree}) if we take extraneous factors into consideration.

    \item \textbf{Discriminants across different components.} 
        The LS discriminant of the triangle has three  factors (Theorem \ref{thm:16168}).
     One of them, the \emph{mixed LS discriminant}, captures when points from different components come together. The factorization of the full LS discriminant $\Delta_{G, u}$ is non-trivial since $V_G$ lives in a product of projective spaces. 
    This requires an extension of \cite[Theorem 4.1]{Stu} to 
    the multigraded setting of \cite{PSS}.
    It is unclear whether the mixed LS discriminants are actually relevant for Landau analysis. 
A priori, they can
     contribute to the singularities of the integral~\eqref{eq:LMintegral}.
           However, we saw in Section \ref{sec:cluster} that $\Delta_{\mathcal{L},\mathrm{mix}}$ is not a product of cluster variables in the rational regime. This seems to be incompatible with the expectation on singularities of $\mathcal{N}=4$ SYM amplitudes. One possible explanation is that  the singularities cancel out in the full amplitude.
           If so, then we might see this by  taking the numerator in~\eqref{eq:LMintegral} from the canonical form of the amplituhedron~\cite{lippstreu2024landau}. 
           It would be interesting    to see if this holds.

    \item \textbf{Refined Landau analysis.} Landau analysis applied to an integral~\eqref{eq:LMintegral} yields the \textit{Landau variety}, which knows all singularities of the integral. However, the numerator $N(\mathbf{L};\mathbf{M})$ in~\eqref{eq:LMintegral} may cancel some potential singularities. The actual set of singularities may be smaller than the full \textit{Landau variety}. In planar $\mathcal{N}=4$ SYM theory, this story is particularly nice, since the full color-ordered amplitude can be written as~\eqref{eq:LMintegral}, where the integrand is the canonical form of the amplituhedron~\cite{amplituhdr}. The numerator of the 
    integrand, called the \textit{adjoint}, 
     was studied recently in~\cite{Ranestad:2024adjoints}. Landau analysis in $\mathcal{N}=4$ SYM can
     therefore be refined using the amplituhedron geometry and 
     its adjoint hypersurface~\cite{chicherin2026geometric, dennen2017landau}. It would be interesting to
      connect our work to the amplituhedron, and providing a refined Landau analysis in the language developed in this paper.
    
    \item \textbf{Reality for the triangle and further seeds.}
By  Conjecture \ref{conj:reality_conjecture},
       we expect  the leading singularities of a planar Landau diagram to be real if the external data are positive. Proving this for a graph $G$ allows the use of $G$ as a seed for the recursive method in Section~\ref{sec:recursive}. In this way, one can prove reality of the fibers, and hence by Theorem~\ref{thm:copositive} positivity of LS discriminants, for large families of Landau diagrams. We applied this to the case where $G$ is the one-vertex graph, therefore proving reality and positivity for Landau diagrams that are trees. It would be nice to prove that the promotion maps associated to (each irreducible component of) the triangle graph $G=K_3$ are copositive. 
       This leads us to explore the real algebraic geometry of the Cayley octad.
       This is a special case of Conjecture~\ref{conj:promotion_maps}, and we verified it 
       computationally. Further interesting seeds are the chain of kissing triangles 
       in Section~\ref{sec:cluster}, or the cycles~$G=C_\ell$.

 \item \textbf{Rational Landau diagrams.} We saw in Section~\ref{sec:cluster} that the factorization of LS discriminants into cluster variables is tightly related to rationality, which we presented in Section~\ref{sec:rationality}. Theorems~\ref{thm:trivalenttree}
 and~\ref{thm:rat_tr_of_l_gon}  identify  two classes of graphs $G$ for which the degeneration $G^\triangle$ yields a rational Landau diagram. On the other hand, such degenerations do not always yield rational graphs, even if $G$ is outerplanar, see Example~\ref{eg:cycles}. It is therefore interesting to determine the class of planar leading Landau diagrams $G_u$, embedded in a disk, for which there exists a degeneration $H$ being a subset of the cycle on the boundary of the disk, such that $G_u \cup H$ is rational. This would shed light on the class of graphs to which Conjecture~\ref{conj:CFC} is expected to apply. Also, it would help in understanding which seeds $G$ are required for building all planar rational Landau diagrams, which in turn provides a first step towards a proof of Conjecture~\ref{conj:CFC}.

    \item \textbf{Further degenerations.} In this paper we focused on degenerations $H$ of the external lines (kinematics) for planar Landau diagrams embedded in a disk, such that $H$ is a subgraph of the cycle at the boundary of the disk. For scattering processes with low number of particles, the relevant Landau diagrams involve degenerations for which several external lines actually coincide. As observed in~\cite{lippstreu2024landau}, we do not expect the associated discriminants to be copositive. In the same reference, however, the authors showed that such non-copositive singularity is cancelled in the amplitude for $\mathcal{N}=4$ SYM. It is therefore an interesting direction
     to extend the analysis of this paper to such degenerations, and understand their relevance to amplitudes in $\mathcal{N}=4$ SYM.

\item \textbf{Cluster compatibility and adjacency.}
A set of cluster variables is \emph{compatible} if they appear in the same cluster. First, one can investigate compatibility 
among the cluster variables appearing in a given LS discriminant. A preliminary analysis suggests that these 
factors are not all compatible.
It would be interesting to learn whether systematic patterns of compatibility or incompatibility emerge.
Second, one can study compatibility between cluster variables arising from different types of Landau singularities. For instance, one may examine pairs with one variable from an LS discriminant and another from an NLS discriminant, as well as pairs from successive orders such as N$^{r}$LS and N$^{r+1}$LS. Understanding these compatibility patterns could shed light on the phenomenon of \emph{cluster adjacency}, namely the observation that cluster variables appearing next to each other in the symbol of scattering amplitudes are compatible~\cite{DFG}.

\item \textbf{Cluster structures on varieties of lines.}
The passage from line incidence varieties in $\Gr(2,4)^d$ to the Grassmannian $\Gr(4,2d)$
was important for realizing substitution maps as promotion maps and, in the rational case, 
for uncovering  cluster structures in Section \ref{sec:cluster}.
However, this involves certain choices and often breaks the $\Gr(2,4)^d$ symmetry. It would therefore be desirable to develop a framework for VRCs and cluster algebras defined directly on $\Gr(2,4)^d$. Moreover, we expect our discriminants and resultants, even when they are not cluster variables, to naturally emerge from the cluster algebra structure. For example, the LS discriminant $\Delta_{K_1,(4)}$ of the $4$-mass box is not itself a cluster variable, but it belongs to the \emph{dual canonical basis} for $\Gr(4,8)$
and arises from limiting rays and stable fixed points of quasi-automorphisms~\cite{ALS,Drummond:2026tropical}.
\end{enumerate}

\bigskip \bigskip

\noindent {\bf Acknowledgement}:
We thank Jacob Bourjaily, Carlos D'Andrea, Dani Kaufman, Zack Greenberg, and Cristian Vergu for helpful communications.
 Our research was supported by
the European Research Council through the synergy grant UNIVERSE+, 101118787.
$\!\!$ \begin{scriptsize}Views~and~opinions expressed
are however those of the authors only and do not necessarily reflect those of the European Union or the 
European
Research Council Executive Agency. Neither the European Union nor the granting authority
can be held responsible for them.
\end{scriptsize}

\bigskip 
\bigskip

\noindent
\footnotesize {\bf Authors' addresses:}
\smallskip

\noindent Ben Hollering, MPI MiS Leipzig, Germany \hfill {\tt benjamin.hollering@mis.mpg.de} \\
\noindent Elia Mazzucchelli, MPI Physics, Garching, Germany \hfill {\tt eliam@mpp.mpg.de} \\
\noindent Matteo Parisi, OIST, Okinawa, Japan \hfill {\tt matteo.parisi@oist.jp} \\
\noindent Bernd Sturmfels, MPI MiS Leipzig, Germany \hfill {\tt bernd@mis.mpg.de} \\

 \end{document}